\documentclass[11pt,a4paper]{article}
\usepackage{amsthm,amsmath,amssymb,mathrsfs}
\usepackage[usenames,dvipsnames]{xcolor}
\usepackage{graphicx}
\usepackage{calrsfs}
\usepackage{hyperref}
\usepackage{enumerate}
\usepackage[nospace, nocompress]{cite}
\usepackage{geometry}

\newcommand{\weq}{\ = \ }

\newcommand{\wle}{\ \le \ }

\numberwithin{equation}{section}

\theoremstyle{plain}
\newtheorem{theorem}{Theorem}[section]
\newtheorem{proposition}[theorem]{Proposition}
\newtheorem{corollary}[theorem]{Corollary}
\newtheorem{lemma}[theorem]{Lemma}
\newtheorem{definition}[theorem]{Definition}

\theoremstyle{remark}
\newtheorem{remark}[theorem]{Remark}
\newtheorem{example}[theorem]{Example}
\newtheorem{examples}[theorem]{Examples}
\newtheorem{assumption}[theorem]{Assumption}

\DeclareMathOperator*{\esssup}{ess\,sup}
\DeclareMathOperator*{\essinf}{ess\,inf}
\DeclareMathOperator{\dist}{dist}
\DeclareMathOperator{\supp}{supp}

\DeclareMathOperator{\lip}{Lip}
\DeclareMathOperator{\loc}{loc}
\DeclareMathOperator{\id}{id}

\DeclareMathOperator{\tr}{Tr}

\DeclareMathOperator{\diverg}{div}

\begin{document}

\renewcommand{\thefootnote}{\fnsymbol{footnote}}

\author{Michael Hinz\footnotemark[2] \, and \, Jonas M. T\"olle\footnotemark[3] \, and \, Lauri Viitasaari\footnotemark[4]}

\title{Variability of paths and differential equations with $BV$-coefficients\footnotemark[1]}

\footnotetext[1]{All three authors would like to thank the anonymous referees for their kind interest and their helpful suggestions.}
\footnotetext[2]{Universit\"at Bielefeld, Fakult\"at f\"ur Mathematik, Postfach 100131, 33501 Bielefeld, Germany,\\ {\url{mhinz@math.uni-bielefeld.de}}.\\
Research supported in part by the DFG IRTG 2235 `Searching for the regular in the irregular: Analysis of singular and random systems' and by the DFG CRC 1283, `Taming uncertainty and profiting from randomness and low regularity in analysis, stochastics and their applications'.}
\footnotetext[3]{Aalto University, Department of Mathematics and Systems Analysis, PO Box 11100 (Otakaari 1, Espoo), 00076 Aalto, Finland, {\url{jonas.tolle@aalto.fi}}.\\
JMT would like to thank Panu Lahti and Mario Santilli for sharing their insight into the world of $BV$-functions. JMT was supported by the Academy of Finland and the European Research Council (ERC) under the European Union's Horizon 2020 research and innovation programme (grant agreements no. 741487 and no. 818437). Financial support by the Dean's Office, Faculty of Mathematics and Economics, Universit\"at Ulm is gratefully acknowledged.}
\footnotetext[4]{Uppsala University, Department of Mathematics, 751 06, Uppsala, Sweden,\\ {\url{lauri.viitasaari@math.uu.se}}.}
\maketitle


\begin{abstract}
We define compositions $\varphi(X)$ of H\"older paths $X$ in $\mathbb{R}^n$ and functions of bounded variation $\varphi$ under a relative condition involving the path and the gradient measure of $\varphi$. We show the existence and properties of generalized Lebesgue-Stieltjes integrals of compositions $\varphi(X)$ with respect to a given H\"older path $Y$. These results are then used, together with Doss' transform, to obtain existence and, in a certain sense, uniqueness results for differential equations in $\mathbb{R}^n$ driven by H\"older paths and involving coefficients of bounded variation. Examples include equations with discontinuous coefficients driven by paths of two-dimensional fractional Brownian motions.
\end{abstract}



{\small
\tableofcontents
}

\section{Introduction}

We prove new results on the existence and regularity of generalized Lebesgue-Stieltjes integrals 
\begin{equation}\label{E:eva}
\int_0^t \varphi(X_u)dY_u,\quad t\in [0,T],
\end{equation}
as in \cite{NualartRascanu, Zahle98, Zahle01}, where $X:[0,T]\to \mathbb{R}^n$ and $Y:[0,T]\to \mathbb{R}$ are H\"older continuous functions with sum of H\"older orders greater than one and $\varphi:\mathbb{R}^n\to \mathbb{R}$ is a function locally of bounded variation, \cite{AFP, Ziemer}, possibly discontinuous. We then employ these results to study equations in $\mathbb{R}^n$ of form
\begin{equation}\label{E:differentialeq}
X_t=\mathring{x}+\int_0^t \sigma(X_u)dY_u,\quad t\in [0,T],
\end{equation}
where $Y$ is a given path in $\mathbb{R}^n$, H\"older of order $\gamma>\frac12$, and $\sigma$ is a (bounded) matrix valued function of locally bounded variation. We implement a Doss transform, \cite{Doss, Yamato}, and use it to construct H\"older continuous solutions $X$ to (\ref{E:differentialeq}), unique in a certain class. This produces novel first results for discontinuous coefficients $\sigma$ in dimensions $n\geq 2$. The main difficulties are to provide a meaningful definition of the compositions $\varphi(X)$ (resp. $\sigma(X)$) and to show they are regular enough for the integrals in (\ref{E:eva}) or (\ref{E:differentialeq}), respectively, to make sense and for the Doss transformation method \cite{Doss, Yamato} to work. Our main tool is a quantitative condition which ensures that $X$ spends little time in regions where the gradient measure of $\varphi$ (resp. $\sigma$) is very concentrated. 

\subsection{Related literature}
To study equations of type (\ref{E:differentialeq}) for deterministic integrators $Y$ of low regularity or for probabilistic integrators $Y$ lacking semimartingale or other good distributional properties, the use of Stieltjes type integrals, \cite{NualartRascanu, Young, Zahle98, Zahle01}, and much more generally, the theory of rough paths, \cite{Gubinelli, LyonsQian, Lyons, LLC}, have become established tools. However, rather little is known about equations with irregular diffusion coefficients $\sigma$, and we are only aware of the few references mentioned below. In view of possible applications it seems particularly desirable to obtain results for discontinuous diffusion coefficients. They become necessary if one wants to model sharp interfaces between different media at which the solution $X$ abruptly changes its speed. If $n=1$ and $X$ solves (\ref{E:differentialeq}) with $\sigma=a+b\mathbf{1}_{(-c,c)}$, where $a>0$, $b>0$ and $c\in\mathbb{R}$, then the movement of $X$, dictated by the driver $Y$, is faster inside the sharply bounded strip $(-c,c)$ than outside. If $n=2$ and $Y=(Y^1,Y^2)$, then a $(2\times 2)$-matrix $\sigma$, with each entry being such a discontinuous function, could be used to determine polygonal regions inside of which the accelerating effect of $Y^1$ or $Y^2$ on the respective components $X^1$ and $X^2$ of $X=(X^1,X^2)$ is amplified or damped; this may be of interest for mixed market models \cite{BSV2008, BV2016, Ch2001}.

Stochastic differential equations with respect to Brownian motion involving non-Lipschitz (drift or diffusion) coefficients can be discussed in several different ways, \cite{RevuzYor}. Most classical and recent results for singular, irregular, or degenerate coefficients, and notions of uniqueness such as \cite{Davie, Figalli, FlandoliGubinelliPriola, KrylovRoeckner, LeeTrutnau, Veretennikov, Zhang10}, and results specific to the one-dimensional case, \cite{EngelbertSchmidt1, EngelbertSchmidt2, LeGall, Nakao}, are built upon the connection to diffusion theory and partial differential equations. For equations (\ref{E:differentialeq}) driven by rough deterministic or fractional Gaussian signals $Y$, such tools are not available.

For an integrator $Y$ that is H\"older of order $\gamma>\frac12$, Peano existence for solutions to (\ref{E:differentialeq}) in $\mathbb{R}^n$ is well known for coefficients $\sigma$ that are $s$-H\"older continuous provided that $\gamma >\frac{1}{1+s}$. Moreover, Picard existence and uniqueness holds if the coefficient is $\mathcal{C}^{1,s}$ (with the same $s$), \cite{Gubinelli, LLC}. See \cite{MaslowskiNualart, NualartRascanu, Zahle98, Zahle01} for the more classical Lipschitz resp. $\mathcal{C}^2$-cases. In \cite{LeonNualartTindel}, new results have been obtained for equations (\ref{E:differentialeq}) in $\mathbb{R}^n$ for the case $\gamma<\frac{1}{1+s}$.
For $n=1$ and continuous coefficients $\sigma$ whose reciprocal is integrable on compact intervals around zero, the authors of \cite{LeonNualartTindel} constructed solutions to (\ref{E:differentialeq}) by means of a Lamperti transform \cite{Lamperti}, see \cite[Theorem 3.7]{LeonNualartTindel}. For $n\geq 1$ they can solve (\ref{E:differentialeq}) if the components of the coefficient $\sigma$ are bounded from below by $|x|^s$, their gradients are H\"older continuous of order $\frac{1}{\gamma}-1$ away from zero, and the integral is understood in terms of Riemann sum approximations, \cite[Theorem 4.15]{LeonNualartTindel}. The first results on the existence of Stieltjes integrals with discontinuous coefficients (in the case $n=1$) were obtained in \cite{CLV16} (see also \cite[Chapter 5]{Chen}). There the authors proved the existence of (\ref{E:eva}) if $\varphi$ is of locally finite variation and $X$ is a sufficiently active path, \cite[Theorem 3.1 and Remark 3.3]{CLV16}. For random $X$ an integrability assumptions on its probability densities ensures this condition. They also prove a change of variable formula and several results on the approximation of (\ref{E:eva}) by Riemann-Stieltjes sums. A first study for differential equations (\ref{E:differentialeq}) was provided in \cite{Johanna}, where the authors prove existence and uniqueness of solutions to (\ref{E:differentialeq}) for $n=1$ if $Y$ is a fractional Brownian motion with Hurst index greater $\frac12$ and $\varphi$ is a (scaled) Heaviside step function. The authors of \cite{Johanna} used a Lamperti transform and smoothing arguments. Merging the assumptions from \cite{CLV16} and the transform used in \cite{Johanna}, the authors of \cite{LauriSole} were able to prove existence and uniqueness for (\ref{E:differentialeq}) in the case $n=1$ and in a probabilistic setup. Finally, we mention \cite{Yaskov}, where an alternative existence proof for integrals of type (\ref{E:eva}) was given using Riemann-Stieltjes approximations and suitable controls (avoiding fractional calculus), extending the results of \cite{CLV16}. 

\subsection{Brief description of our approach}
In the present paper we use a quantitative condition on the given individual path $X$ w.r.t. the given coefficient $\varphi$ (or $\sigma$), which we call \emph{$(s,p)$-variability}, Definition \ref{D:SVp-condition}. It may be seen as an deterministic version of the probabilistic Assumption 2.1 of \cite{LauriSole}, and as a higher dimensional analog of a condition in \cite[Corollary 3]{Yaskov}. 
Our first main result is Theorem \ref{thm:existence}, where we state that the composition $\varphi(X)$ of $\varphi$ with a H\"older path that is $(s,1)$-variable w.r.t. $\varphi$ is well defined and a member of a certain fractional Sobolev space, ensuring the existence of (\ref{E:eva}). The stronger assumption of $(s,p)$-variability with large $p$ guarantees that (\ref{E:eva}) is H\"older continuous. A key step to obtain these results is a multiplicative estimate for Gagliardo seminorms of $\varphi(X)$, Proposition \ref{P:basicest}, that can be viewed as a generalization of \cite[Proposition 4.6]{CLV16} (\cite[Proposition 4.1]{LauriSole}) to higher space dimensions. To obtain this estimate, one bounds differences of type $|\varphi(X_t)-\varphi(X_u)|$ in terms of a fractional maximal functions of the total variation $\left\|D\varphi\right\|$ of the gradient measure of $\varphi$, Proposition \ref{P:meanvalue}; this is a fractional version of a prominent argument, \cite[Lemma A.3]{CrippadeLellis}. Then one estimates further using the fact that the fractional maximal functions of order $1-s$ are trivially bounded by Riesz potentials of $\left\|D\varphi\right\|$ of order $(1-s)$, evaluated at $X_t$ (resp. $X_u$). The $(s,p)$-variability condition just means that these functions have the desired integrability in time. The $(s,1)$-variability of $X$ w.r.t. $\varphi$ is tantamount to saying that the total variation $\left\|D\varphi\right\|$ of the gradient measure of $\varphi$ and the occupation measure $\mu_X^{[0,T]}$ on $[0,T]$ have a finite mutual Riesz energy of  order $1-s$, see Remark \ref{R:SVviaoccu}. Phenomenologically this means that these two measures are sufficiently disperse with respect to each other to make the singular (repulsive) interaction kernel of order $-n+1-s$ integrable, which is a polarized version of well known arguments, see \cite{Falconer, Landkof, Mattila} for background and \cite{Doug} for a related application. Mutual Riesz energies are not necessarily easy to handle, but they encode dimensional properties in a neat way and this permits to easily connect to the well known scaling properties of gradient measures of $BV$-functions, \cite[Section 3.9]{AFP}, and well known scaling properties of fractal curves, \cite{Falconer}, such as realizations of prominent stochastic processes, \cite{Ayache, Falconer, Kahane, Xiao06}. If in a certain region of space $\varphi$ has a jump or strong oscillation so that $\left\|D\varphi\right\|$ is too concentrated, this can be compensated if $X$ is so fast moving in that part of space that the Hausdorff dimension of $\mu_X^{[0,T]}$ is sufficiently high to guarantee sufficient integrability, see Corollary \ref{C:localizeB}. The idea that increased activity of a path can compensate low coefficient regularity is also central in regularization by noise, \cite{CatellierGubinelli, Davie, Flandoli, FlandoliGubinelliPriola, GG20a, HarangPerkowski, KrylovRoeckner, BG2018}, see Remark \ref{R:occu_use}. It is closely related to the notion of irregularity studied in \cite{CatellierGubinelli} and \cite{GG20b}, see Subsection \ref{SS:Gubinelli}. However, while irregularity is a property of the path alone, variability is a property of a path relative to a given coefficient.
We begin our discussion of differential equations (\ref{E:differentialeq}) by showing that a uniform boundedness condition on the Riesz potential of order $1-s$ of the total variations of the gradient measures naturally takes us back to the case of $\gamma$-H\"older coefficients with $\gamma>\frac{1}{1-s}$, so that a well known Peano existence argument applies,  Theorem \ref{T:exnaive}. This is the extreme case, where the activity of the solution path is not used. To obtain existence results taking into account the activity of the solution path, we implement a Doss transform for $BV$-coefficients $\sigma$ under the main assumptions that $\sigma$ is invertible, Assumption \ref{A:main}, its inverse has curl-free columns, Assumption \ref{A:main2}, and an angle condition holds, (\ref{E:angular}). Of course in particular the curl free condition is quite restrictive, but as in classical implementations of the Doss method, \cite{Yamato}, it is inevitable. In lack of other existence results for equations with $BV$-coefficients it seems reasonable to establish a $BV$-variant of Doss' transformation under these assumptions. They  guarantee the existence of a Lipschitz function $f$ so that, roughly speaking, a solution $X$ is obtained as an image of the driver $Y$ under $f$. This uses the fact that dimensional lower bounds for $Y$ are stable under Lipschitz transformations and produces our second main result, Theorem \ref{T:exjoint}, which states the existence of 
H\"older continuous solutions $X$ to (\ref{E:differentialeq}) with $BV$-coefficients $\sigma$. A one-dimensional version of this result, partially under less restrictive assumptions, is formulated in Theorem \ref{T:exonedim}. In Theorem \ref{T:exshift} we assume that the occupation measure of $Y$ satisfies a kind of weighted upper regularity condition and that the gradient measures of the coefficient obey a specific moment condition `at the starting point'. Under these assumptions we can again observe the existence of H\"older solutions. These theorems are purely deterministic, the regularization effect of the irregular path is rather mild. Our third main result is Corollary \ref{C:exshiftrandom}. It is a probabilistic variant of Theorem \ref{T:exshift}, in which we assume that $Y$ is a stochastic process satisfying the weighted upper regularity condition in a mean value sense and obtain H\"older continuous solutions for almost every realization of $Y$. It may be applied to fractional Brownian motions $Y$ in $\mathbb{R}^n$ with $n\geq 2$ and Hurst index $H>\frac12$. One can regard Corollary \ref{C:exshiftrandom} as (a partial) extension of the probabilistic \cite[Theorem 2.1]{LauriSole}. Our fourth main result is a related uniqueness result, Theorem \ref{thm:DE-uniqueness}. It shows that  Assumptions \ref{A:main} and \ref{A:main2} guarantee uniqueness in the class of variability solutions.

It would be desirable to replace the Doss transform by standard fixed point arguments. The main open problem to be settled is to prove that --- under reasonable assumptions --- the integral process itself will be variable. Another goal for future research is to target equations that, in addition to a $BV_{\loc}$-diffusion coefficient, involve a drift vector field of low regularity. First results on variability and compositions involving discontinuous paths can be found in \cite{HTV2022}.

\subsection{Structure of the article}
The structure of the article is as follows: In Section \ref{S:integrals} we introduce the notion of $(s,p)$-variability, define compositions $\varphi(X)$ ($\sigma(X)$, respectively), and state our results on existence and properties of (\ref{E:eva}). We also provide a change of variable formula and a result on Riemann sum approximation. Section \ref{S:SDE} contains our results on existence and uniqueness of variability solutions to (\ref{E:differentialeq}). In Section \ref{S:suffvar} we provide a systematic discussion of $(s,1)$-variability, some of its immediate consequences, conditions sufficient to ensure it, and some probabilistic examples. We briefly compare variability to irregularity, verify the mentioned multiplicative estimate, the properties of (\ref{E:eva}) and the change of variable formula; we also point out links to currents. The Doss transformation and the claimed existence and uniqueness results for (\ref{E:differentialeq}) are proved in Section \ref{S:exandunique}. Basic facts on Riesz kernels, mollification, maximal functions, and fractional calculus are collected in Appendices.

By $\vert \cdot\vert$ we denote the Euclidean norm in $\mathbb{R}^n$. We write $B(x,r)$ for the open ball of radius $r>0$ centered at $x\in \mathbb{R}^n$. The symbol $\mathcal{L}^n$ stands for the $n$-dimensional Lebesgue measure and the symbol $\mathcal{H}^d$ for the $d$-dimensional Hausdorff measure on $\mathbb{R}^n$. For spaces of $\mathbb{R}^m$-valued functions we use notations like  $BV(\mathbb{R}^n)^m$ (to stay close to reference \cite{AFP}) or $L^1_{\loc}(\mathbb{R}^n, \mathbb{R}^m)$ (because it is more practical for other function spaces). For $m=1$, we suppress $\mathbb{R}^m$ from notation and write $L^1_{\loc}(\mathbb{R}^n)$. For a Borel measure $\nu$ on $\mathbb{R}^n$, we denote its (topological) support by $\supp\nu$.

\section{Stieltjes integrals and $BV$-coefficients}\label{S:integrals}

\subsection{Compositions of paths with $BV$-functions}

Recall that a function $\varphi\in L^1_{\loc}(\mathbb{R}^n)$ is of \emph{locally bounded variation}, denoted $\varphi\in BV_{\loc}(\mathbb{R}^n)$, if its distributional partial derivatives $D_i\varphi$ are signed Radon measures, $i=1,...,n$. We write $D\varphi=(D_1\varphi,...,D_n\varphi)$ for its $\mathbb{R}^n$-valued gradient measure, and $\left\|D\varphi\right\|$ for the total variation of $D\varphi$. If $\varphi\in L^1(\mathbb{R}^n)$ and $\left\|D\varphi\right\|(\mathbb{R}^n)<+\infty$, then $\varphi$ is said to be of \emph{bounded variation}, $\varphi\in BV(\mathbb{R}^n)$.

Let $T>0$. We consider continuous paths from $[0,T]$ into $\mathbb{R}^n$, that is, continuous functions $X=(X^1,\ldots,X^n):[0,T]\to \mathbb{R}^n$. The following definition is our key tool to provide a meaningful and sufficiently regular definition of the composition $\varphi\circ X$ of a $BV_{\loc}$-function $\varphi$ and a path $X$. As usual, $L^p(0,T)$ denotes the Lebesgue space of classes of $p$-integrable functions on $(0,T)$. 

\begin{definition}\label{D:SVp-condition}
Let $\varphi \in BV_{\loc}(\mathbb{R}^n)$, $p\in [1,+\infty]$ and $s\in (0,1)$. We say that a path $X=(X^1,\ldots,X^n):[0,T]\to \mathbb{R}^n$ \emph{is $(s,p)$-variable with respect to $\varphi$} if there is a relatively compact open neighborhood $\mathcal{U}$ of $X([0,T])$ such that  
\begin{equation}
\label{eq:SVp-condition}
\int_{\mathcal{U}}|X_{\cdot}-z|^{-n+1-s}\left\|D\varphi\right\|(dz) \in L^p(0,T).
\end{equation} 
We write $V(\varphi,s,p)$ for the class of paths $X$ that are $(s,p)$-variable w.r.t. $\varphi$ and use the short notation $V(\varphi,s):=V(\varphi,s,1)$.
\end{definition}

Note that $V(\varphi,s,p)\subset V(\varphi,s,q)$ for any $q<p$ and $V(\varphi,s,p)\subset V(\varphi,r,p)$ for any $r<s$.
The $(s,p)$-variability condition (\ref{eq:SVp-condition}) is a quantitative and relative condition on the path $X$ and the function $\varphi$. Roughly speaking, it ensures that $X$ varies sufficiently around sites where $\varphi$ has strong oscillations or jumps, encoded in the requirement that the Riesz potential of order $1-s$ of the restriction of $\left\|D\varphi\right\|$ to $\mathcal{U}$ is in $L^p(0,T)$, see Section \ref{S:suffvar} for a systematic discussion. The use of an open neighborhood $\mathcal{U}$ of $X([0,T])$ in (\ref{eq:SVp-condition}) simplifies several arguments (e.g. mollification).
We admit a component-wise point of view upon functions with values in $\mathbb{R}^m$.

\begin{definition}\label{D:SVcomponentwise}
Let $\varphi=(\varphi_1,...,\varphi_m) \in  BV_{\loc}(\mathbb{R}^n)^m$, $p\in [1, +\infty]$ and $s\in (0,1)$. We say that a path $X:[0,T]\to \mathbb{R}^n$ \emph{is $(s,p)$-variable with respect to $\varphi$} if it is $(s,p)$-variable with respect to each component $\varphi_i$ of $\varphi$. That is, $V(\varphi,s,p):=\bigcap_{i=1}^m V(\varphi_i,s,p)$. Similarly as before, we write $V(\varphi,s):=V(\varphi,s,1)$.
\end{definition}

Recall the following classical definition.

\begin{definition}\label{D:geometric-point-sets}
A function $\varphi\in L^1_{\loc}(\mathbb{R}^n,\mathbb{R}^m)$ is said to have an \emph{approximate limit} at $x\in  \mathbb{R}^n$ if there exists $\lambda_\varphi(x)\in\mathbb{R}^m$ such that 
\[\lim_{r\to 0}\frac{1}{\mathcal{L}^n(B(x,r))}\int_{B(x,r)}|\varphi(y)-\lambda_{\varphi}(x)| dy=0.\]
In this situation, the unique value $\lambda_\varphi(x)$ is called the \emph{approximate limit} of $\varphi$ at $x$.
The set of points $x\in  \mathbb{R}^n$ for which this property does not hold is called \emph{approximate discontinuity set} (or \emph{exceptional set}) and is denoted by $S_\varphi$.
\end{definition}

The set $S_\varphi$ does not depend on the choice of the representative for $\varphi$. If $\widetilde{\varphi}$ is a representative of $\varphi\in L^1_{\loc}(\mathbb{R}^n,\mathbb{R}^m)$ then a point $x\not\in S_\varphi$ with $\widetilde{\varphi}(x)=\lambda_\varphi(x)$ is called a \emph{Lebesgue point} of $\widetilde{\varphi}$, and the set of all Lebesgue points of $\widetilde{\varphi}$ is called the \emph{Lebesgue set} of $\varphi$. See for instance \cite[Definition 3.63]{AFP}. The set $S_\varphi$ is Borel and of zero Lebesgue measure, \cite[Proposition 3.64]{AFP}.
If $\varphi=(\varphi_1,...,\varphi_m) \in BV(\mathbb{R}^n)^m$ then by the Federer-Vol'pert theorem, \cite[Theorem 3.78]{AFP}, the set $S_\varphi$ is countably $\mathcal{H}^{n-1}$-rectifiable. 

We say that a Borel function $\widetilde{\varphi}:\mathbb{R}^n\to\mathbb{R}^m$ is a \emph{Lebesgue representative} of $\varphi=(\varphi_1,...,\varphi_m)\in L^1_{\loc}(\mathbb{R}^n,\mathbb{R}^m)$ if 
\begin{equation}\label{E:admissible}
\widetilde{\varphi}(x)=\lambda_{\varphi}(x),\quad x\in \mathbb{R}^n\setminus S_\varphi.
\end{equation}
Using Definition \ref{D:geometric-point-sets} and the equivalence of norms on $\mathbb{R}^n$ it is easy to see that
\begin{equation}\label{E:Sphiisunion}
S_\varphi=\bigcup_{i=1}^m S_{\varphi_i}.
\end{equation}
In particular, if for any $i$ the function $\widetilde{\varphi}_i:\mathbb{R}^n\to\mathbb{R}$ is a Lebesgue representative of $\varphi_i$ then $(\widetilde{\varphi}_1,...,\widetilde{\varphi}_m)$ is a Lebesgue representative of $\varphi$, and we refer to such representatives as \emph{component-wise Lebesgue representatives}.

The following observation will be proved in Section \ref{S:suffvar}. 

\begin{lemma}\label{lem:welldef}
Let $\varphi=(\varphi_1,...,\varphi_m)\in BV_{\loc}(\mathbb{R}^n)^m$ and $X\in V(\varphi,s)$ for some $s\in (0,1)$. Then for any component-wise Lebesgue representatives $\widetilde{\varphi}^{(1)}$ and $\widetilde{\varphi}^{(2)}$ of $\varphi$ we have 
\[\widetilde{\varphi}^{(1)}(X_t)=\widetilde{\varphi}^{(2)}(X_t)\] 
at $\mathcal{L}^1$-a.e. $t\in [0,T]$.
\end{lemma}

Lemma \ref{lem:welldef} could be rephrased by saying that under the $(s,1)$-variability condition the equivalence class $\varphi$ has a well defined trace on the range $X([0,T])\subset \mathbb{R}^n$ of $X$, endowed with a suitable measure. See \cite{AH96} or \cite{JW84} for pointwise redefinitions of functions and traces to closed subsets of $\mathbb{R}^n$ in other contexts.

\begin{definition}\label{D:comp}
Let $\varphi \in BV_{\loc}(\mathbb{R}^n)^m$ and suppose that $X\in V(\varphi,s)$ for some $s\in (0,1)$. We define the \emph{composition} 
$\varphi\circ X$
to be the $\mathcal{L}^1$-equivalence class of $t\mapsto \widetilde{\varphi}(X_t)$ on $[0,T]$, where $\widetilde{\varphi}$ is a component-wise Lebesgue representative of $\varphi$. Given $p\in [1,+\infty]$ we say that $\varphi$ is \emph{$p$-integrable w.r.t. $X$}, in symbols $\varphi\in L^p(X, \mathbb{R}^m)$, if $\varphi\circ X$  is an element of $L^p(0,T,\mathbb{R}^m)$. In the case $n=1$ we write $L^p(X)$ instead of $L^p(X,\mathbb{R})$.
\end{definition}

Thanks to Lemma \ref{lem:welldef} the composition and the notion of $p$-integrability w.r.t. $X$ are well defined. The component-wise choice of representatives is not essential, but it is convenient in conjunction with Definitions \ref{D:SVp-condition} and \ref{D:SVcomponentwise}.

We discuss $(s,1)$-variability in some examples.

\begin{example}\label{Ex:Lipschitz}
If $\varphi:\mathbb{R}^n\to \mathbb{R}$ is locally Lipschitz, then for any $s\in (0,1)$ any path $X$ is in $V(\varphi,s,\infty)$.
This follows from the fact that $\left\|D\varphi\right\|=|\nabla f|\cdot\mathcal{L}^n$ 
with $|\nabla f|\in L^\infty(\mathcal{U})$ on any relatively compact open set $\mathcal{U}\subset \mathbb{R}^n$.
\end{example}

\begin{example}\label{Ex:Cantor}
Let $\mathcal{C}\subset [0,1]$ be the classical middle third Cantor set and $\nu_\mathcal{C}$ the unique self-similar probability measure with support $\mathcal{C}$, see \cite{Falconer}. Let $\varphi_{\mathcal{C}}:\mathbb{R}^n\to [0,1]$ be the function 
that satisfies $\varphi_{\mathcal{C}}(x)=\nu_{\mathcal{C}}((0,x_1))$ for all $x=(x_1,x_2,...,x_n)\in [0,1]\times \mathbb{R}^{n-1}$, $\varphi_{\mathcal{C}}(x)=0$ for all $x\in (-\infty,0)\times \mathbb{R}^{n-1}$ and $\varphi_{\mathcal{C}}(x)=1$ for all $x\in (1,+\infty)\times \mathbb{R}^{n-1}$. Then $\varphi_{\mathcal{C}}\in BV_{\loc}(\mathbb{R}^n)$, and on $[0,1]^n$ we have $\left\|D\varphi_{\mathcal{C}}\right\|=D\varphi_{\mathcal{C}}=\nu_{\mathcal{C}}\otimes \mathcal{H}^{n-1}$. 
Writing $d_\mathcal{C}=\frac{\log 2}{\log 3}$ for the Hausdorff dimension of $\mathcal{C}$, we find that for $s\in (0,d_\mathcal{C})$ any path $X$ in $\mathbb{R}^n$ is in $V(\varphi_{\mathcal{C}},s,\infty)$. Now suppose $s\in (d_\mathcal{C},1)$. The constant path $X\equiv (\frac{1}{2},0,...,0)$ in $\mathbb{R}^n$ is in $V(\varphi_{\mathcal{C}},s,\infty)$, but the constant path $X\equiv (0,0,...,0)$ is not in $V(\varphi_{\mathcal{C}},s)$. For $n=1$ any smooth function $X:(0,T)\to (0,1)$ with a finite number of critical points is in $V(\varphi_{\mathcal{C}},s,\infty)$. For $n=2$ a smooth curve $X:(0,T)\to (0,1)^2$, parametrized to have unit speed, does not have to be in $V(\varphi_{\mathcal{C}},s)$. On the other hand, a path of Brownian motion is in $V(\varphi_{\mathcal{C}},s,\infty)$ with probability one. For $n\geq 3$ paths of fractional Brownian motions with Hurst index $H\in (0,\frac{1}{n-1+s})$ are in $V(\varphi_{\mathcal{C}},s,\infty)$ with probability one. See Example \ref{ex:detex} and Subsection \ref{SS:probab} for details.
\end{example}

The function $\varphi$ in Example \ref{Ex:Cantor} is H\"older continuous. The next example discussed variability with respect to discontinuous functions.

\begin{example}\label{Ex:jump}
Suppose that $\mathcal{O} \subset \mathbb{R}^n$ is a smooth domain with $\mathcal{H}^{n-1}(\partial \mathcal{O})<+\infty$. By \cite[Proposition 3.62]{AFP} the function $\mathbf{1}_\mathcal{O}$ is in $BV(\mathbb{R}^n)$ and $\mathcal{O}$ has finite perimeter $\mathcal{P}(\mathcal{O})=\left\|D\mathbf{1}_\mathcal{O}\right\|(\mathbb{R}^n)<+\infty$.
Let $s\in (0,1)$ be arbitrary.
If a smooth unit speed curve $X:[0,T]\to \mathbb{R}^n$ hits $\partial \mathcal{O}$ in finitely many points then we have $X\in V(\mathbf{1}_\mathcal{O},s)$, but if $n\geq 2$ and $X$ spends $\mathcal{L}^1$-positive time in $\partial \mathcal{O}$ then it cannot be an element of $V(\mathbf{1}_\mathcal{O},s)$, see Example \ref{ex:detexjump}. For $n=1$ or $n=2$ the path of a Brownian motion is in $V(\mathbf{1}_\mathcal{O},s,\infty)$ with probability one. For arbitrary $n\geq 1$, the path of a fractional Brownian motion with Hurst index $H\in (0,\frac{1}{n-1+s})$ is in $V(\mathbf{1}_\mathcal{O},s,\infty)$ with probability one. For arbitrary $n\geq 1$ it also follows that if $H\in (0,\frac{1}{s})$ and the fractional Brownian motion is started in $(\partial \mathcal{O})^c$ then it is in $V(\mathbf{1}_\mathcal{O},s)$ with probability one, see Subsection \ref{SS:probab}.
\end{example}

\subsection{Existence and properties of Stieltjes integrals}

As mentioned, we are interested in generalized Lebesgue-Stieltjes integrals defined in terms of fractional calculus, \cite{SKM}, introduced in \cite{Zahle98, Zahle01} and used e.g. in \cite{NualartRascanu}. 

We introduce suitable function spaces to discuss the existence and the continuity properties of the integral. Let $p\in [1,+\infty)$ and $0<\theta<1$. The Gagliardo seminorm of order $\theta$ with exponent $p$ of a measurable function $f: (0,T) \to \mathbb{R}^m$ is defined as
\begin{equation}
 \label{eq:Gagliardo}
 [f]_{\theta, p}
 \weq \left(\int_0^T \int_0^T \frac{|f(t) - f(u)|^p}{|t-u|^{1+\theta p}}  du  dt \right)^{\frac{1}{p}}.
\end{equation}
By $W^{\theta,p}(0,T, \mathbb{R}^m)$ we denote the space of measurable functions $f: (0,T) \to \mathbb{R}^m$ such that
\[\left\|f\right\|_{W^{\theta,p}(0,T,\mathbb{R}^m)}:=\left\|f\right\|_{L^p(0,T,\mathbb{R}^m)}+ [f]_{\theta, p}<+\infty.\]
Recall that for $m=1$ we agreed to suppress $\mathbb{R}^m$ from notation and simply write $W^{\theta,p}(0,T)$, which we do similarly for the spaces in the sequel. The H\"older seminorm of order $0<\theta<1$ of a measurable function $f: [0,T] \to \mathbb{R}^m$ is denoted by
\[
 [f]_{\theta,\infty}
 \weq \sup_{0 \le u < t \le T} \frac{|f(t) - f(u)|}{|t-u|^\theta},
\]
and we write $\mathcal{C}^{\theta}([0,T],\mathbb{R}^m)$ to denote the space of H\"older continuous functions $f:[0,T]\to\mathbb{R}^m$, endowed with the norm
\[\left\|f\right\|_{\mathcal{C}^\theta([0,T],\mathbb{R}^m)}:=\sup_{t\in [0,T]}|f(t)|+ [f]_{\theta,\infty}.\]

\begin{remark}
The spaces $W^{\theta,p}(0,T)$ and $\mathcal{C}^{\theta}([0,T])$ are classical Besov spaces of type $B_{p,p}^\theta$ and $B_{\infty,\infty}^\theta$, see for instance \cite{Triebel}.
\end{remark}

Because they appear naturally in connection with Stieltjes integrals we also consider the following more specific types of spaces, which (in this or a similar form) were introduced in \cite{NualartRascanu}. Accepting a slight abuse of notation we use the symbol $W^{\theta,\infty}(0,T,\mathbb{R}^m)$ to denote the space of all measurable $f:(0,T)\to \mathbb{R}^m$ such that 
\[\left\|f\right\|_{W^{\theta,\infty}(0,T,\mathbb{R}^m)}:=\left\|f\right\|_{L^\infty(0,T,\mathbb{R}^m)}+\esssup_{t\in [0,T]}\int_0^t\frac{|f(t)-f(u)|}{|t-u|^{1+\theta}}\:du<+\infty.\]
By $W^{\theta,\infty}_T(0,T,\mathbb{R}^m)$ we denote the space of measurable functions $f: (0,T) \to \mathbb{R}^m$ such that $f(T-) \in \mathbb{R}^m$ exists and
\[
 \|f\|_{W^{\theta,\infty}_T(0,T,\mathbb{R}^m)}
 \weq \sup_{t \in (0,T)} \frac{|f(T-) - f(t)|}{(T-t)^\theta} + \sup_{t \in (0,T)} \int_t^T\frac{|f(t)-f(u)|}{|t-u|^{1+\theta}}  du < \infty.
 \]
 It is well known and easily seen that 
\begin{equation}\label{E:HolderinSobo}
\mathcal{C}^{\theta+\varepsilon}([0,T], \mathbb{R}^m)\subset W^{\theta,\infty}(0,T, \mathbb{R}^m)\quad \text{and}\quad \mathcal{C}^{\theta+\varepsilon}([0,T], \mathbb{R}^m)\subset W^{\theta,\infty}_T(0,T,\mathbb{R}^m).
\end{equation}
We write $W^{\theta,p}_0(0,T,\mathbb{R}^m)$ for the space of measurable functions $f: (0,T) \to \mathbb{R}^m$ such that
\[\Vert f\Vert_{W^{\theta,p}_0(0,T,\mathbb{R}^m)}:= \int_0^T \frac{|f(t)|^p}{t^{\theta p}} dt + [f]_{\theta, p}<+\infty.\]
We emphasize that in the present paper the symbols $W^{\theta,\infty}(0,T)$ and $W^{\theta,p}_0(0,T)$ do not have the standard meaning.

The following definition is due to \cite{Zahle98, Zahle01}, see also \cite{NualartRascanu}. By $D^{\theta}_{0+}$ and $D^{1-\theta}_{T-}$ we denote the (left and right sided) fractional Weyl-Marchaud derivatives of orders $\theta$ and $1-\theta$, respectively, see formulas (\ref{E:WMforward}) and (\ref{E:WMbackward}) in Appendix \ref{S:fraccalc}. Background information on fractional derivatives can be found in \cite{SKM}.

\begin{definition}
\label{def:integral}
Let  $f \in W^{\theta,1}_0(0,T)$ and $g \in W^{1-\theta,\infty}_T(0,T)$ for some $\theta \in (0,1)$. Then we define the integral by 
\begin{equation}
 \label{eq:ZSIntegralSimple-definition}
 \int_0^T f_t  dg_t
 \weq (-1)^\theta \int_0^T D^{\theta}_{0+} f(t)\: D^{1-\theta}_{T-} (g-g(T-))(t)  dt.
\end{equation}
\end{definition}

The right hand side of (\ref{eq:ZSIntegralSimple-definition}) is a real number; the complex prefactor (used in \cite{Zahle98} to ensure natural formulas) compensates with another complex prefactor in the right sided Weyl-Marchaud derivative, cf. (\ref{E:WMbackward}). The definition is correct: for $f$ and $g$ that satisfy the respective hypotheses, the value of the integral in (\ref{eq:ZSIntegralSimple-definition}) is independent of the choice of $\theta$, \cite[Proposition 2.1]{Zahle98}.
If $f$ and $g$ are as in the definition and $g$ has bounded variation, then \eqref{eq:ZSIntegralSimple-definition} equals to the classical Lebesgue-Stieltjes integral of $f$ w.r.t. $g$, \cite[Theorem 2.4]{Zahle98}. The following duality estimate and restriction property are well known, see \cite{Zahle98,Zahle01} or \cite{NualartRascanu}. By $\Gamma$ we denote the Euler gamma function.

\begin{proposition}
\label{the:ZSIntegralBound}
Assume that $f \in W^{\theta,1}_0(0,T)$ and $g \in W^{1-\theta,\infty}_T(0,T)$ for some $\theta \in (0,1)$. Then the integral in (\ref{eq:ZSIntegralSimple-definition})  admits the bound 
\begin{equation}
 \label{eq:ZSIntegralBound}
 \left|\int_0^T f_t  dg_t \right|
 \wle \frac{\|f\|_{W^{\theta,1}_0}  \|g\|_{W^{1-\theta,\infty}_T}}{\Gamma(\theta) \Gamma(1-\theta)}.
\end{equation}
Hence, for every $t\in[0,T]$ the restriction $\mathbf{1}_{[0,t]}f$ belongs  to $W^{\theta,1}_0(0,T)$ and 
the integral
$$
\int_0^t f_u  dg_u = \int_0^T \mathbf{1}_{[0,t]}(u)f_u  dg_u
$$
is well-defined.
\end{proposition}

Our first new contribution is the following Theorem \ref{thm:existence} on the regularity and integrability of the composition $\varphi\circ X$ of a H\"older path $X$ with a $BV$-function $\varphi$, the existence of the generalized Lebesgue-Stieltjes integral 
\begin{equation}\label{E:pathwiseint}
\int_0^t \varphi(X_u)d Y_u
\end{equation}
of $\varphi(X)$ with respect to a given H\"older path $Y$ in the sense of \eqref{eq:ZSIntegralSimple-definition}, and the H\"older regularity of this integral, seen as a function of $t\in [0,T]$. 

\begin{theorem}
\label{thm:existence}
Suppose that $\varphi\in BV_{\loc}(\mathbb{R}^n)$ and that $X\in\mathcal{C}^\alpha([0,T],\mathbb{R}^n)$ is $(s,1)$-variable with respect to $\varphi$ for some $s\in (0,1)$.
\begin{enumerate}
\item[(i)] For any $0<\beta<\alpha s$ the composition $\varphi\circ X$ is an element of  $W^{\beta,1}_0(0,T)$. 
\item[(ii)] If in addition 
$Y\in\mathcal{C}^\gamma([0,T],\mathbb{R})$ and $\alpha s+\gamma>1$ then for any $t\in [0,T]$ the integral (\ref{E:pathwiseint})
exists. 
\item[(iii)] Moreover, if $X$ is $(s,p)$-variable with respect to $\varphi$ for some $p\in (1,+\infty]$, then for any $0<\beta<\alpha s$ we have $\varphi\circ X \in W^{\beta,p}(0,T)$. If in addition $Y\in\mathcal{C}^\gamma([0,T],\mathbb{R})$ with $\frac{1}{p}<1-\beta<\gamma$,   where we write $\frac{1}{+\infty}:=0$, then 
\[\int_0^{\cdot} \varphi(X_u)dY_u\in \mathcal{C}^{1-\beta-1/p}([0,T],\mathbb{R})\]
and 
\begin{equation}\label{E:contimultbound}
\left\|\int_0^{\cdot}\varphi(X)dY\right\|_{\mathcal{C}^{1-\beta-1/p}([0,T])}\leq c\left\|\varphi\circ X\right\|_{W^{\beta,p}(0,T)}\left\|Y\right\|_{\mathcal{C}^\gamma([0,T])}.
\end{equation}
\end{enumerate}
\end{theorem}

In Subsections \ref{SS:compositions} and \ref{SS:intexandreg} of Section \ref{S:suffvar} we provide a proof of Theorem \ref{thm:existence}, along with quantitative estimates for the integral (\ref{E:pathwiseint}) involving (\ref{eq:SVp-condition}).

Variability permits the following change of variable formula.

\begin{theorem}
\label{thm:ito}
Let $F\in W^{1,1}_{\loc}(\mathbb{R}^n)$ be such that $\partial_k F\in BV_{\loc}(\mathbb{R}^n)$ for $k=1,\ldots,n$. If $X\in\mathcal{C}^\alpha([0,T],\mathbb{R}^n)$ with $\alpha>\frac12$ is a path which is $(s,1)$-variable w.r.t. each $\partial_k F$ for some $s\in \left(\frac{1-\alpha}{\alpha},1\right)$, then we have 
\begin{equation}\label{E:changeofvar}
F(X_t) = F(\mathring{x}) + \sum_{k=1}^n \int_0^t \partial_k F(X_u)dX^k_u
\end{equation}
for $\mathcal{L}^1$-a.e. $t\in [0,T]$, provided that $\mathring{x} \in \mathbb{R}^n\setminus S_F$. If in addition $F$ is continuous, then (\ref{E:changeofvar}) holds for all $t\in [0,T]$ and no matter where $\mathring{x}\in\mathbb{R}^n$ is located. 
\end{theorem}

The proof of this theorem can be found in Subsection \ref{SS:changeofvarproof} of Section \ref{S:suffvar}.

The following result on the coincidence of (\ref{E:pathwiseint}) and the corresponding Riemann-Stieltjes integral under the $(s,p)$-variability condition 
with large enough $p$ is immediate from Proposition \ref{P:basicest} and Lemma \ref{lemma:Lp-integrability} below together with \cite[Theorems 4.1.1 and 4.2.1]{Zahle98}. We provide this result for systematic reasons while it will not be used in the sequel.

\begin{theorem}
\label{thm:RS-integral}
Let $\varphi\in BV_{\loc}(\mathbb{R}^n)$, let $X\in\mathcal{C}^\alpha([0,T],\mathbb{R}^n)$ be a path which is $(s,p)$-variable with respect to $\varphi$ for some $s\in (0,1)$ and $p\in (\frac{1}{\alpha s},+\infty]$. Then  $\varphi\circ X$ has a (unique) Borel version which is H\"older continuous of any order less than $\alpha s-\frac{1}{p}$. If in addition $Y\in\mathcal{C}^{\gamma}([0,T],\mathbb{R})$ for some $\gamma>1-\alpha s+\frac{1}{p}$, where $\frac{1}{+\infty}:=0$, then both the generalized Lebesgue-Stieltjes integral $\int_0^T \varphi(X_u)\:dY_u$ as in (\ref{E:pathwiseint}) and the Riemann-Stieltjes integral of $\varphi(X)$ w.r.t. $Y$ over $[0,T]$ exist and agree. 

Moreover, if in this case we are given $0<\varepsilon <\alpha s-1/p-1+\gamma$, a refining sequence $(\pi_k)_{k\geq 1}$ of finite partitions $\pi_k = \{0=t_0^{(k)} < t_1^{(k)}<\ldots <t_{N_k}^{(k)}=T\}$
of $[0,T]$ and points $\xi_i^{(k)} \in [t_{i-1}^{(k)},t_i^{(k)}]$, then we have 
\[\left|\int_0^T \varphi(X)\:dY - \sum_{i=1}^{N_k} \varphi(X_{\xi_i^{(k)}})(Y_{t_i^{(k)}}-Y_{t_{i-1}^{(k)}})\right| \leq c|\pi_k|^{\alpha s-1/p-1+\gamma-\varepsilon}\]
for all $k$, where $|\pi_k|=\max_i (t_i^{(k)}-t_{i-1}^{(k)})$ denotes the mesh of the partition $\pi_k$ and $c>0$ is a constant depending on $\alpha$, $\gamma$, $s$, $p$ and $Y$.
\end{theorem}

\section{Systems of differential equations}\label{S:SDE}

In this section we discuss equations of the form 
\begin{equation}\label{eq:SDE}
X_t=\mathring{x}+\int_0^t \sigma(X_u)dY_u, \quad t\in [0,T],
\end{equation}
where $T>0$, $\mathring{x}\in\mathbb{R}^n$, $\sigma: \mathbb{R}^n \to \mathbb{R}^{n\times n}$ is a coefficient function $\sigma=(\sigma_{jk})_{1\leq j,k\leq n}$ and $Y=(Y^1,...,Y^n):[0,T]\to\mathbb{R}^n$ is a given H\"older path. As usual, (\ref{eq:SDE}) is to be understood in the sense that all components $X^j$ of $X=(X^1,...,X^n):[0,T]\to\mathbb{R}^n$ satisfy the equations 
\begin{equation}\label{eq:SDEcomp}
X_t^j=\mathring{x}_j+\sum_{k=1}^n \int_0^t\sigma_{jk}(X_u^1,\ldots,X_u^n)dY_u^k,
\end{equation}
where $\mathring{x}=(\mathring{x}_1,...,\mathring{x}_n)$. Each integral in these sums is a generalized Lebesgue-Stieltjes integral as in (\ref{E:pathwiseint}). 

We employ Definition \ref{D:SVcomponentwise} for matrix coefficients: Given 
$\sigma=(\sigma_{jk})_{1\leq j,k\leq n}$ with $\sigma_{jk}\in BV_{\loc}(\mathbb{R}^n)$ for all $j$ and $k$ and parameters $s\in (0,1)$ and $p\in [1,+\infty]$, a path $X =(X^1,\ldots,X^n):[0,T]\to \mathbb{R}^n$ is called $(s,p)$-variable with respect to $\sigma$, in symbols $X\in V(\sigma,s,p)$, if it is $(s,p)$-variable with respect to all $\sigma_{jk}$. We consider the following notion of solution.

\begin{definition}\label{D:variabilitysolutions}
Let $\sigma=(\sigma_{jk})_{1\leq j,k\leq n}$ be such that $\sigma_{jk}\in BV_{\loc}(\mathbb{R}^n)\cap L^\infty(\mathbb{R}^n)$ for all $j$ and $k$ and let $Y\in\mathcal{C}^\gamma([0,T],\mathbb{R}^n)$. A path $X:[0,T]\to \mathbb{R}^n$ is called a \emph{variability solution for $\sigma$ and $Y$ started at $\mathring{x}\in\mathbb{R}^n$} if 
\begin{enumerate}
\item[(i)] $X_0=\mathring{x}$
\item[(ii)] the path $X$ is in $\mathcal{C}^\alpha([0,T],\mathbb{R}^n)$ and also in $V(\sigma,s)$
for some $s\in (\frac{1-\gamma}{\alpha},1)$ 
\item[(iii)] $X$ satisfies (\ref{eq:SDE}).
\end{enumerate}
\end{definition}

\begin{examples}
Let $n=1$ and let $\sigma=\varphi_\mathcal{C}$ be as in Example \ref{Ex:Cantor} (with $n=1$). If $Y$ is a typical realization of fractional Brownian motion with Hurst index $H>\frac{\log 3}{\log 3+\log 2}$ then by Theorem \ref{T:exnaive} below, variability solutions for $\sigma$ and $Y$ started at zero exist. Also the zero function $X\equiv 0$ is a variability solution. 
\end{examples}

Not every solution is a variability solution as the following example shows.

\begin{examples}
If $n=1$, $\kappa\in (0,1)$, $\sigma(x)=|x|^\kappa$ for $x\in (-1,1)$ and $\sigma(x)\equiv 1$ for $x\in\mathbb{R}\setminus (-1,1)$ and $Y\in\mathcal{C}^\gamma([0,T])$ is nowhere constant then by Theorem \ref{T:exonedim} there are (non-constant) variability solutions for $\sigma$ and $Y$ started at zero. Also the zero function $X\equiv 0$ is a solution, but it can be a variability solution only if $\gamma + \kappa>1$: Since in a neighborhood of zero $\Vert D\sigma\Vert(dx) = |x|^{\kappa-1}dx$ the zero solution is a variability solution only if we can find $s\in (1-\gamma,1)$ such that $|x|^{-s+\kappa-1}$ is integrable around the origin. 
\end{examples}

\begin{remark}
In \cite{LeonNualartTindel} the authors considered monotone and continuous $\sigma$ having power-type non-linearity at the origin similarly as in the preceding example. The assumptions of \cite[Theorem 3.6]{LeonNualartTindel} ensure that their solution candidate is actually an $(s,1)$-variable solution. As pointed out in \cite{LeonNualartTindel}, one cannot expect unique solutions, because also the zero function is a solution. 
\end{remark}

\subsection{Upper regularity and H\"older continuity of coefficients}

To put our results in perspective we begin with a case that encodes a known result. 

Recall that, given $d\geq 0$, a Borel measure $\mu$ on $\mathbb{R}^n$ is said to be \emph{upper $d$-regular on a Borel set $B\subset \mathbb{R}^n$} if there are constants $c>0$ and $r_0>0$ such that 
\[\mu(B(x,r))\leq cr^d,\quad x\in B\cap \supp \mu,\quad 0<r<r_0.\]
We call a function $\varphi\in BV_{\loc}(\mathbb{R}^n)$ \emph{upper $d$-regular on $B$} if $\left\|D\varphi\right\|$ is \emph{upper $d$-regular on $B$}. We say that a function $\varphi=(\varphi_1,...,\varphi_m)\in (BV_{\loc}(\mathbb{R}^n))^m$ is upper $d$-regular on $B$ if each of its components $\varphi_i$ is. If $B=\mathbb{R}^n$ we simply say \emph{upper $d$-regular}. 

The following is verified in Proposition \ref{P:naive2}.

\begin{proposition}\label{P:Hoelderanyway}
Let $\sigma=(\sigma_{jk})_{1\leq j,k\leq n}$ be such that $\sigma_{jk}\in BV_{\loc}(\mathbb{R}^n)$ for all $j$ and $k$, let $s\in (0,1)$, and assume that $\sigma$ is upper $d$-regular with $d>n-1+s$. Then $\sigma$ has a (unique) Borel version which is  H\"older continuous of order $s$ and extends any component-wise Lebesgue representative. Moreover, any path is $(s,\infty)$-variable w.r.t. $\sigma$.
\end{proposition} 

Concerning (\ref{eq:SDE}) we are led back to a well known Peano type existence result.
\begin{theorem}\label{T:exnaive}
Let $\sigma=(\sigma_{jk})_{1\leq j,k\leq n}$ be such that $\sigma_{jk}\in BV_{\loc}(\mathbb{R}^n)$ for all $j$ and $k$, let $s\in (0,1)$, and assume that $\sigma$ is upper $d$-regular with $d>n-1+s$. If $Y\in\mathcal{C}^\gamma([0,T],\mathbb{R}^n)$ for some $\gamma \in (\frac{1}{1+s},1)$, then for any $\mathring{x}\in\mathbb{R}^n$ there exists a variability solution $X\in \mathcal{C}^{\gamma}([0,T],\mathbb{R}^n)$ for $\sigma$ and $Y$ starting at $\mathring{x}$, and $X \in V(\sigma, s,\infty)$. 
\end{theorem} 
For a proof, see for instance \cite[Proposition 5]{Gubinelli} or, in a slightly different formulation, \cite[Theorem 1.20]{LLC}.

We also consider upper regularity conditions for paths. If for a given number $s>0$, a Borel set $B\subset \mathbb{R}^n$ and a path $Y:[0,T]\to \mathbb{R}^n$ we have
\begin{equation}\label{E:strict}
\sup_{x\in Y([0,T])\cap B}\int_0^T |Y_t-x|^{-s}dt<+\infty,
\end{equation}
then there are constants $c>0$ and $r_0>0$ such that 
\begin{equation}\label{E:Xupperreg}
\mathcal{L}^1(\{t\in [0,T]: |Y_t-Y_u|<r\})\leq cr^s, \quad u\in \{\tau\in [0,T]: Y_\tau\in B\}, \quad 0<r<r_0. 
\end{equation}
If (\ref{E:Xupperreg}) holds with some $d>s$ in place of $s$, then (\ref{E:strict}) holds. These facts are standard, see Proposition \ref{P:samething}. Condition (\ref{E:Xupperreg}) says that the occupation measure of $Y$ on $[0,T]$ is upper $s$-regular on $B$, and for $B=\mathbb{R}^n$ this encodes a lower bound on its Hausdorff dimension, see Remark \ref{R:Hausdorfflowerbounded}. If (\ref{E:Xupperreg}) holds, we also say that \emph{$Y$ is upper $s$-regular on $B$} and \emph{upper $s$-regular} if $B=\mathbb{R}^n$.

\begin{remark}\mbox{}
\begin{enumerate}
\item[(i)] No non-constant $\varphi\in BV_{\loc}(\mathbb{R}^n)$ can be upper $d$-regular with $d>n$, as a Lebesgue differentiation argument shows.
\item[(ii)] Suppose $n=1$. Then we trivially have $\dim_H Y([0,T])\leq 1$ for any path $Y$, with equality if $Y$ is non-constant. For a constant path clearly $\dim_H Y([0,T])=0$. If $Y$ is upper $d$-regular, then $d\leq \dim_H Y([0,T])$,
\cite[Theorem 4.13]{Falconer}. A path $Y$ constant on some nonempty open interval cannot be upper $d$-regular for any $d>0$. A nowhere constant Lipschitz path $Y$ on $[0,T]$ that is upper $d$-regular for a given number $d\in (0,1]$ but not upper $d'$-regular for any $d'>d$ is $Y_t=t^{\frac{1}{d}}$, $t\in [0,T]$. In this case we have $\mu_Y^{[0,T]}([0,r))=r^d$ for any $0<r<T^{\frac{1}{d}}$.
\end{enumerate}
\end{remark}

\subsection{Solutions in dimension one}\label{SS:one}

For the case $n=1$ we obtain the following slight modification of the constructive existence results \cite[Theorem 3.3]{Johanna} and \cite[Theorem 2.1]{LauriSole}, see Example \ref{Ex:LauriSole} below. It allows to compensate the failure of $\sigma$ to be H\"older of sufficiently high order in some space region by sufficient activity of $Y$ in some other space region, stated in terms of (\ref{E:strict}).

\begin{theorem}\label{T:exonedim}
Let $\sigma\in BV_{\loc}(\mathbb{R})\cap L^\infty(\mathbb{R})$ be nonnegative $\mathcal{L}^1$-a.e. such that $\frac{1}{\sigma} \in L^1_{\loc}(\mathbb{R})$.
\begin{enumerate}
\item[(i)] The function $g(x)=\int_0^x\frac{dz}{\sigma(z)}$, $x\in\mathbb{R}$, is absolutely continuous and strictly increasing on $\mathbb{R}$. Its inverse $f:=g^{-1}$ is Lipschitz and satisfies $\sigma(f)=f'$ $\mathcal{L}^1$-a.e. on $\mathbb{R}$. 
\item[(ii)] Let $s\in (0,1)$, $\gamma\in (\frac{1}{1+s},1)$, $Y\in \mathcal{C}^\gamma([0,T])$ with $Y_0=0$ and $\mathring{x}\in\mathbb{R}$. Let $-\infty\leq a <b\leq +\infty$. Suppose that $Y$ satisfies (\ref{E:strict}) for $B=(g(a),g(b))$ or that $\sigma$ is upper $d$-regular on $(a,b)$ with $d>s$, and similarly for $\mathbb{R}\setminus (g(a),g(b))$ and $\mathbb{R}\setminus (a,b)$, respectively. Then the function $X\in \mathcal{C}^{\gamma}([0,T])$, defined by
\[X_t=f(Y_t+g(\mathring{x})),\quad t\in [0,T],\]
satisfies (\ref{eq:SDE}). If $\frac{1}{\sigma} \in L^\infty_{\loc}(\mathbb{R})$, then $X\in \mathcal{C}^{\gamma}([0,T])\cap V(\sigma,s)$ is a variability solution for $\sigma$ and $Y$ started at $\mathring{x}$.
\end{enumerate}
\end{theorem}

In the case that in Theorem \ref{T:exonedim} no upper regularity of $\sigma$ is used but $Y$ is assumed to be upper $d_Y$-regular, the H\"older continuity and upper regularity of $Y$ together can produce sufficient regularity: The hypotheses in Theorem \ref{T:exonedim} force the condition 
\begin{equation}\label{E:addcond}
\frac{1}{\gamma}<1+d_Y, 
\end{equation}
and for $s\in (\frac{1}{\gamma}-1, d_Y)$ the theorem yields solutions. If the H\"older order $\gamma$ of $Y$ is higher (lower), lower (higher) upper regularity of $Y$ is needed.

\begin{example}\label{Ex:LauriSole}
A real valued fractional Brownian motion $(B_t^H)_{t\in [0,T]}$ of arbitrary Hurst index $H\in (0,1)$ is almost surely $\gamma$-H\"older for any $\gamma<H$. On the other hand, it is upper $1$-regular almost surely, as can for instance be concluded from the joint continuity of its local times, see \cite[p. 1271]{Berman70}, \cite[Theorem]{Berman73} or \cite{Xiao97}, which entails their local boundedness, see Proposition \ref{P:Xupper-localtime} below. If $H>\frac12$ then for any $\gamma\in (\frac12,H)$ condition (\ref{E:addcond}) holds and for almost every realization (seen as a deterministic path) in place of $Y$, Theorem \ref{T:exonedim} yields variability solutions for any starting point $\mathring{x}$ and no matter whether $\sigma$ is upper regular of any order or not. This recovers the fractional Brownian motion special case of \cite[Theorem 2.1]{LauriSole}, which roughly speaking made use of the facts that, for any $\gamma<H$, the paths of $B^H$ are $\gamma$-H\"older almost surely and that
\begin{equation}\label{E:LauriSole}
\sup_{x\in\mathbb{R}}\mathbb{E}\int_0^T|Y_t-x|^{-\frac{1}{H}}dt<+\infty,
\end{equation}
\cite[Assumption 2.1 and Example 2.1]{LauriSole}. Condition (\ref{E:LauriSole}) implies that for any $s\in (0,1)$
with probability one we have $Y\in V(\sigma,s,1)$ for any $BV$-function $\sigma$. 
\end{example}

\begin{example}
Functions $Y:[0,1]\to\mathbb{R}$ having H\"older and upper regularity properties satisfying the assumptions of Theorem \ref{T:exonedim} can be found in the class of self-affine functions \cite[Definition 1]{Kamae}. A self-affine function of order $H\in (0,1]$ is H\"older continuous of order $H$ but of no order $H'>H$, see \cite[Theorem 1]{Kamae} or \cite{Kono}. For any choice of integers $2\leq b<a$ one can construct self-affine functions $Y$ of order $H=\frac{\log b}{\log a}$ which have bounded local times $L_1^Y$, see \cite[Proposition 7 et sa preuve]{Bertoin}, and for any such $Y$ Theorem \ref{T:exonedim} yields variability solutions no matter whether $\sigma$ is upper regular of any order or not. Although Theorem \ref{T:exonedim} does not apply to this case, it is interesting to note that one can find self-affine functions of order $H=\frac12$ having square integrable local times, \cite[p. 438, Remarques]{Bertoin}, what ensures $Y$ is upper $\frac12$-regular, Proposition \ref{P:Xupper-localtime}.
\end{example}

\begin{example}
Suppose that $\sigma(x)=\varphi_{\mathcal{C}}(\frac{x}{3})\mathbf{1}_{[0,3]}(x)+4\cdot \mathbf{1}_{(3,+\infty)}(x)$, $x\geq 0$, 
where $\varphi_\mathcal{C}$ is as in Example \ref{Ex:Cantor}, and $\sigma(x):=\sigma(-x)$, $x<0$. Then $\sigma(0)=0$ and $\sigma$ is $d_\mathcal{C}$-H\"older except at $x=\pm3$, where it jumps. Since $\nu_{\mathcal{C}}$ is lower $d_{\mathcal{C}}$-regular, we have $\varphi_{\mathcal{C}}(r)\geq cr^{d_{\mathcal{C}}}$ for small $r>0$, which implies that $\sigma^{-1}$ is locally integrable. Since $\nu_\mathcal{C}$ is upper $d_\mathcal{C}$-regular, $\sigma$ is upper $d_{\mathcal{C}}$-regular on $(-2,2)$. Let $d,\delta\in (0,1]$ and let $Z:[0,2]\to\mathbb{R}$ be the function defined by $Z_t=t^{\frac{1}{d}}$ for $t\in (0,1]$ and $Z_t=(1-(1-t))^\delta$ for $t\in (1,2]$. Let $g$ be as in Theorem \ref{T:exonedim} and suppose that $Y$ is such that on each of the intervals $[0,\frac{1}{2}g(1)]$, $[\frac{1}{2}g(1), \frac{3}{4}g(1)]$, $[\frac{3}{4}g(1), \frac{7}{8}g(1)]$ etc. it equals a scaled down copy of $Z$, and on $[g(1),+\infty)$ it equals a typical path of fractional Brownian motion with Hurst index $H$ started at zero. Then $Y$ is H\"older of any order $\gamma< H\wedge \delta$, upper $d$-regular on $(g(-1),g(1))$ and upper $1$-regular on $\mathbb{R}\setminus (g(-1),g(1))$. If $d=\frac12$ and $\delta=H=\frac34$ then Theorem \ref{T:exonedim} yields variability solutions for all starting points, note that on $(-2,2)$ we can use the upper $d_\sigma$-regularity of $\sigma$ and on $\mathbb{R}\setminus (-g(1),g(1))$ the upper $1$-regularity of $Y$. 
\end{example}

\subsection{Doss' transformation}

To understand analogs of Theorem \ref{T:exonedim} for $n\geq 2$ we implement a multidimensional Doss transformation, \cite{Doss,Yamato}. Let $\sigma:\mathbb{R}^n\to\mathbb{R}^{n\times n}$ be a matrix coefficient as before, and let $\sigma_j=(\sigma_{j1},...,\sigma_{jn})$ denote its $j$th row. Suppose that $f = (f_1,\ldots,f_n):\mathbb{R}^n\to\mathbb{R}^n$ is a function with components $f_k : \mathbb{R}^n \to \mathbb{R}$ which satisfies 
the deterministic equation 
\begin{equation}
\label{eq:deterministic-DE}
\nabla f = \sigma(f),
\end{equation}
where we have used the symbol
\[\nabla f=\left(\partial_k f_j\right)_{1\leq j,k\leq n}\] 
for the Jacobian matrix of $f$ to avoid confusion with the symbol for gradient measures. The $j$th row of $\nabla f$ is the gradient $\nabla f_j=(\partial_1 f_j,...,\partial_n f_j)$ of $f_j$, and (\ref{eq:deterministic-DE}) states that $\nabla f_j=\sigma_j(f)$ for all $j$. If we set  
\[X_t^j =  f_j(Y_t^1,\ldots,Y_t^n),\]
then $X=(X^1,...,X^n)$ is a solution to (\ref{eq:SDE}) since
\begin{align}\label{E:strategy}
X_t^j&= f_j(Y^1_0,\ldots, Y^n_0) + \sum_{k=1}^n\int_0^t \partial_kf_j(Y^1_u,\ldots,Y^n_u)dY_u^k \notag\\
&=f_j(Y^1_0,\ldots, Y^n_0)+\sum_{k=1}^n\int_0^t \sigma_{jk}(f_1(Y^1_u,\ldots,Y^n_u),\ldots,f_n(Y^1_u,\ldots,Y^n_u))dY_u^k\\
&= X_0^j + \sum_{k=1}^n \int_0^t \sigma_{jk}(X_u^1,\ldots,X_u^n)dY_u^k,\notag
\end{align}
provided that we can justify the use of the change of variable formula (\ref{E:changeofvar}) with $f_j$ in place of $F$ and (\ref{eq:deterministic-DE}). As in the classical case, \cite{Yamato}, one needs strong assumptions to find solutions $f$ to (\ref{eq:deterministic-DE}), and in our case an additional difficulty is that the components of $\sigma$ are $BV_{\loc}$ only.

Our first main assumption on $\sigma$ is as follows. 

\begin{assumption}\mbox{}\label{A:main}
\begin{enumerate}
\item[(i)] $\sigma=(\sigma_{jk})_{1\leq j,k\leq n}$ with components $\sigma_{jk}\in BV_{\loc}(\mathbb{R}^n)\cap L^\infty(\mathbb{R}^n)$ for all $j$ and $k$,
\item[(ii)] $\det (\sigma) >\varepsilon$ $\mathcal{L}^n$-a.e. on $\mathbb{R}^n$ for some $\varepsilon>0$.
\end{enumerate}
\end{assumption}

\begin{remark}
Since all paths are continuous and $T$ is finite, it suffices to have $\sigma$ with the respective properties defined on a bounded domain in $\mathbb{R}^n$. To save notation and have shorter proofs we formulate Assumption \ref{A:main} as stated.
\end{remark}

The following lemma is a direct consequence of Assumption \ref{A:main} and the Cayley-Hamilton theorem, see Subsection \ref{SS:Prelim}.

\begin{lemma}\label{L:niceinverse}
Let Assumption \ref{A:main} be satisfied. Then there exists a matrix $\hat{\sigma}=(\hat{\sigma}_{jk})_{1\leq j,k\leq n}$ of functions $\hat{\sigma}_{jk}\in BV_{\loc}(\mathbb{R}^n)\cap L^\infty(\mathbb{R}^n)$ so that 
\[\sigma\hat{\sigma}=\hat{\sigma}\sigma=I\quad \text{ $\mathcal{L}^n$-a.e. } \]
and this matrix is unique up to $\mathcal{L}^n$-a.e. equivalence. Moreover, there exists some $\hat{\varepsilon}>0$ such that 
$\det(\hat{\sigma})>\hat{\varepsilon}$ $\mathcal{L}^n$-a.e. on $\mathbb{R}^n$. 
\end{lemma}

By Assumption \ref{A:main} and Lemma \ref{L:niceinverse} our second main assumption makes sense. We write $D_i$ to denote the partial differentiation in direction of the $i$th coordinate in the sense of tempered distributions.

\begin{assumption}[curl-free assumption]\mbox{}\label{A:main2}
For all $i$, $j$, and $k$ we have 
\begin{equation}\label{E:cons} 
D_i\hat{\sigma}_{kj}-D_j\hat{\sigma}_{ki}=0
\end{equation}
in the sense of tempered distributions, that is, $\int_{\mathbb{R}^n} \hat{\sigma}_{kj}\:\partial_i\varphi\:dx-\int_{\mathbb{R}^n} \hat{\sigma}_{ki}\:\partial_j\varphi\:dx=0$ for any $i$, $j$ and $k$ any Schwartz function $\varphi\in\mathcal{S}(\mathbb{R}^n)$.
\end{assumption}

\begin{example}
Let $\sigma$ be a diagonal matrix with entries $\sigma_{ii} \in BV(\mathbb{R}^n)\cap L^\infty(\mathbb{R}^n)$ such that $\sigma_{ii}>\varepsilon$ $\mathcal{L}^1$-a.e. on $\mathbb{R}$ for some $\varepsilon>0$ and for all $i=1,2,\ldots,n$.
Then Assumption \ref{A:main} is clearly satisfied. If also Assumption \ref{A:main2} is satisfied, it follows that $\sigma_{ii}$ depends only on $x_i$. Consequently, we are led back to one-dimensional equations as discussed in Subsection \ref{SS:one}.
\end{example}

\begin{example}\label{Ex:jumpalongline}
Let $n=2$, $c>1$ and consider 
\[\sigma(x_1,x_2)=\left(\begin{array}{cc} c & 1\\ \mathbf{1}_{\{cx_1<x_2\}} & c\end{array}\right),\quad (x_1,x_2)\in\mathbb{R}^2.\]
The function $\sigma$ is bounded and in $BV_{\loc}(\mathbb{R}^2)$, and it jumps across the straight line $J=\{x_2=cx_1\}$. Obviously $\det \sigma\geq c^2-1$, and we find
\[\hat{\sigma}(x_1,x_2)=\frac{\mathbf{1}_{\{cx_1<x_2\}}}{c^2-1}\left( \begin{array}{cc} c & -1\\ -1 & c\end{array} \right) + \frac{\mathbf{1}_{\{cx_1\geq x_2\}}}{c^2}\left( \begin{array}{cc} c & -1\\ 0 & c\end{array} \right).\]
Since $\hat{\sigma}_{11}$, $\hat{\sigma}_{12}$ and $\hat{\sigma}_{22}$ are constant and by \cite[(3.90)]{AFP} the measure $D\hat{\sigma}_{21}$ is a constant multiple of $\mathcal{H}^1|_J$, (\ref{E:cons}) trivially holds. 
\end{example}

Assumptions \ref{A:main} and \ref{A:main2} allow the following result for solutions to (\ref{eq:deterministic-DE}).

\begin{proposition}[angle condition]\label{P:solvedeteq}
Suppose that $\sigma$ satisfies Assumption \ref{A:main}. For $n\geq 2$ suppose also that it satisfies Assumption \ref{A:main2} and that there exists $\delta>-1$ such that for $\mathcal{L}^n$-a.e. $x\in\mathbb{R}^n$ we have 
\begin{equation}\label{E:angular}
\left\langle \xi, \sigma(x)\:\xi\right\rangle \geq \delta\:|\sigma(x)\:\xi|\:|\xi|,\quad \xi\in\mathbb{R}^n.
\end{equation}
Then there exists a bi-Lipschitz function $f:\mathbb{R}^n\to\mathbb{R}^n$ which solves (\ref{eq:deterministic-DE}).
\end{proposition}

Condition (\ref{E:angular}) ensures the applicability of a global inversion result, Proposition \ref{P:Kovalev}, which yields a global Lipschitz solution $f$ to (\ref{eq:deterministic-DE}).

\begin{remark}
Let $n\ge 2$. Then $\sigma\in L^1_{\operatorname{loc}}(\mathbb{R}^n,\mathbb{R}^{n\times n})$ satisfies \eqref{E:angular} if and only if $\operatorname{spec}(\sigma(x))\cap (-\infty,0)=\emptyset$ for $\mathcal{L}^n$-a.e. $x\in\mathbb{R}^n$,
where $\operatorname{spec}(\sigma(x))$ denotes the spectrum of the square matrix $\sigma(x)$, see \cite[p. 3]{Kovalev},
\cite{Auzinger} and \cite[Proposition 2.1]{LinsSpitkovskyZhong}. This holds in particular if $\sigma$ has an $\mathcal{L}^n$-a.e. nonnegative definite symmetric part.
\end{remark}

\begin{example}\label{Ex:jumpalongline2}
The coefficient $\sigma$ in Example \ref{Ex:jumpalongline} is positive definite, and hence (\ref{E:angular}) is immediate. Moreover, the function 
\[f(y_1,y_2)=(cy_1+y_2, y_1\mathbf{1}_{\{y_1<0\}}+cy_2), \quad (y_1,y_2)\in\mathbb{R}^2,\] 
solves (\ref{eq:deterministic-DE}) for this $\sigma$. 
\end{example}

\begin{example}\label{Ex:jumpalongthreelines}
Let $n=2$, $0<a<b$ and consider the cone $C=\left\lbrace\frac{a}{b}x_1<x_2<\frac{b}{a}x_1\right\rbrace$. The function
\[\sigma(x_1,x_2)=\left(\begin{array}{cc} b & a\mathbf{1}_{C}(x_1,x_2)\\
a\mathbf{1}_{C}(x_1,x_2) & b\end{array}\right),\quad (x_1,x_2)\in\mathbb{R}^2,\]
satisfies Assumptions \ref{A:main} and \ref{A:main2} and (\ref{E:angular}). A solution to (\ref{eq:deterministic-DE}) is given by
\[f(y_1,y_2)=(by_1+ay_2\mathbf{1}_{\{y_1>0,y_2>0\}}, ay_1\mathbf{1}_{\{y_1>0,y_2>0\}} + by_2),\quad (y_1,y_2)\in\mathbb{R}^2.\]
\end{example}

\begin{example}\label{Ex:jumpandCantor}
Let $n=2$, let $\varphi_{\mathcal{C}}:\mathbb{R}\to\mathbb{R}$ be as in Example \ref{Ex:Cantor} (with $n=1$), and $\Phi(z)=\int_0^z(1+\varphi_{\mathcal{C}}(z))dz$, $z\in\mathbb{R}$. The function
\[
\sigma(x_1,x_2)=\mathbf{1}_{\{x_1\leq 0\}}\left(\begin{array}{cc} 1 & 0\\
0 & 1+\varphi_{\mathcal{C}}(\Phi^{-1}(x_2))\end{array}\right)
+ \mathbf{1}_{\{x_1> 0\}}\left(\begin{array}{cc} 1 & 0\\
1 & 1+\varphi_{\mathcal{C}}(\Phi^{-1}(x_2-x_1))\end{array}\right) ,\quad (x_1,x_2)\in\mathbb{R}^2,
\]
satisfies Assumptions \ref{A:main} and \ref{A:main2} and (\ref{E:angular}), the fact that $\sigma_{22}\in BV_{\loc}(\mathbb{R}^n)$ follows using \cite[Theorem 3.16]{AFP}. For this choice of $\sigma$ the function
\[f(y_1,y_2)=(y_1, y_1\mathbf{1}_{\{y_1>0\}}+\Phi(y_2)), \quad (y_1,y_2)\in\mathbb{R}^2,\]
solves (\ref{eq:deterministic-DE}).
\end{example}

\begin{remark}
The existence of almost everywhere defined local inverses of Sobolev functions is treated in \cite{FonsecaGangbo, FonsecaGangbobook}. In \cite[Section 5]{Barchiesietal} fine properties of a.e. local inverses are studied
under the requirement that $\det(\hat{\sigma}(x))>0$ for $\mathcal{L}^n$-a.e. $x\in \mathbb{R}^n$. The local inverse becomes an everywhere defined local homeomorphism on open subsets where the determinant of $\hat{\sigma}$ equals a positive constant, see \cite[Corollary 6.3]{FonsecaGangbobook}.
\end{remark}

\subsection{Solutions in arbitrary dimensions}

The following is a version of Theorem \ref{T:exonedim} for arbitrary $n\geq 1$. 

\begin{theorem}\label{T:exjoint}
Suppose that $\sigma$ is as in Assumption \ref{A:main} and $f:\mathbb{R}^n\to\mathbb{R}^n$ is a bi-Lipschitz function which solves (\ref{eq:deterministic-DE}). Let $s\in (0,1)$, $\gamma\in (\frac{1}{1+s},1)$, $Y\in\mathcal{C}^\gamma([0,T],\mathbb{R}^n)$ with $Y_0=0$, $B\subset \mathbb{R}^n$ is a Borel set and $\mathring{x} \in\mathbb{R}^n$.  Suppose that $\sigma$ is upper $d$-regular on $B$ with $d>n-1+s$ or that 
\begin{equation}\label{E:condonY}
\sup_{x\in Y([0,T])\cap (f^{-1}(B)-f^{-1}(\mathring{x}))} \int_0^T |Y_t-x|^{-n+1-s}dt<+\infty,
\end{equation}
and suppose the same is true for $B^c$ in place of $B$. Then the path 
\begin{equation}\label{E:candidate}
X_t=f(Y_t+f^{-1}(\mathring{x})),\quad t\in [0,T],
\end{equation}
is a variability solution $X\in \mathcal{C}^{\gamma}([0,T],\mathbb{R}^n)\cap V(\sigma,s)$ for $\sigma$ and $Y$ started at $\mathring{x}$.
\end{theorem}

By Proposition \ref{P:solvedeteq}, the first sentence in Theorem \ref{T:exjoint} could be replaced by requiring $\sigma$ to satisfy Assumptions \ref{A:main} and, if $n\geq 2$, also Assumption \ref{A:main2} and (\ref{E:angular}). Note that for $n=1$, Theorem \ref{T:exonedim} gives the same result under less restrictive assumptions on $\sigma$.

If only the upper regularity of $\sigma$ is used, Theorem \ref{T:exjoint} complements Theorem \ref{T:exnaive} by constructing a solution. If also the upper regularity of $Y$ is used one can, in some cases, obtain solutions for discontinuous $\sigma$. However, the upper regularity condition on $Y$ is quite restrictive: Already if $n=2$ and $\sigma$ has jumps one cannot hope to use Theorem \ref{T:exjoint} to obtain solutions to (\ref{eq:SDE}) when $Y$ is a typical path of a fractional Brownian motion $B^H$ with Hurst index $H>\frac12$: On the one hand, $B^H$ is almost surely $\gamma$-H\"older of any order $\gamma<H$, what forces $\frac{1}{H}<1+s$. On the other hand, the Hausdorff dimension $\dim_H B^H([0,T])$ of $B^H([0,T])$ is known to equal $\frac{1}{H}$ almost surely, see \cite[Chapter 11]{Falconer} or \cite[Section 18.3, Theorem 1 and Corollary]{Kahane}, so that we would run into the contradictory chain of inequalities $s<d_Y\leq \dim_H Y([0,T])=\frac{1}{H}<1+s$. For $n=1$ one has $\dim_H B^H([0,T])=1$ almost surely, and this contradiction does not occur.


The next result uses a weighted version of (\ref{E:strict}) together with a moment condition on the gradient measures of the coefficient.

\begin{theorem}\label{T:exshift}
Suppose that $\sigma$ satisfies Assumption \ref{A:main} and $f:\mathbb{R}^n\to\mathbb{R}^n$ is a bi-Lipschitz function which solves (\ref{eq:deterministic-DE}). Let $s\in (0,1)$, $\gamma\in (\frac{1}{1+s},1)$, $Y\in\mathcal{C}^\gamma([0,T],\mathbb{R}^n)$ with $Y_0=0$ and let $\mathring{x}\in\mathbb{R}^n$. Suppose that there are $\varepsilon\in (0,1-s)$, $c>0$, and $\delta\in (0,n-1+s+\varepsilon)$ such that 
\begin{equation}\label{E:LauriSolen}
\int_0^T |Y_t-x|^{-n+\varepsilon}dt<c|x|^{-n+\delta},\quad x\in\mathbb{R}^n,
\end{equation}
and for all $j$ and $k$ we have
\begin{equation}\label{E:momentconditionx0}
\int_{\mathbb{R}^n}|x-\mathring{x}|^{-n+1-s-\varepsilon+\delta}\left\|D\sigma_{jk}\right\|(dx)<+\infty.
\end{equation}
Then (\ref{E:candidate}) defines a variability solution $X\in \mathcal{C}^{\gamma}([0,T],\mathbb{R}^n)\cap V(\sigma,s)$ for $\sigma$ and $Y$ started at $\mathring{x}$.
\end{theorem}

In comparison with (\ref{E:LauriSole}) conditions (\ref{E:strict}) and (\ref{E:LauriSolen}) are quite restrictive. They ensure a rather weak regularization solely due to dimensional effects. Condition (\ref{E:LauriSole}) encodes a strong additional regularization by randomness. Also Theorem \ref{T:exshift} becomes efficient in a probabilistic context. The following corollary may be seen as a generalization of \cite[Theorem 3.3]{Johanna}, and \cite[Theorem 2.1]{LauriSole}. 

\begin{corollary}\label{C:exshiftrandom}
Suppose that $\sigma$ satisfies Assumption \ref{A:main} and $f:\mathbb{R}^n\to\mathbb{R}^n$ is a bi-Lipschitz function which solves (\ref{eq:deterministic-DE}). Let $s\in (0,1)$, $\gamma\in (\frac{1}{1+s},1)$, let $Y=(Y_t)_{t\in [0,T]}$ be an $\mathbb{R}^n$-valued stochastic process with $Y_0=0$ on a probability space $(\Omega,\mathcal{F},\mathbb{P})$ with paths $\mathbb{P}$-a.s. H\"older continuous of order $\gamma$, and let $\mathring{x}\in\mathbb{R}^n$. Suppose that there are $\varepsilon\in (0,1-s)$, $c>0$, and $\delta\in (0,n-1+s+\varepsilon)$ such that 
\begin{equation}\label{E:LauriSolenrandom}
\mathbb{E}\int_0^T |Y_t-x|^{-n+\varepsilon}dt<c|x|^{-n+\delta},\quad x\in\mathbb{R}^n,
\end{equation}
and for all $j$ and $k$ we have (\ref{E:momentconditionx0}). Then for $\mathbb{P}$-a.e. $\omega\in \Omega$ the path 
\begin{equation}\label{E:candidaterandom}
X_t(\omega)=f(Y_t(\omega)+f^{-1}(\mathring{x})), \quad t\in [0,T],
\end{equation}
is a variability solution $X\in \mathcal{C}^{\gamma}([0,T],\mathbb{R}^n)\cap V(\sigma,s)$ for $\sigma$ and $Y$ started at $\mathring{x}$.
\end{corollary}

At any time $t>0$ the averaging effect of the process $Y$ leads to a regularization. The moment condition (\ref{E:momentconditionx0}) excludes a too bad behavior of $\sigma$ at the starting point $\mathring{x}$ at time $t=0$ when this effect is not yet present.

\begin{example}
Let $n= 2$ and let $Y=B^H$ be a fractional Brownian motion with Hurst index $H>\frac12$. It satisfies (\ref{E:LauriSolenrandom}) with $0<\varepsilon<2-\frac{1}{H}$ and $\delta=\frac{1}{H}+\varepsilon$, see Example \ref{Ex:fBMexample}. To apply Corollary \ref{C:exshiftrandom} suppose that $s\in (0,1)$ and that $\varepsilon$ above also satisfies $0<\varepsilon<1-s$. If $\sigma$ satisfies Assumptions \ref{A:main} and \ref{A:main2}, (\ref{E:angular}), and $\mathring{x}$ is such that 
\begin{equation}\label{E:simplemomentcond}
\int_{\mathbb{R}^n}|x-\mathring{x}|^{-n+1-s+\frac{1}{H}}\left\|D\sigma_{jk}\right\|(dx)<+\infty, \quad j,k=1,2,
\end{equation}
then for $\mathbb{P}$-a.s. realization of $Y$ the path $X$ as in Corollary \ref{C:exshiftrandom} is a variability solution for $\sigma$ and $Y$ started at $\mathring{x}$, and $X\in \mathcal{C}^{\gamma}([0,T],\mathbb{R}^2) \cap V(\sigma,s)$ for any $\gamma<H$. Note that if all $\sigma_{jk}$ are actually in $BV(\mathbb{R}^n)$ then (since obviously $s<\frac{1}{H}$) condition (\ref{E:simplemomentcond}) is automatically satisfied for $\mathcal{H}^{n-1}$-a.e. $\mathring{x}\in\mathbb{R}^2$, see Remark \ref{R:densitycomparison}.
\end{example}

\subsection{A uniqueness result}

We establish a uniqueness result for variability solutions.

\begin{theorem}
\label{thm:DE-uniqueness}
Suppose that $\sigma=(\sigma_{jk})_{1\leq j,k\leq n}$ satisfies Assumptions \ref{A:main} and \ref{A:main2}, $Y\in\mathcal{C}^\gamma([0,T],\mathbb{R}^n)$ for some $\gamma\in (0,1)$, and $\mathring{x}\in\mathbb{R}^n$. 
Then there exists at most one variability solution of H\"older order greater $\frac{1}{2}$ for $\sigma$ and $Y$ started at $\mathring{x}$. 
\end{theorem}

The proof of Theorem \ref{thm:DE-uniqueness} shows that the only solution candidate is (\ref{E:candidate}). We can conclude the following results.
\begin{corollary}
\label{cor:existence-uniqueness}
Suppose that Assumptions \ref{A:main} and \ref{A:main2} hold and let $f$ be the solution to \eqref{eq:deterministic-DE}. If assumptions of Theorem \ref{T:exjoint} (or Theorem \ref{T:exonedim} in the case $n=1$) are satisfied, then for any $\mathring{x}\in\mathbb{R}^n$ there exists a unique  variability solution $X\in \mathcal{C}^{\gamma}([0,T],\mathbb{R}^n)$ for $\sigma$ and $Y$ starting at $\mathring{x}$, given by (\ref{E:candidate}).
\end{corollary}

\begin{corollary}
\label{cor:existence-uniqueness-random}
Suppose that Assumptions \ref{A:main} and \ref{A:main2} hold and let $f$ be the solution to \eqref{eq:deterministic-DE}. If assumptions of Corollary \ref{C:exshiftrandom} are satisfied, then for any $\mathring{x}\in\mathbb{R}^n$ and $\mathbb{P}$-a.e. $\omega\in\Omega$ there exists a unique  variability solution $X(\omega)\in \mathcal{C}^{\gamma}([0,T],\mathbb{R}^n)$ for $\sigma$ and $Y(\omega)$ starting at $\mathring{x}$, given by (\ref{E:candidaterandom}).
\end{corollary}

\begin{remark}
Following the ideas of \cite{LauriSole}, we could replace $\det (\sigma) >\varepsilon$ in Assumption \ref{A:main} with
$\det (\sigma) \geq 0$. In this case we would obtain a uniqueness result similar to Theorem \ref{thm:DE-uniqueness} up to the first time when $\det(\sigma(X_t)) =0$. However, the existence of the solution becomes more complicated. 
\end{remark}

\section{Variability and consequences}\label{S:suffvar}

We provide a potential theoretic interpretation of the $(s,1)$-variability. This allows to prove that $\varphi\circ X$ is a well defined class in $W^{\beta,1}_0(0,T)$.

\subsection{Riesz potentials and occupation measures}\label{SS:occu}

The \emph{Riesz potential of order $0<\gamma<n$} of a nonnegative Borel measure $\nu$ on $\mathbb{R}^n$ is defined by
\begin{equation}\label{E:potential}
U^\gamma \nu(x):=c_\gamma\int_{\mathbb{R}^n} |x-y|^{-n+\gamma} \nu(dy), \quad x\in\mathbb{R}^n,
\end{equation}
where $c_\gamma>0$ is a well-known constant depending only on $n$ and $\gamma$, see for instance \cite[Chapter I, Section 3]{Landkof} or \cite[Section V.4]{BliedtnerHansen}. The \emph{mutual Riesz energy of order $0<\gamma<n$} of two nonnegative Borel measures $\nu_1$ and $\nu_2$ on $\mathbb{R}^n$ is defined by
\begin{equation}\label{E:energy}
I^\gamma(\nu_1,\nu_2):=c_\gamma\int_{\mathbb{R}^n}\int_{\mathbb{R}^n}|x-y|^{-n+\gamma} \nu_1(dy)\nu_2(dx).
\end{equation}
This quantity takes values in $[0,+\infty]$, is symmetric in $\nu_1$ and $\nu_2$, and we have
\[I^\gamma(\nu_1,\nu_2)=\int_{\mathbb{R}^n} U^\gamma \nu_1(x)\nu_2(dx),\]
see \cite[Chapter I, Section 4]{Landkof}. We write $I^\gamma \nu:=I^\gamma(\nu,\nu)$.

Given a path $X=(X_t)_{t\in [0,T]}$ we can define a finite nonnegative Radon measure $\mu_X^{[0,T]}$ on $\mathbb{R}^n$ by the identity
\begin{equation}\label{E:occutimeformula}
\int_{\mathbb{R}^n} g(y)\mu_X^{[0,T]}(dy)=\int_0^T g(X_t)dt,
\end{equation}
valid for any bounded Borel function $g:\mathbb{R}^n\to\mathbb{R}$. To $\mu_X^{[0,T]}$ one usually refers as the \emph{occupation measure of $X$} and to (\ref{E:occutimeformula}) as the \emph{occupation time formula}. The measure $\mu_X^{[0,T]}$ is supported on the compact set $X([0,T])$. Clearly the occupation measure depends on the parametrization of the path $X$.

\begin{remark}\label{R:occu_geo}\mbox{}
\begin{enumerate}
\item[(i)] Discussions of absolutely continuous occupation measures in terms of their densities, referred to as \emph{local times},  are a classical subject in probability theory, \cite{Berman69, Berman70, GemanHorowitz, Xiao06}. As mentioned in \cite[Section 3]{GemanHorowitz}, applications of occupation densities to nonrandom functions are much less discussed. Results on possibly singular occupation measures are rather sparse and have mainly been used to obtain results on dimensions of images or graphs of stochastic processes, see for instance \cite[Section 16]{Falconer} and \cite[Section 4]{Xiao} and the references cited there. 
\item[(ii)] In the special case that $X$ is an absolutely continuous curve parametrized to have unit speed the right hand side of (\ref{E:occutimeformula}) is just the line integral of $g$ along $X$, and formula (\ref{E:occutimeformula}) may be seen as a special case of the area formula for the path $X$, \cite[Theorem 3.2.6]{Federer}. Another interpretation of  formulas of type (\ref{E:occutimeformula}) in the spirit of geometric measure theory have been established in \cite{BevFlan14}. There the authors proved an occupation time formula for $\mathbb{R}^n$-valued semimartingales which has features of a coarea formula for $\mathcal{C}^2(\mathbb{R}^n)$-functions. They observed that the fluctuations of the semimartingale can lead to the existence of certain `transversal' (in a Rokhlin sense) densities which may be seen as generalizations of local times of one-dimensional semimartingales, see \cite[Theorems 1 and 3]{BevFlan14}.
\end{enumerate}
\end{remark}

The definitions (\ref{E:potential}) and (\ref{E:energy}) and the identity (\ref{E:occutimeformula}) now immediately allow to rewrite (\ref{eq:SVp-condition}) in terms of Riesz energies. Given a function $\varphi \in BV_{\loc}(\mathbb{R}^n)$ and a set $\mathcal{U}\subset \mathbb{R}^n$ we use
\[\mu_{\varphi,\mathcal{U}}:=\left\|D\varphi\right\||_{\mathcal{U}}\]
to abbreviate the restriction of $\left\|D\varphi\right\|$ to $\mathcal{U}$.

\begin{remark}\label{R:SVviaoccu}
Let $\varphi \in BV_{\loc}(\mathbb{R}^n)$, $p\in [1,+\infty]$ and $s\in (0,1)$. A path $X$ is $(s,p)$-variable with respect to $\varphi$ if and only if there is a relatively compact open neighborhood $\mathcal{U}$ of $X([0,T])$ such that 
\begin{equation}\label{E:SVp2}
U^{1-s}\mu_{\varphi,\mathcal{U}} \in L^p(\mathbb{R}^n,\mu_X^{[0,T]}).
\end{equation}
In particular, $X$ is $(s,1)$-variable with respect to $\varphi$ if and only if there is a relatively compact open neighborhood of $X([0,T])$ such that the mutual Riesz energy of order $1-s$ of $\mu_{\varphi,\mathcal{U}}$ and $\mu_X^{[0,T]}$ is finite, 
\begin{equation}\label{E:SVviaoccu}
\int_{\mathbb{R}^n}U^{1-s}\mu_{\varphi,\mathcal{U}}(x)\mu_X^{[0,T]}(dx)<+\infty.
\end{equation}
\end{remark}

\begin{remark}\label{R:occu_use}\protect
\begin{enumerate}
\item[(i)] Both ideas, the variability of curves and the use of Riesz energies, also appear in connection with regularization properties of operators in harmonic analysis, see for instance \cite{TaoWright} for the first idea and \cite{GrahamHareRitter, HareRoginskaya} for the second.
\item[(ii)] Occupation measures appear naturally in connection with \emph{regularization by noise}, \cite{CatellierGubinelli, Davie, Flandoli, FlandoliGubinelliPriola, GG20a, HarangPerkowski, KrylovRoeckner, NualartOuknine, Veretennikov, BG2018}. The addition (in the simplest case) of a `fast moving' perturbation to an equation can provoke well-posedness for otherwise non well-posed equations. A typical key step is to observe the H\"older (or even $\mathcal{C}^{1+\alpha}$-) regularity of integrals of functions (with certain continuity or integrability properties) w.r.t. occupation measures, seen as a function of the starting point of the path. See for instance \cite[Section 2.1 and in particular, Theorems 2.5 and 2.6]{Flandoli}.
\end{enumerate}
\end{remark}

The following consequence of Remark \ref{R:SVviaoccu} allows to define compositions of paths and $BV$-functions. It views (\ref{E:SVviaoccu}) as a condition on $X$ relative to $\varphi$. Recall the meaning of the symbols $S_\varphi$ and $L^\infty(X)$ from Definitions \ref{D:geometric-point-sets} and \ref{D:comp}.

\begin{proposition}\label{P:key}
Let $\varphi \in BV_{\loc}(\mathbb{R}^n)$ and suppose that $X$ is $(s,1)$-variable with respect to $\varphi$ for some $s\in (0,1)$. Then the discontinuity set $S_\varphi$ of $\varphi$ is a zero set for the occupation measure $\mu_X^{[0,T]}$ of $X$,
\begin{equation}\label{E:badsetnullset}
\mu_X^{[0,T]}(S_\varphi)=0.
\end{equation}
In particular, we have 
\begin{equation}\label{E:timenullset}
\mathcal{L}^1(\{t\in [0,T]: X_t\in S_\varphi\})=0
\end{equation}
and 
\begin{equation}\label{E:Linftybound}
\left\|\varphi\right\|_{L^\infty(X)}\leq \left\|\varphi\right\|_{L^\infty(\mathbb{R}^n)}.
\end{equation}
\end{proposition}

Proposition \ref{P:key} immediately implies Lemma \ref{lem:welldef}, and therefore the correctness of Definition \ref{D:comp}.

\begin{proof}[Proof of Lemma \ref{lem:welldef}]
If $\widetilde{\varphi}^{(1)}$ and $\widetilde{\varphi}^{(2)}$ both are (component-wise) Lebesgue representatives of $\varphi=(\varphi_1,...,\varphi_m)\in BV_{\loc}(\mathbb{R}^n)^m$ and $X$ is $(s,1)$-variable w.r.t. each $\varphi_i$, then by (\ref{E:admissible}), (\ref{E:Sphiisunion}) and (\ref{E:timenullset}) we have
\[\int_0^T |\widetilde{\varphi}^{(1)}(X_t)-\widetilde{\varphi}^{(2)}(X_t)|dt= \int_{\{t\in [0,T]:\: X_t\in \mathbb{R}^n\setminus S_\varphi\}} |\widetilde{\varphi}^{(1)}(X_t)-\widetilde{\varphi}^{(2)}(X_t)|dt =0.\]
\end{proof}

Recall that for any $\gamma\geq 0$ the upper $\gamma$-density $\Theta^\ast_\gamma\nu(x)$ of a Borel measure $\nu$ on $\mathbb{R}^n$ at a point $x\in\mathbb{R}^n$ is defined by 
\begin{equation}\label{E:upperdensity}
\Theta^\ast_\gamma\nu(x)=\limsup_{r\to 0}\frac{\nu(B(x,r))}{r^\gamma}, 
\end{equation}
see for instance \cite[Definition 2.2.5]{AFP} or \cite[Definition 6.8]{Mattila}. We prove Proposition \ref{P:key}.

\begin{proof}[Proof of Proposition \ref{P:key}]
Let $\mathcal{U}\subset\mathbb{R}^n$ be a relatively compact open set containing the support $X([0,T])$ of $\mu_X^{[0,T]}$ as in (\ref{E:SVviaoccu}). Since the integral is finite we can find a Borel set $N\subset \mathbb{R}^n$ such that $\mu_X^{[0,T]}(N)=0$
and, for all $x\in \mathbb{R}^n\setminus N$, we have $U^{1-s}\mu_{\varphi,\mathcal{U}}(x)<+\infty$. Therefore, by a standard conclusion, cf. \cite[Chapter 8]{Mattila}, 
\begin{align}\label{E:standard}
\Theta^\ast_{n-1+s}\mu_{\varphi,\mathcal{U}}(x)&=\limsup_{r\to 0}r^{-n+1-s}\mu_{\varphi,\mathcal{U}}(B(x,r))\notag\\
&\leq \lim_{r\to 0}\int_{B(x,r)}|x-y|^{-n+1-s}\mu_{\varphi,\mathcal{U}}(dy)=0
\end{align}
for all such $x$. Since $\mu_{\varphi,\mathcal{U}}=D(\varphi|_\mathcal{U})$, where $\varphi|_\mathcal{U}\in BV(\mathcal{U})$ is the restriction of $\varphi$ to $\mathcal{U}$, (\ref{E:standard}) trivially implies
\begin{equation}\label{E:n-1densityzero}
\Theta^\ast_{n-1}\left\|D(\varphi|_\mathcal{U})\right\|(x)=\limsup_{r\to 0}r^{-n+1}\left\|D(\varphi|_\mathcal{U})\right\|(B(x,r))=0
\end{equation}
and, with suitable $r_0(x)>0$,  
\begin{equation}\label{E:integralfinite}
\int_0^{r_0(x)}r^{-n}\left\|D(\varphi|_\mathcal{U})\right\|(B(x,r))\leq \int_0^{r_0(x)}r^{s-1}dr<+\infty
\end{equation}
for all $x\in \mathcal{U}\setminus N$. However, from (\ref{E:n-1densityzero}) and (\ref{E:integralfinite}) it then follows that $\varphi|_\mathcal{U}$ has an approximate limit at all $x \in \mathcal{U}\setminus N$, see \cite[Remark 3.82]{AFP}. In other words, $S_\varphi \cap \mathcal{U} \subset N$, and since $\mu_X^{[0,T]}(S_\varphi \setminus \mathcal{U})\leq \mu_X^{[0,T]}(\mathbb{R}^n \setminus \mathcal{U})=0$, (\ref{E:badsetnullset}) follows. By monotone convergence (\ref{E:occutimeformula}) extends to all nonnegative Borel function $g:\mathbb{R}^n\to [0,+\infty]$, and for $g=\mathbf{1}_{S_\varphi}$ we obtain 
(\ref{E:timenullset}). Estimate (\ref{E:Linftybound}) follows from (\ref{E:admissible}): For any $x\in \mathbb{R}^n\setminus S_{\varphi}$ we have 
\[|\varphi(x)|\leq \limsup_{r\to 0}\frac{1}{\mathcal{L}^n(B(x,r))}\int_{B(x,r)}|\varphi(y)|dy\leq \left\|\varphi\right\|_{L^\infty(\mathbb{R}^n)}.\]
\end{proof}

\begin{remark}\label{R:upperdensities}\mbox{}
Condition (\ref{E:SVviaoccu}) may also be seen as a requirement on $\varphi$ relative to the path $X$: The measure $\mu_X^{[0,T]}$ is concentrated on $\mathbb{R}^n\setminus N$ with $N$ as in the above proof. Potential theory forbids $\mu_{\varphi,\mathcal{U}}$ to charge subsets of $\mathbb{R}^n\setminus N$ with finite $\mathcal{H}^{n-1+s}$-measure: Any such set $A$ would have $(1-s)$-Riesz capacity zero, \cite[Chapter III, Theorem 3.14]{Landkof}, and this would imply the existence of some point $\mathring{x}\in A$ with $U^{1-s}\mu_X^{[0,T]}(\mathring{x})=+\infty$, \cite[Chapter II, Theorem 2.4]{Landkof}. An alternative, geometric proof of this fact can be given in terms of upper densities: Suppose that $A$ is a Borel subset of $\mathbb{R}^n\setminus N$ with $\mathcal{H}^{n-1+s}(A)<+\infty$. Then by (\ref{E:standard}) together with the density comparison theorem, \cite[Theorem 2.56]{AFP} (or \cite[Chapter 6, Theorem 6.9]{Mattila}), we have 
\[\mu_{\varphi,\mathcal{U}}(A)\leq \sup_{x\in A} \Theta^\ast_{n-1+s}\mu_{\varphi,K}(x)\:\mathcal{H}^{n-1+s}(A)=0. \]
\end{remark}

\begin{remark}\label{R:densitycomparison}
The density comparison theorem also implies that for any $\varphi\in BV_{\loc}(\mathbb{R}^n)$ the set of points $x\in \mathbb{R}^n$ at which $\Theta^\ast_{n-1}\left\|D\varphi\right\|(x)=+\infty$ is a $\mathcal{H}^{n-1}$-null set, see \cite[proof of Lemma 3.75]{AFP}.
\end{remark}

The following example gives a coarse sufficient condition for (\ref{E:SVviaoccu}).

\begin{example}\label{Ex:CS}
If $s\in (0,1)$, 
\begin{equation}\label{E:finiteenergypath}
I^{1-s}\mu_X^{[0,T]}<+\infty
\end{equation}
and $I^{1-s}\mu_{\varphi,\mathcal{U}}<+\infty$, then (\ref{E:SVviaoccu}) holds by the Cauchy-Schwarz inequality for the energy, see \cite[Chapter I, Section 4]{Landkof}. 
\end{example}

\subsection{Upper regularity and bounded potentials}\label{SS:suff}

The first maximum principle, \cite[Chapter I, Theorem 1.10]{Landkof}, states that, for any $s\in (0,1)$ and any Borel measure $\nu$, a bound 
\begin{equation}\label{E:maxprinciple}
U^{1-s}\nu(x)\leq M,
\end{equation}
valid with some $M>0$ for $\nu$-a.e. $x\in\mathbb{R}^n$, implies (\ref{E:maxprinciple}) for all $x\in\mathbb{R}^n$. 
\begin{proposition}\label{P:naive} Let $s\in (0,1)$. Suppose that there is a constant $M>0$ such that 
\begin{equation}\label{E:Xextreme}
U^{1-s}\mu_X^{[0,T]}(x)<M
\end{equation}
for $\mu_X^{[0,T]}$-a.e. $x\in X([0,T])$. Then (\ref{E:SVviaoccu}) holds with any $\varphi\in BV_{\loc}(\mathbb{R}^n)$ and all relatively compact open $\mathcal{U}$, and $\mu_X^{[0,T]}$ cannot charge any set of $A$ of finite $\mathcal{H}^{n-1+s}$-measure.
\end{proposition}
\begin{proof}
Since $\mu_X^{[0,T]}$ is supported in $X([0,T])$, the bound (\ref{E:Xextreme}) holds for $\mu_X^{[0,T]}$-a.e. $x\in\mathbb{R}^n$. By the maximum principle it then holds for all $x\in\mathbb{R}^n$ and as a consequence, (\ref{E:SVviaoccu}) holds with any $\varphi\in BV_{\loc}(\mathbb{R}^n)$ and all relatively compact open $\mathcal{U}$. The second part of the statement follows as in Remark \ref{R:upperdensities}.
\end{proof}
Note that (\ref{E:Xextreme}) implies (\ref{E:finiteenergypath}), but in contrast to Example \ref{Ex:CS} no assumption is made on the measures $\mu_{\varphi,\mathcal{U}}$.

We prepare a counterpart of Proposition \ref{P:naive} in terms of $\left\|D\varphi\right\|$. The gradient measure $D\varphi$ of a function $\varphi\in BV(\mathbb{R}^n)$ does not charge any $\mathcal{H}^{n-1}$-null set, \cite[Lemma 3.76]{AFP}.
Let $D^a\varphi$ denote the part of $D\varphi$ which is absolutely continuous w.r.t. $\mathcal{L}^n$ and $D^s\varphi$ the part that is singular  w.r.t. $\mathcal{L}^n$. Clearly $D^a\varphi$ vanishes on $\mathcal{L}^n$-null sets.  The Cantor part $D^c\varphi$ of $D\varphi$ is defined as $D^c\varphi=D^s\varphi|_{\mathbb{R}^n\setminus S_\varphi}$, \cite[Definition 3.91]{AFP}, and in general it can, roughly speaking, charge subsets of any Hausdorff dimension between $n-1$ and $n$. There is an $(n-1)$-dimensional Borel subset $J_\varphi$ of the approximate discontinuity set $S_\varphi$, called the \emph{set of approximate jump points}, \cite[Definition 3.67]{AFP}, and the jump part of the gradient measure is defined to be $D^j\varphi=D^s\varphi|_{J_\varphi}$, \cite[Definition 3.91]{AFP}. The set $S_\varphi\setminus J_\varphi$ is of zero $\mathcal{H}^{n-1}$-measure, \cite[Theorem 3.78]{AFP}. The gradient measure of $\varphi$ admits the unique decomposition
\begin{equation}\label{E:gradientdecomp}
D\varphi=D^a\varphi+D^j\varphi+D^c\varphi.
\end{equation}
A description of the different summands in terms of geometric scaling properties can be found in \cite[Proposition 3.92]{AFP}.

\begin{proposition}\label{P:naive2} Let $\varphi\in BV_{\loc}(\mathbb{R}^n)$ and $s\in (0,1)$.
Suppose that for any compact $K$ there is a constant $M_K>0$ such that 
\begin{equation}\label{E:abscontextreme}
U^{1-s}\mu_{\varphi,K}(x)<M_K
\end{equation}
for $\mu_{\varphi,K}$-a.e. $x\in K$. Then $\varphi$ cannot have jumps and the Cantor part $D^c\varphi$ of $D\varphi$ cannot charge sets of finite $\mathcal{H}^{n-1+s}$-measure. The function $\varphi$ has a unique continuous Borel version, H\"older of order $s$, and any Lebesgue representative coincides with this version on the Lebesgue set. Moreover, (\ref{E:SVp2}) holds for any path $X$, any relatively compact open $\mathcal{U}\supset X([0,T])$ and any $p\in [1,+\infty]$.
\end{proposition}

Condition (\ref{E:abscontextreme}) does not impose any restrictions on $D^a\varphi$.

\begin{proof}
Again by the maximum principle (\ref{E:abscontextreme}) holds for all $x\in\mathbb{R}^n$. Taking into account (\ref{E:gradientdecomp}) we can conclude that if $\varphi$ satisfies (\ref{E:abscontextreme}) for all compact $K$, then $\varphi$ cannot have jumps: Since the set $J_\varphi$ of approximate jump points has Hausdorff dimension $n-1$ it has zero $\mathcal{H}^{n-1+s}$-measure, and as before we can see such sets cannot be charged by $D\varphi$. The statement on H\"older continuity follows using large closed balls in place of $K$, a simple cut-off argument and Corollary \ref{C:justHoelder}.
\end{proof}

\begin{remark}
If $D\varphi$ is absolutely continuous, then the upper $n-1+s$-regularity of $\varphi$ for $s\in (0,1)$ implies the local $s$-H\"older continuity of $\varphi$ on $\mathcal{C}^1$-smooth domains by a variant of Morrey's inequality in Sobolev-Morrey spaces by \cite[Theorem 2.2]{LambertiVespri}.
\end{remark}

\begin{remark}\label{R:Hausdorfflowerbounded}
Given a Borel measure $\nu$ on $\mathbb{R}^n$ the \emph{pointwise lower Hausdorff dimension of $\nu$ at $x\in\mathbb{R}^n$} is defined as 
\[\underline{\dim}_H\nu(x)=\sup\{\gamma\geq 0: \Theta^\ast_\gamma \nu(x)=0\}=\inf\{\gamma\geq 0: \Theta^\ast_\gamma \nu(x)=+\infty\},\]
where $\Theta^\ast_\gamma \nu(x)$ is as in (\ref{E:upperdensity}). Its \emph{lower Hausdorff dimension} is defined as $\underline{\dim}_H \nu=\essinf_{x\in\mathbb{R}^n}  \underline{\dim}_H\nu(x)$, and it is well known that $\underline{\dim}_H \nu=\inf\{\dim_H A: \text{$A\subset\mathbb{R}^n$ Borel and $\nu(A)>0$}\}$. See for instance \cite{BarreiraPesinSchmeling, MattilaMoranRey}. Conditions (\ref{E:Xextreme}) and (\ref{E:abscontextreme}) imply that $\underline{\dim}_H \mu_X^{[0,T]}\geq n-1+s$ and, respectively, $\underline{\dim}_H \left\| D\varphi\right\|\geq n-1+s$.
\end{remark}

The following is well known.
\begin{proposition}\label{P:samething}
Let $\nu$ be a finite Borel measure on $\mathbb{R}^n$ with compact support. If with some $M>0$ we have $\int_{\mathbb{R}^n}|x-y|^{-s}\nu(dy)<M$ for all $x\in\supp\nu$, then there is some $c>0$ such that $\nu(B(x,r))\leq cr^s$ for all $x\in\supp\nu$ and $r>0$. If there is some $r_0>0$ such that $\nu(B(x,r))\leq cr^d$ for all $x\in\supp\nu$ and $0<r<r_0$, then for any $s<d$ there is some $M>0$ such that $\int_{\mathbb{R}^n}|x-y|^{-s}\nu(dy)<M$ for all $x\in\supp\nu$.
\end{proposition}
We briefly sketch the folklore proof for the convenience of the reader.
\begin{proof}
The first statement is immediate from (\ref{E:trivialbound}). To see the second claim note that since $\nu$ is finite, we can readjust $c$ to obtain $\nu(B(x,r))\leq c\:r^d$ for all $x\in \supp\nu$ and $r>0$.
If $R>0$ is such that $B(0,R/2)$ contains $\supp \nu$, then a classical argument, \cite[p. 109]{Mattila}, shows that for any $x\in \supp\nu$ we have
\[\int_{\mathbb{R}^n}|x-y|^{-s}\nu(dy)=s\int_0^{2R}r^{-s-1}\nu(B(x,r))dr\leq cs\int_0^{2R} r^{d-s}dr.\]
\end{proof}

\begin{corollary}\label{C:jointuppereg}\mbox{}
\begin{enumerate}
\item[(i)] If $X$ is upper $d$-regular with $d>n-1+s$, then (\ref{E:Xextreme}) holds.
\item[(ii)] If $\varphi$ is upper $d$-regular with $d>n-1+s$, then (\ref{E:abscontextreme}) holds for any compact $K$. 
\end{enumerate}
\end{corollary}

For certain paths $X$ the occupation measure $\mu_X^{[0,T]}$ of $X$ is absolutely continuous. To the Radon--Nikodym derivative 
$\frac{d\mu_X^{[0,T]}}{d\mathcal{L}^n} = L_T^X$ one refers as \emph{the local times} of $X$.

\begin{proposition}
\label{P:Xupper-localtime}
Suppose that $\mu_X^{[0,T]}$ of $X$ is absolutely continuous with density $L_T^X \in L^p(\mathbb{R}^n)$ for some $p\in[1,\infty]$. Then $X$ is upper $\frac{n}{q}$-regular, where $\frac{1}{p}+\frac{1}{q}=1$ with the agreement $\frac{1}{+\infty}=0$. In particular, if the local times are bounded, $L_T^X \in L^{\infty}(\mathbb{R}^n)$, then $X$ is upper $n$-regular. 
\end{proposition}
\begin{proof}
We have 
$$
\mu_X^{[0,T]}(B(x,r)) = \int_{B(x,r)} L_T^X(y)dy 
\leq \Vert L_T^X\Vert_{L^p(\mathbb{R}^n)} \mathcal{L}^n(B(x,r))^{\frac{1}{q}} \leq c\: \Vert L_T^X\Vert_{L^p(\mathbb{R}^n)} r^{\frac{n}{q}}.
$$
\end{proof}

\begin{remark}
Existence and regularity of local times are well-studied in the case of Gaussian processes, see e.g. \cite{Ayache, Berman69,Berman70, GemanHorowitz}. A key property is the so-called local non-determinism which guarantees that, roughly speaking, the increments of the process are not too degenerate. 
\end{remark}

In some situations it may be useful to localize the conditions on $\varphi$ and $X$.
\begin{corollary}\label{C:localizeB}
Let $\varphi\in BV_{\loc}(\mathbb{R}^n)$, let $X$ be a path and $s\in (0,1)$. If $B\subset \mathbb{R}^n$ is a Borel set such that (\ref{E:Xextreme}) holds for $\mu_X^{[0,T]}$-a.e. $x\in B$ and (\ref{E:abscontextreme}) holds for $\mu_{\varphi,\mathcal{U}}$-a.e. $x\in B^c$, then (\ref{E:SVviaoccu}) holds. 
\end{corollary}

Of course this matters only if $B\cap X([0,T])\neq \emptyset$ and $B^c\cap X([0,T])\neq \emptyset$ and similarly for $\supp D\varphi$. In the situation of Corollary \ref{C:localizeB} the Hausdorff dimension of $X([0,T])\cap B$ must be greater or equal to $n-1+s$, but on $B^c$ the path $X$ may be arbitrary. On the other hand, $\varphi$ may be arbitrary on $B$, in particular, it may have jumps there, but it cannot jump on $B^c$.

\begin{proof}
The claim follows from 
\begin{align*}
I^{1-s}(\mu_{\varphi,\mathcal{U}},\mu_X^{[0,T]})=&I^{1-s}(\mu_{\varphi,\mathcal{U}}|_B,\mu_X^{[0,T]}|_B)+I^{1-s}(\mu_{\varphi,\mathcal{U}}|_B,\mu_X^{[0,T]}|_{B^c})\\
&+ I^{1-s}(\mu_{\varphi,\mathcal{U}}|_{B^c},\mu_X^{[0,T]}|_B) + I^{1-s}(\mu_{\varphi,\mathcal{U}}|_{B^c},\mu_X^{[0,T]}|_{B^c})
\end{align*}
and the maximum principle for the restrictions $\mu_{\varphi,\mathcal{U}}|_B$ etc. together with the preceding considerations.
\end{proof}

\begin{example}
Suppose that $\varphi_{\mathcal{C}}:\mathbb{R}^2\to \mathbb{R}$ is as in Example \ref{Ex:Cantor} for $n=2$. If $X$ is a smooth unit speed curve outside $\mathbb{R}\setminus [0,1] \times \mathbb{R}$ and a typical Brownian bridge path inside $[0,1]\times\mathbb{R}$, then $X$ is $(s,1)$-variable w.r.t. $\varphi_{\mathcal{C}}$. The same is true if $\varphi=\mathbf{1}_\mathcal{O}$ and $\mathcal{O} \subset [0,1]\times \mathbb{R}$ is a set of finite perimeter. 
\end{example}

We discuss details of Example \ref{Ex:Cantor} and \ref{Ex:jump}.
\begin{example}\label{ex:detex}
Let $\varphi_{\mathcal{C}}:\mathbb{R}^n\to \mathbb{R}$ be as in Example \ref{Ex:Cantor}. If $s\in (0,d_{\mathcal{C}})$, then the fact that 
$\left\|D\varphi_{\mathcal{C}}\right\|(B(x,r))\leq cr^{n-1+d_{\mathcal{C}}}$ together with Corollary \ref{C:jointuppereg} (ii) shows that any path $X$ is in $V(\varphi_{\mathcal{C}},s,\infty)$. From now on assume that $s\in (d_\mathcal{C},1)$. For a constant path $X\equiv x$ with fixed $x\in\mathbb{R}^n$ we have 
\[\mu_X^{[0,T]}=T\delta_x,\]
where $\delta_x$ is the point mass probability measure at $x$. If $x=(\frac12,0,...,0)$, then $x$ has distance $\frac16$ to $\supp\left\|D\varphi_{\mathcal{C}}\right\|$, and hence $U^{1-s}\left\|D\varphi_{\mathcal{C}}\right\|(x)$ is bounded and $X\in V(\varphi_{\mathcal{C}},s,\infty)$. On the other hand, if $x=(0,...,0)$, then 
\[I^{1-s}(\mu_{\varphi_{\mathcal{C}},\mathcal{U}},\mu_X^{[0,T]})=cT\int_\mathcal{U} |y|^{-n+1-s} \left\|D\varphi_{\mathcal{C}}\right\|(dy)=+\infty\]
for any open neighborhood $\mathcal{U}$ of the origin, as follows from $\left\|D\varphi_{\mathcal{C}}\right\|(B(y,r))\geq cr^{n-1+d_\mathcal{C}}$ and standard arguments (\cite[p. 109]{Mattila}). Consequently $X$ is not in $V(\varphi_{\mathcal{C}},s)$ in this case. For a non-constant smooth function $X$ on $\mathbb{R}$
\[\mu_X^{[0,T]}((x-r,x+r))\leq cr\] 
implies that $X\in V(\varphi_{\mathcal{C}},s,\infty)$, in the case $n=1$. For $n=2$ we have $\left\|D\varphi_{\mathcal{C}}\right\|(B(y,r))\geq cr^{1+d_\mathcal{C}}$, so that, again by standard arguments, $U^{1-s}\left\|D\varphi_{\mathcal{C}}\right\|(y)=+\infty$ for any $y\in \{\frac13\}\times [0,1]$.
For any $M>0$ and any such $y$ we can find an open neighborhood $U_y\subset \mathbb{R}^2$ of $y$ on which we have $U^{1-s}\left\|D\varphi_{\mathcal{C}}\right\|>M$, \cite[Chapter I, Theorem 1.3]{Landkof}, and consequently this must hold on all of $\{\frac13\}\times [0,1]$. If $T=1$ and $X$ is the unit speed motion along $X([0,1])=\{\frac13\}\times [0,1]$, then 
\[\mu_X^{[0,1]}=\mathcal{H}^1|_{X([0,1])},\] 
\cite[Theorem 3.2.6]{Federer}, and $I^{1-s}(\left\|D\varphi_{\mathcal{C}}\right\||_{X([0,1]},\mu_X^{[0,1]})=+\infty$. Hence $X$ is not in $V(\varphi_{\mathcal{C}},s)$.

\end{example}

\begin{example}\label{ex:detexjump} 
Let $\mathcal{O}\subset \mathbb{R}^n$ be as in Example \ref{Ex:jump} and $s\in (0,1)$. Suppose $X$ is a unit speed smooth curve that hits $\partial \mathcal{O}$ finitely often. By the additivity of the occupation measures w.r.t. the time interval we may assume all hitting points are different. By the smoothness of $\partial \mathcal{O}$ and $X$ it suffices to show that we have $I^{1-s}(\mu,\nu)<+\infty$, where $\mu=\mathcal{H}^1|_{\{(0,..., 0,x_n):\:|x_n|<1\}}$ and $\nu=\mathcal{H}^{n-1}|_{\{(x_1,...,x_{n-1},0):\:|(x_1,...,x_{n-1})|<1\}}$. For $n\geq 2$ and $\varepsilon>0$ we have $\int_{-1}^1 (\varepsilon^2+|t|)^{(-n+1-s)/2}dt\leq c\varepsilon^{-n+3-s}$. Therefore
\[\int_{B(0,2^{-k})\setminus B(0,2^{-k-1})}\int_{\{(0,..., 0,y_n):\:|y_n|<1\}}|x-y|^{-n+1-s}\mu(dy)\nu(dx)\leq c\:2^{-k(2-s)},\]
which is summable over $k=0,1,2,...$ For $n=1$ the integral over $t\in [-1,1]$ is even bounded. It follows that the mutual energy of $\mu$ and $\nu$ is finite, and hence for such $\mathcal{O}$, $X$ and $s$ we have $X\in V(\mathbf{1}_\mathcal{O},s)$. If $n=2$ and $X$ spends $\mathcal{L}^1$-positive time in $\partial\mathcal{O}$, then we can find $0\leq a<b\leq T$ such that $X([a,b])\subset \partial \mathcal{O}$ and $\mu_X^{[a,b]}=\mathcal{H}^1|_{[a,b]}$. Since $\left\|D\mathbf{1}_\mathcal{O}\right\|$ equals a constant times $\mathcal{H}^{n-1}|_\mathcal{O}$,
it follows that 
\[
I^{1-s}(\left\|D\varphi\right\|_{X([0,T])},\mu_X^{[0,T]})
\geq c\int_{X([a,b])}\int_{X([a,b])}|x-y|^{-n+1-s}\mathcal{H}^1(dx)\mathcal{H}^{n-1}(dy)=+\infty,
\]
so that $X$ is not in $V(\mathbf{1}_\mathcal{O},s)$.
\end{example}

\subsection{Fourier transform and trading of regularity}

Condition (\ref{E:finiteenergypath}) requires that the Fourier transform 
\[\hat{\mu}_X^{[0,T]}(\xi)=\int_{\mathbb{R}^n} e^{-i\left\langle\xi,x\right\rangle}\mu_X^{[0,T]}(dx),\quad\xi\in\mathbb{R}^n\]
of the occupation measure $\mu_X^{[0,T]}$ is square integrable w.r.t. $|\xi|^{s-1}\:d\xi$, note that
\begin{equation}\label{E:Fourierint}
I^{1-s}\mu_X^{[0,T]}=c(s,n)\int_{\mathbb{R}^n} |\xi|^{s-1}|\hat{\mu}_X^{[0,T]}(\xi)|^2\:d\xi 
\end{equation}
by \cite[Lemma 12.12]{Mattila}.  Increased activity of $X$ allows larger $s\in (0,1)$, see \cite[Chapter 17, p. 250]{Kahane} for this classical discussion. Since $\hat{\mu}_X^{[0,T]}$ is bounded, integrability w.r.t. $|\xi|^{s-1}\:d\xi$ already implies square integrability, giving (\ref{E:finiteenergypath}). (Alternatively, a Riemann-Lebesgue argument shows that integrability implies (\ref{E:Xextreme}), which in turn implies (\ref{E:finiteenergypath}).) The integral in (\ref{E:SVviaoccu}) is a polarized version of (\ref{E:Fourierint}),
\begin{equation}\label{E:polarizedversion}
I^{1-s}(\mu_{\varphi,\mathcal{U}},\mu_X^{[0,T]})=c(s,n)\operatorname{Re}\:\int_{\mathbb{R}^n}|\xi|^{s-1}\:\hat{\mu}_{\varphi,\mathcal{U}}(\xi)\:\overline{\hat{\mu}_X^{[0,T]}}(\xi)\:d\xi,
\end{equation}
here $\hat{\mu}_{\varphi,\mathcal{U}}$ denotes the Fourier transform of the measure $\mu_{\varphi,\mathcal{U}}$.

Formulas (\ref{E:SVviaoccu}) and (\ref{E:polarizedversion}) suggest a trading of regularity between the measures.

\begin{proposition}\label{P:shiftovergamma} Let $\varphi\in BV_{\loc}(\mathbb{R}^n)$, $X:[0,T]\to \mathbb{R}^n$ and $s\in (0,1)$.
For any $\gamma\in (0, 1-s)$ we have
\begin{equation}\label{E:shiftovergamma}
\int_{\mathbb{R}^n}U^{1-s}\mu_{\varphi,\mathcal{U}}(x)\mu_X^{[0,T]}(dx)=\int_{\mathbb{R}^n}U^{\gamma}\mu_{\varphi,\mathcal{U}}(x)U^{1-s-\gamma} \mu_X^{[0,T]}(x)dx.
\end{equation}
\end{proposition}
\begin{proof}
Proposition \ref{P:convolution} (ii), applied with $\nu_1=\mu_{\varphi,\mathcal{U}}$, $\nu_2=\mu_X^{[0,T]}$,  $\gamma_1=\gamma$ and $\gamma_2=1-s-\gamma$, yields the result.
\end{proof}

Since it only allows to trade the smoothness of the measures in $n-(1-s)$ and $n$ with pivotal level at $n-(1-s)/2$, formula (\ref{E:shiftovergamma}) and upper regularity in the form considered above do not improve the statements in Corollary \ref{C:localizeB}. However, if a coefficient $\varphi$ is fixed and $\left\|D\varphi\right\|$ satisfies a certain moment condition, a `weighted' upper regularity condition can ensure $(s,1)$-variability.

\begin{corollary}\label{C:altcond}
Let $\varphi\in BV_{\loc}(\mathbb{R}^n)$, let $X$ be a path and $s\in (0,1)$. Suppose that there are $\varepsilon\in (0,1-s)$, $c>0$, $x_1,...,x_N\in\mathbb{R}^n$ and $\delta_1,...,\delta_N \in (0,n-1+s+\varepsilon)$ such that we have 
\begin{equation}\label{E:LauriSolenN}
\int_0^T |X_t-x|^{-n+\varepsilon}dt<c\sum_{k=1}^N|x-x_k|^{-n+\delta_k},\quad x\in\mathbb{R}^n,
\end{equation}
and 
\begin{equation}\label{E:momentconditionN}
\sum_{k=1}^N\int_{\mathbb{R}^n}|x-x_k|^{-n+1-s-\varepsilon+\delta_k}\left\|D\varphi\right\|(dx)<+\infty.
\end{equation}
Then $X$ is $(s,1)$-variable w.r.t. $\varphi$.
\end{corollary}

For $N=1$ conditions (\ref{E:LauriSolenN}) and (\ref{E:momentconditionN}) (applied to the components of the coefficient) are exactly (\ref{E:LauriSolen}) and (\ref{E:momentconditionx0}). In Corollary \ref{C:altcondrandom} below the mechanism of Corollary \ref{C:altcond} is used efficiently in a probabilistic context.

\begin{proof}
Under (\ref{E:LauriSolenN}) and (\ref{E:momentconditionN}) we have, by (\ref{E:shiftovergamma}) and Proposition \ref{P:convolution} (i),
\begin{multline}
I^{1-s}(\left\|D\varphi\right\|,\mu_X^{[0,T]})=\int_{\mathbb{R}^n}U^{\varepsilon}\mu_X^{[0,T]}(x) U^{1-s-\varepsilon}\left\|D\varphi\right\|(x)dx\leq c\sum_{k=1}^N \:\int_{\mathbb{R}^n} |x-x_k|^{-n+\delta_k}U^{1-s-\varepsilon}\left\|D\varphi\right\|(x)dx\notag\\
=c\sum_{k=1}^N \:\int_{\mathbb{R}^n}\int_{\mathbb{R}^n}|x-x_k|^{-n+\delta_k}|x-y|^{-n+1-s-\varepsilon}dx \:\left\|D\varphi\right\|(dy)= c\sum_{k=1}^N \:\int_{\mathbb{R}^n} |y-x_k|^{-n+1-s-\varepsilon+\delta_k}\left\|D\varphi\right\|(dy).\notag
\end{multline}
\end{proof}

\subsection{Probabilistic examples}\label{SS:probab}

We connect to probabilistic examples. A condition of type (\ref{E:LauriSole}) can ensure variability w.r.t. any $BV$-function.

\begin{corollary}\label{C:upperregrandom}
Let $s\in (0,1)$ and suppose that $Y=(Y_t)_{t\in [0,T]}$ is an $\mathbb{R}^n$-valued stochastic process with $Y_0=0$ on a probability space $(\Omega,\mathcal{F},\mathbb{P})$. If 
\begin{equation}\label{E:crudebound}
\sup_{x\in\mathbb{R}^n} \mathbb{E}\int_0^T|Y_t-x|^{-n+1-s}dt<+\infty,
\end{equation}
then we can find an event $\Omega_0\in \mathcal{F}$ with $\mathbb{P}(\Omega_0)=1$ such that for all $\omega \in \Omega_0$ the path $Y(\omega)$ is $(s,1)$-variable w.r.t. any $\varphi\in BV(\mathbb{R}^n)$.
\end{corollary}
\begin{proof}
Using Fubini's theorem and the above bound,
\[\mathbb{E}\int_{\mathbb{R}^n}\int_0^T|Y_t-x|^{-n+1-s}dt\left\|D\varphi\right\|(dx)\leq \left\|D\varphi\right\|(\mathbb{R}^n)\sup_{x\in\mathbb{R}^n} \mathbb{E}\int_0^T|Y_t-x|^{-n+1-s}dt<+\infty.\]
\end{proof}

\begin{example}
Let $B^H=(B_t^H)_{t\geq 0}$ be an $n$-dimensional fractional Brownian motion with Hurst index $0<H<1$ over a probability space $(\Omega,\mathcal{F},\mathbb{P})$. Let $s\in (0,1)$. We claim that if 
\begin{equation}\label{E:fBmcond}
n-1+s<\frac{1}{H},
\end{equation}
then there is a constant $c(n,s,H)>0$ such that for all $x\in\mathbb{R}^n$ we have 
\begin{equation}\label{E:finiteness}
\mathbb{E}\int_0^T|B_t^H-x|^{-n+1-s}dt<c(n,s,H).
\end{equation}
For Brownian motion - the special case $H=\frac12$ -  we have (\ref{E:fBmcond}) in dimensions $n=1$ and $n=2$ for any $s\in (0,1)$. For $n=3$ a fractional Brownian motion with Hurst index $H<\frac12$ satisfies (\ref{E:fBmcond}) for $0<s<\frac{1}{H}-2$. A higher dimension $n$ of the ambient space forces a higher irregularity of the path in order to have an occupation measure supported on a set of sufficiently small codimension. The arguments for (\ref{E:finiteness}) are rather standard, \cite{Falconer, Kahane}. The expectation equals
\begin{align}
\frac{1}{(2\pi)^{n/2}}&\int_0^Tt^{-nH}\int_{\mathbb{R}^n} \exp\left(-\frac{|y|^2}{2t^{2H}}\right) |x-y|^{-n+1-s}dy\:dt\notag\\
&=\frac{1}{(2\pi)^{n/2}}\int_0^Tt^{-nH}\int_{|\xi|>t^H} \exp\left(-\frac{|x-\xi|^2}{2t^{2H}}\right) |\xi|^{-n+1-s}d\xi\:dt\notag\\
&\ +\frac{1}{(2\pi)^{n/2}}\int_0^Tt^{-nH}\int_{|\xi|\leq t^H} \exp\left(-\frac{|x-\xi|^2}{2t^{2H}}\right) |\xi|^{-n+1-s}d\xi\:dt.\notag
\end{align}
The first summand is bounded by 
\begin{align}
\frac{1}{(2\pi)^{n/2}}&\int_0^T t^{-2nH+(1-s)H}\int_{\mathbb{R}^n} \exp\left(-\frac{|x-\xi|^2}{2t^{2H}}\right) d\xi\:dt\notag\\
&=\frac{1}{(2\pi)^{n/2}}\int_0^T t^{-2nH+(1-s)H}\int_{\mathbb{R}^n} \exp\left(-\frac{|z|^2}{2t^{2H}}\right)dzdt\notag\\
&=\int_0^T t^{-(n-1+s)H}dt.\notag
\end{align}
The second summand does not exceed
\[\frac{1}{(2\pi)^{n/2}}\int_0^T t^{-nH}\int_{|\xi|\leq t^H} |\xi|^{-n+1-s}d\xi dt =\frac{1}{(2\pi)^{n/2}}\int_0^T t^{-(n-1+s)H}dt.\]
\end{example}

The above idea can be generalized to link $(s,1)$-variability into  upper regularity of the underlying probability measures. 
\begin{corollary}
\label{C:probability-upperregularity}
Let $Y=(Y_t)_{t\in [0,T]}$ be an $\mathbb{R}^n$-valued stochastic process with $Y_0=0$ on a probability space $(\Omega,\mathcal{F},\mathbb{P})$. For any $t\in[0,T]$, let $\nu_t:=P_{Y_t}$ denote the law of $Y_t$. Suppose that $s\in (0,1)$ and that there exists some 
$g \in L^1(0,T)$ such that for $\mathcal{L}^1$-a.e. $t\in (0,T)$ we have 
\begin{equation}
\label{E:probability-upperregular}
\nu_t(B(x,r)) \leq g(t)r^{d},\quad x\in\mathbb{R}^n,\quad r>0,
\end{equation}
with some $d>n-1+s$. Then (\ref{E:crudebound}) holds.
\end{corollary}
\begin{proof}
By Fubini's theorem
\[\mathbb{E}\int_0^T |Y_t-x|^{-n+1-s}dt = \int_0^T \int_{\mathbb{R}^n}|y-x|^{-n+1-s}\nu_t(dy)dt,\]
and as in Proposition \ref{P:samething} it follows that for $\mathcal{L}^1$-a.e. $t\in (0,T)$ the inner integral is bounded by
\[(n-1+s)\int_0^\infty r^{-n-s}\nu_t(B(x,r))dr\leq 1+(n-1+s)\:g(t)\:\int_0^1 r^{-n-s+d}dr<+\infty.\]
\end{proof}

A probabilistic version of (\ref{E:LauriSolenN}) gives similar results.
\begin{corollary}\label{C:altcondrandom}
Let $\varphi\in BV_{\loc}(\mathbb{R}^n)$ and $s\in (0,1)$. Suppose that $Y=(Y_t)_{t\in [0,T]}$ is an $\mathbb{R}^n$-valued stochastic process with $Y_0=0$ on a probability space $(\Omega,\mathcal{F},\mathbb{P})$. Suppose that there are $\varepsilon\in (0,1-s)$, $c>0$, $x_1,...,x_N\in\mathbb{R}^n$ and $\delta_1,...,\delta_N \in (0,n-1+s+\varepsilon)$ such that we have 
\begin{equation}\label{E:LauriSolenrandomN}
\mathbb{E}\int_0^T |Y_t-x|^{-n+\varepsilon}dt<c\sum_{k=1}^N|x-x_k|^{-n+\delta_k},\quad x\in\mathbb{R}^n,
\end{equation}
and (\ref{E:momentconditionN}) holds. Then we can find an event $\Omega_0\in \mathcal{F}$ with $\mathbb{P}(\Omega_0)=1$ such that for all $\omega \in \Omega_0$ the path $Y(\omega)$ is $(s,1)$-variable w.r.t. $\varphi$.
\end{corollary}
\begin{proof}
Under condition (\ref{E:LauriSolenrandomN}) identity (\ref{E:shiftovergamma}) and Fubini's theorem give
\[\mathbb{E}I^{1-s}(\left\|D\varphi\right\|,\mu_Y^{[0,T]})=\int_{\mathbb{R}^n} \mathbb{E}U^{\varepsilon}\mu_Y^{[0,T]}(x) U^{1-s-\varepsilon}\left\|D\varphi\right\|(x)dx\leq c\sum_{k=1}^N \:\int_{\mathbb{R}^n} |x-x_k|^{-n+\delta_k}U^{1-s-\varepsilon}\left\|D\varphi\right\|(x)dx,\]
which by (\ref{E:momentconditionN}) and the arguments in the proof of Corollary \ref{C:altcond} is seen to be finite. 
\end{proof}

\begin{example}\label{Ex:fBMexample}
For $n\geq 2$, the fractional Brownian motion with $H>\frac{1}{n}$ satisfies (\ref{E:LauriSolenrandomN}) with $N=1$, any $0<\varepsilon<n-\frac{1}{H}$ and $\delta_1=\frac{1}{H}+\varepsilon$: We have
\[\mathbb{E}\int_0^T|Y_t-x|^{-n+\varepsilon}dt=\frac{1}{\sqrt{2\pi}}\int_{\mathbb{R}^n}|y-x|^{-n+\varepsilon}\int_0^Tt^{-nH}\exp\left(-\frac{|y|^2}{2t^{2H}}\right)dt\:dy.\]
Substituting $u=\frac{1}{2}|y|^2 t^{-2H}$ we see that for any $y\in\mathbb{R}^n\setminus \{0\}$ the inner integral equals
\[\frac{2^{n/2}}{2^{1+\frac{1}{2H}}H}\:|y|^{\frac{1}{H}-n}\int_{\frac{1}{2}|y|^2T^{-2H}} u^{\frac{n}{2}-\frac{1}{2H}-1}\exp(-u)du\leq c(n,H)|y|^{\frac{1}{H}-n}.\]
Using the convolution identity for Riesz kernels,
\[\mathbb{E}\int_0^T|Y_t-x|^{-n+\varepsilon}dt\leq c(n,H)\int_{\mathbb{R}^n}|y-x|^{-n+\varepsilon} |y|^{\frac{1}{H}-n}dy=c(n,H)|x|^{\frac{1}{H}+\varepsilon-n}.\]
\end{example}

The following generalization of this example is immediate.

\begin{corollary}
Let $Y=(Y_t)_{t\in [0,T]}$ be an $\mathbb{R}^n$-valued stochastic process with $Y_0=0$ on a probability space $(\Omega,\mathcal{F},\mathbb{P})$. For any $t\in[0,T]$, let $\nu_t:=P_{Y_t}$ denote the law of $Y_t$. Suppose that $s\in (0,1)$, $\varepsilon$, $c$, $x_k$ and  $\delta_k$ are as in Corollary \ref{C:altcondrandom}, and that there exists some $g \in L^1(0,T)$ such that for $\mathcal{L}^1$-a.e. $t\in (0,T)$ we have 
\[\int_{\mathbb{R}^n}|y-x|^{-n+\varepsilon}\nu_t(dy)\leq g(t)\:\sum_{k=1}^N|x-x_k|^{-n+\delta_k},\quad x\in\mathbb{R}^n.\]
Then (\ref{E:LauriSolenrandom}) holds.
\end{corollary}

\begin{remark}
Corollary \ref{C:altcondrandom} can also be applied to Gaussian bridges: Suppose that $n\geq 2$, let $B^H$ be a fractional Brownian motion with $H>\frac{1}{n}$, and let $Y$ be the Gaussian bridge obtained by conditioning to have $B^H_T = \mathbf{1}$, where $\mathbf{1}=(1,...,1)\in\mathbb{R}^n$. Then using \cite[Theorem 3.1]{Sottinen-Yazigi} it can be shown that
$$
\mathbb{E}\int_0^T |Y_t-x|^{-n+\varepsilon}dt<c\left[|x|^{\frac{1}{H}+\varepsilon-n} + |x-1|^{\frac{1}{H}+\varepsilon-n}\right],\quad x\in\mathbb{R}^n.
$$
Similar arguments can be applied also to more general bridges conditioned, e.g., on several time marginals, cf. \cite[Remark 3.6(ii)]{Sottinen-Yazigi}.
\end{remark}

\subsection{A brief discussion of variability versus irregularity}\label{SS:Gubinelli}

In \cite{CatellierGubinelli} the authors studied ODEs of type $dx_t=b(t,x_t)dt+dw_t$, where $b$ is an irregular vector field and $w$ is a continuous fast moving perturbation. Although our goal and  methods are different from theirs, our point of view upon irregular paths and occupation measures follow a similar spirit. In cases where $b=b(x)$ is a bounded continuous function and $w$ is `active enough' they can prove existence and uniqueness of a continuous solution $x$, \cite[Theorem 1.9]{CatellierGubinelli}. If $b$ is a distribution only, $b(x_{\cdot})$ must be defined appropriately, and in \cite{CatellierGubinelli} this was done for paths $x$ that differ from the sufficiently fast moving perturbation $w$ only by a H\"older signal, \cite[Definition 1.10 and Theorem 1.11]{CatellierGubinelli}. The needed activity of  $w$ is encoded in the boundedness and (temporal) H\"older continuity in a certain (spatial) H\"older norm of the image 
\[T_t^w b(x)=\int_0^tb(x+w_u)du\]
of $b$ under an operator $T^w$ taking the average of $b$ along $w$. The more active $w$, the stronger is the averaging effect of this operator, and under suitable quantitative assumptions the authors then obtain existence, uniqueness, and flow properties for solutions $x$ even if $b$ is only a distribution, \cite[Theorems 1.12, 1.13, 1.14]{CatellierGubinelli}. To quantify how `fast moving' a path $w\in \mathcal{C}([0,T],\mathbb{R}^n)$ must be they look at the Fourier transform $\hat{\mu}_w^{[0,t]}(\xi)$
of the occupation measure $\mu_w^{[0,t]}$ of $w$ up to time $t$ and, given $\varrho>0$ and $\gamma>0$, say that $w$ is \emph{$(\varrho,\gamma)$-irregular} if
\begin{equation}\label{E:CG}
\sup_{a\in\mathbb{R}^n} \sup_{0\leq u<t\leq T}(1+|a|)^\varrho\:\frac{|\hat{\mu}_w^{[0,u]}(a)-\hat{\mu}_w^{[0,t]}(a)|}{|u-t|^\gamma}<+\infty,
\end{equation}
\cite[Definition 1.3]{CatellierGubinelli}. If a path $w$ satisfies (\ref{E:CG}), then the averaging operator $T^w$ is bounded between certain Fourier-Lebesgue (or H\"older) spaces, \cite[Corollary 1.5, Theorem 1.6, Theorem 1.7]{CatellierGubinelli}. Increased activity of $w$ implies higher regularity (diffusivity) of its occupation measure, and hence better decay of its Fourier transform (larger $\varrho$), which encodes a stronger regularization effect of $T^w$. A refined and very systematical analysis of $(\varrho,\gamma)$-irregularity is provided in \cite{GG20b}.

Condition (\ref{E:CG}) is a condition for single paths, and it is later connected to the individual coefficient $b$ via the mapping properties of $T^w$. In contrast, (\ref{eq:SVp-condition}) is a condition on $X$ relative to a given $BV$-function $\varphi$. In (\ref{E:Fourierint}) and (\ref{E:polarizedversion}) the interval $[0,T]$ is fixed, and the decay of the Fourier transforms at infinity is quantified in terms of integrability properties. It might be interesting to investigate quantities that `interpolate' between (\ref{eq:SVp-condition}) and (\ref{E:CG}), for instance a version of (\ref{E:polarizedversion}) that incorporates time dependencies. It might also be interesting to see whether a concept of variability relative to both a low regularity diffusion coefficient $\sigma$ and a low regularity drift vector field $b$ could be useful.

\subsection{Compositions of $BV_{\loc}$-functions and H\"older paths}\label{SS:compositions}

Proposition \ref{P:key} allows a multidimensional version of \cite[Proposition 4.6]{CLV16}, which ensures that compositions $\varphi\circ X$ are elements of $W^{\beta,p}(0,T)$. We first provide a bound for the Gagliardo seminorm part in the norm $\Vert \varphi \circ X\Vert_{W^{\beta,p}(0,T)}$.

\begin{proposition}\label{P:basicest}
Let $\varphi\in BV_{\loc}(\mathbb{R}^n)$. Let $X:[0,T]\to\mathbb{R}^n$ be a path which is H\"older continuous of order $\alpha\in (0,1]$. Suppose that $s\in (0,1)$, $p\in [1,+\infty)$ and $X\in V(\varphi,s,p)$. Then for any $\beta\in (0,\alpha s)$ there is a constant $c>0$, depending only on $\alpha$, $\beta$, $n$, $p$ and $s$, such that
\begin{equation}\label{E:basicest}
[\varphi\circ X]_{\beta,p}^p\leq c\:[X]_{\alpha,\infty}^{s p}\int_0^T\left[\int_{\mathcal{U}}\frac{\left\|D\varphi\right\|(dy)}{|X_t-y|^{n-1+s}}\right]^{p}dt.
\end{equation} 

If $\varphi$, $X$, $\alpha$ and $s$ are as before but $X\in V(\varphi, s,\infty)$, then for any $\beta\in (0,\alpha s)$ there  is a constant $c>0$, depending only on $\alpha$, $\beta$, $n$ and $s$, such that
\begin{equation}\label{E:basicestinfty}
\esssup_{t\in [0,T]}\int_0^t\frac{|\varphi(X_t)-\varphi(X_\tau)|}{(t-\tau)^{1+\theta}}d\tau\leq c[X]_{\alpha,\infty}^s\esssup_{t\in [0,T]}\int_{\mathcal{U}}\frac{\left\|D\varphi\right\|(dy)}{|X_t-y|^{n-1+s}}.
\end{equation}
\end{proposition}

The proof makes use of maximal functions and some of their basic properties, the necessary definitions and results can be found in Appendix \ref{S:maximalfcts}.

\begin{proof}
By (\ref{E:timenullset}) we have
\begin{equation}\label{E:forgetbadset}
\int_0^T g(X_\tau)d\tau=\int_0^T \mathbf{1}_{\mathbb{R}^n\setminus S_{\varphi}}(X_\tau)g(X_\tau)d\tau
\end{equation}
for any nonnegative Borel function $g:\mathbb{R}^n\to [0,+\infty]$. For any fixed $t\in [0,T]$ we have 
\[\int_0^t\frac{|\varphi(X_t)-\varphi(X_\tau)|^p}{(t-\tau)^{1+\beta p}}\:d\tau=  \int_0^t\frac{|\varphi(X_t)-\varphi(X_\tau)|^p}{(t-\tau)^{1+\beta p}}\:\mathbf{1}_{\mathbb{R}^n\setminus S_{\varphi}}(X_\tau)d\tau\]
by (\ref{E:forgetbadset}), and viewing this as a nonnegative function of $t$, also
\begin{equation}\label{E:cutout}
\int_0^T\int_0^t \frac{|\varphi(X_t)-\varphi(X_\tau)|^p}{(t-\tau)^{1+\beta p}}\:d\tau\:dt
=\int_0^T\int_0^t \frac{|\varphi(X_t)-\varphi(X_\tau)|^p}{(t-\tau)^{1+\beta p}}\:\mathbf{1}_{\mathbb{R}^n\setminus S_{\varphi}}(X_\tau)\mathbf{1}_{\mathbb{R}^n\setminus S_{\varphi}}(X_t)\:d\tau\:dt.
\end{equation}
Let $\mu:=\left\|D\varphi\right\||_{\mathcal{U}}$ denote the restriction of $\left\|D\varphi\right\|$ to $\mathcal{U}$.
By Proposition \ref{P:meanvalue} and the $\alpha$-H\"older continuity of $X$ the right hand side of (\ref{E:cutout}) is seen to be bounded by
\begin{align}
& c(n,s)^p\int_0^T\int_0^t\frac{|X_t-X_\tau|^{sp}}{(t-\tau)^{1+\beta p}}[\mathcal{M}_{1-s,4|X_t-X_\tau|}\mu(X_t)+\mathcal{M}_{1-s,4|X_t-X_\tau|}\mu(X_\tau)]^p\:d\tau\:dt\notag\\
&\leq c(n,s)^p[X]_{\alpha,\infty}^{sp}\int_0^T\int_0^t (t-\tau)^{\alpha s p-\beta p-1}\times\notag\\
&\hspace{100pt} \times[\mathcal{M}_{1-s,4|X_t-X_\tau|}\mu(X_t)+\mathcal{M}_{1-s,4|X_t-X_\tau|}\mu(X_\tau)]^p\:d\tau\:dt;\notag  
\end{align}
here $\mathcal{M}_{1-s,R}\mu$ denotes the fractional maximal function of $\mu$ of order $1-s$ (and truncated at radius $R>0$), see (\ref{E:maximalfct}). The trivial estimate (\ref{E:trivialbound}) implies 
\begin{equation}
\label{eq:maximal-function-bound}
\mathcal{M}_{1-s, 4|X_t-X_s|}\mu(X_t)\leq c\int_{\mathbb{R}^n}\frac{\mu(dy)}{|X_t-y|^{n-1+s}}
\end{equation}
for any $t\in [0,T]$ and with $c>0$ depending only on $n$ and $s$, and therefore 
\begin{align*}
&\int_0^T\int_0^t (t-\tau)^{\alpha s p-\beta p-1}[\mathcal{M}_{1-s,4|X_t-X_\tau|}\mu(X_t)]^p\: d\tau dt\\
\leq & c\int_0^T \int_0^t (t-\tau)^{\alpha s p-\beta p-1} d\tau \left[\int_{\mathbb{R}^n}\frac{\mu(dy)}{|X_t-y|^{n-1+s}}\right]^p dt
\leq  c\:\frac{T^{(\alpha s-\beta)p}}{(\alpha s-\beta)p} \int_0^T\left[\int_{\mathcal{U}}\frac{\left\|Df\right\|(dy)}{|X_t-y|^{n-1+s}}\right]^p dt.
\end{align*}
Using Fubini's theorem we obtain the same upper bound for the summand with $\mu(X_\tau)$ in place of $\mu(X_t)$. Combining the estimates and using the symmetry of the integrand in the Gagliardo seminorm, we arrive at (\ref{E:basicest}). The estimate (\ref{E:basicestinfty}) follows similarly.
\end{proof}

\begin{remark}\label{R:Chen}
In \cite[Proposition 5.0.3 and Remark 5.0.3]{Chen} it has been shown that for discontinuous $\varphi\in BV(\mathbb{R})$ one cannot expect $\varphi(X)$ to have finite $p$-variation for any $p\geq 1$, if $X$ visits a point of discontinuity of $\varphi$ infinitely many times. In particular, $\varphi(X)$ cannot be H\"older continuous of any order in this case.
This motivates to use Sobolev norms and generalized Stieltjes type integrals rather than $p$-variation and Young integrals. 
\end{remark}

If $X\in V(\varphi,s,p)$, a pinning argument shows that $\varphi$ must be in $L^{p}(X)$. This entails that, for appropriate $\beta$ and $p$, the composition $\varphi\circ X$ is in $W^{\beta, p}(0,T)$ and, for large enough $p$, even H\"older continuous.

\begin{lemma}
\label{lemma:Lp-integrability}
Let $\varphi\in BV_{\loc}(\mathbb{R}^n)$. Let $X:[0,T]\to\mathbb{R}^n$ be a path which is H\"older continuous of order $\alpha\in (0,1)$. Suppose that $s\in (0,1)$, $p\in [1,+\infty]$ and $X\in V(\varphi,s,p)$. Then
$\varphi \in L^{p}(X)$ and $\varphi\circ X \in W^{\beta, p}(0,T)$ for any $\beta\in (0,\alpha s)$. Moreover, if $\alpha s>\frac{1}{p}$, then $\varphi\circ X$ has a (unique) Borel version which is H\"older continuous of any order smaller than $\alpha s-\frac{1}{p}$. (We again use the agreement that $\frac{1}{+\infty}:=0$.)
\end{lemma}
\begin{proof}
Suppose first that $p\in [1,+\infty)$. Choose $t_0\in[0,T]$ such that $X_{t_0} \in \mathbb{R}^n\setminus S_\varphi$. By Proposition \ref{P:key} such $t_0$ clearly exists. Then we have 
$$
|\varphi(X_t)|^p \leq 2^{p-1}\left[|\varphi(X_t)-\varphi(X_{t_0})|^p+|\varphi(X_{t_0})|^p\right].
$$
for a.e. $t\in [0,T]$. Using (\ref{eq:maximal-function-bound}) we obtain
\begin{equation}
\begin{split}
&|\varphi(X_t)-\varphi(X_{t_0})|^p \notag\\
&\leq |X_t-X_{t_0}|^{sp}[\mathcal{M}_{1-s,4|X_t-X_{t_0}|}\mu(X_t)+\mathcal{M}_{1-s,4|X_t-X_{t_0}|}\mu(X_{t_0})]^p \\
&\leq c\:2^{p-1}[X]_{\alpha,\infty}^{s p}T^{\alpha sp}\left(\left[\int_{\mathcal{U}}\frac{\left\|D\varphi\right\|(dy)}{|X_t-y|^{n-1+s}}\right]^{p}+\left[\int_{\mathcal{U}}\frac{\left\|D\varphi\right\|(dy)}{|X_{t_0}-y|^{n-1+s}}\right]^{p}\right),
\end{split}
\end{equation}
and integration yields 
\begin{equation}\begin{split}
\int_0^T |(\varphi \circ X)(t)|^p dt
\leq & c\:2^{p-1}T|(\varphi\circ X)(t_0)|^p 
+ c\:4^{p-1}[X]_{\alpha,\infty}^{s p}T^{\alpha sp+1}\left[\int_{\mathcal{U}}\frac{\left\|D\varphi\right\|(dy)}{|X_{t_0}-y|^{n-1+s}}\right]^{p} \\
&+c\:4^{p-1}[X]_{\alpha,\infty}^{s p}T^{\alpha sp} \int_0^T \left[\int_{\mathcal{U}}\frac{\left\|D\varphi\right\|(dy)}{|X_t-y|^{n-1+s}}\right]^{p}dt.
\end{split}
\end{equation}
In order to show that $\varphi \circ X \in L^p(0,T)$ it now suffices to prove that there is some $t_0 \in [0,T]$ such that $X_{t_0} \in \mathbb{R}^n \setminus S_\varphi$,
\begin{equation}
\label{eq:point-bound}
|\varphi(X_{t_0})|^p < +\infty,
\end{equation}
and 
\begin{equation}
\label{eq:seminorm-point-bound}
\int_{\mathcal{U}}\frac{\left\|D\varphi\right\|(dy)}{|X_{t_0}-y|^{n-1+s}} < +\infty.
\end{equation}
Let $N\subset [0,T]$ be the set of all $t$ such that $X_t \in S_\varphi$. As seen before, $N$ is a Lebesgue null set, and 
by Definition \ref{D:comp} and \eqref{E:admissible} we have $\widetilde{\varphi}(X_t) = \lambda_\varphi(X_t) < \infty$
for all $t\in[0,T]\setminus N$ since $X_t \in \mathbb{R}^n \setminus S_\varphi$ for such $t$. By $(s,p)$-variability, \eqref{eq:seminorm-point-bound} must hold for all $t\in [0,T]\setminus N'$, where $N'$ is a Lebesgue null set. Thus both 
\eqref{eq:point-bound} and \eqref{eq:seminorm-point-bound} must hold for any $t_0\in [0,T]\setminus (N\cup N')$, i.e. for a.e. $t_0\in [0,T]$. In the case that $p=+\infty$ the $(s,\infty)$-variability of $X$ guarantees \eqref{eq:seminorm-point-bound} for all $t_0\in [0,T]$, so that the desired result follows by obvious modifications of the above arguments. The membership in  $W^{\beta,p}(0,T)$ is seen using Proposition \ref{P:basicest} and the statement on continuity using Sobolev embedding, \cite[Theorem 8.2]{DiNezza}.
\end{proof}

\begin{remark}
By Remark \ref{R:Chen} and Lemma \ref{lemma:Lp-integrability} one cannot expect $X$ to be $(s,p)$-variable with respect to $\varphi$ for $p> \frac{1}{\alpha s}$ if $X$ visits discontinuity points of $\varphi$ infinitely often. 
\end{remark}

In the rest of this subsection we derive an estimate for the weighted $L^p$-term in the norm $\Vert \varphi \circ X\Vert_{W_0^{\beta,p}(0,T)}$. Only the special case $p=1$ will be used later on.

\begin{proposition}
\label{prop:dyda-bound}
Let $\varphi\in BV_{\loc}(\mathbb{R}^n)$. Let $X:[0,T]\to\mathbb{R}^n$ be a path which is H\"older continuous of order $\alpha\in (0,1]$ and $X\in V(\varphi, s,p)$ for some $s\in (0,1)$ and $p\in [1,+\infty)$. Then for any $\beta\in (0,\alpha s \wedge \frac{1}{p})$ there is a constant $c>0$, depending only on $\alpha$, $\beta$, $n$, $p$ and $s$, such that
\begin{equation}\label{E:dyda-bound}
\int_0^T \frac{|(\varphi \circ X)(t)|^p}{t^{\beta p}} dt
\leq c\bigg([X]_{\alpha,\infty}^{s p}\int_0^T\left[\int_{\mathcal{U}}\frac{\left\|D\varphi\right\|(dy)}{|X_t-y|^{n-1+s}}\right]^{p}dt + \int_0^T |(\varphi \circ X)(t)|^p dt\bigg).
\end{equation} 
\end{proposition}

To prove Proposition \ref{prop:dyda-bound} we make use of the following fact.

\begin{lemma}\label{L:Dyda}
Let $\beta>0 $ and $q>0$ such that $\beta q < 1$. Then for all continuous functions $u$ on $[0,T]$ we have
\begin{equation}
\label{eq:dyda-general}
\int_0^T \frac{|u(t)|^q}{t^{\beta q}}dt \leq c\left(\int_0^T \int_0^T \frac{|u(t)-u(\tau)|^q}{|t-\tau|^{1+{\beta q}}}dtd\tau + \int_0^T |u(t)|^q dt\right),
\end{equation}
where $c>$ is a constant depending only on $\beta$ and $q$.  
\end{lemma}

Lemma \ref{L:Dyda} is an slight adaption of the following result in \cite{ChenSong} and \cite{dyda}: Suppose $\mathcal{D}\subset\mathbb{R}^d$ is a bounded Lipschitz domain and let $\delta_{\mathcal{D}}(x) = \inf\{|y-x|:y\in \mathcal{D}^c\}$ denote the distance of $x$ to its complement $\mathcal{D}^c$. Then, by \cite[Equation (17)]{dyda}, we have, for any $q>0$ and $\alpha\in (0,1)$, that
\begin{equation}\label{E:Bartek}
\int_\mathcal{D} \frac{|u(x)|^q}{[\delta_{\mathcal{D}}(x)]^\alpha}dx \leq c\left(\int_\mathcal{D}\int_\mathcal{D} \frac{|u(x)-u(y)|^q}{|x-y|^{d+\alpha}}dxdy + \int_\mathcal{D} |u(x)|^q dx\right)
\end{equation}
for all $u \in \mathcal{C}_c(\mathcal{D})$.

\begin{proof} 
An application of (\ref{E:Bartek}) to the case $d=1$, $\alpha = \beta q$ and $\mathcal{D} =(0,T)$ yields
\begin{equation*}
\int_0^T \frac{|u(t)|^q}{t^\beta}dt \leq c\left(\int_0^T \int_0^T \frac{|u(t)-u(\tau)|^q}{|t-\tau|^{1+\beta}}dtd\tau + \int_0^T |u(t)|^q dt\right)
\end{equation*} 
for all $u\in \mathcal{C}_c(0,T)$. Note also that $t^{-\beta} \leq [\delta_{\mathcal{D}}(t)]^{-\beta}$. Now suppose that $u\in \mathcal{C}([0,T])$. For each $n$ let $u_n$ be the continuous function on $[0,T]$ such that $u_n=u$ on $\left[\frac{1}{n}, T-\frac{1}{n}\right]$, $u_n$ is linear on $\left[\frac{1}{2n},\frac{1}{n}\right] \cup \left[T-\frac{1}{n}, T-\frac{1}{2n}\right]$, and $u_n \equiv 0$ on $ \left[0,\frac{1}{2n}\right]\cup \left[T-\frac{1}{2n},T\right]$. Then obviously $u_n\in \mathcal{C}_c(0,T)$ and we have
\begin{equation}
\label{eq:u_n-u:bound}
\sup_{t\in[0,T]}|u_n(t)| \leq S,
\end{equation}
where $S:=\sup_{t\in [0,T]}|u(t)|$. Since bounded convergence implies
\[\lim_n\int_0^T \frac{|u_n(t)|^q}{t^{\beta q}}dt = \int_0^T \frac{|u(t)|^q}{t^{\beta q}}dt\quad \text{ and }\quad 
\lim_n\int_0^T |u_n(t)|^q dt = \int_0^T |u(t)|^q dt,\]
it suffices to show that 
$$
\limsup_n\int_0^T \int_0^T \frac{|u_n(t)-u_n(\tau)|^q}{|t-\tau|^{1+{\beta q}}}dtd\tau \leq c \int_0^T \int_0^T \frac{|u(t)-u(\tau)|^q}{|t-\tau|^{1+{\beta q}}}dtd\tau.
$$
On $\left[0,\frac{1}{n}\right]$ the function $u_n$ obeys the Lipschitz bound $|u_n(t) - u_n(\tau)| \leq nS|t-\tau|$,
which implies
\[\int_0^{n^{-1}}\int_0^{n^{-1}} \frac{|u_n(t)-u_n(\tau)|^q}{|t-\tau|^{1+{\beta q}}}dtd\tau \leq S^q n^q \int_0^{n^{-1}}\int_0^{n^{-1}} |t-\tau|^{q-1-{\beta q}}d\tau dt \leq cS^qn^{\beta q-1}. \]
This goes to zero as $n\to \infty$. Similarly, using \eqref{eq:u_n-u:bound}, and writing $c$ for positive constants depending only on $q$ and $\beta$ and possibly changing from line to line, 
\begin{align}
&\int_0^{n^{-1}}\int_{n^{-1}}^T \frac{|u_n(t)-u_n(\tau)|^q}{|t-\tau|^{1+\beta q}}dtd\tau \notag\\
&\leq c\int_0^{n^{-1}}\int_{n^{-1}}^T \frac{|u_n(t)-u(t)+u(\tau)-u_n(\tau)|^q}{|t-\tau|^{1+\beta}}dtd\tau + c\int_0^{n^{-1}}\int_{n^{-1}}^T \frac{|u(t)-u(\tau)|^q}{|t-\tau|^{1+\beta q}}dtd\tau\notag\\
&\leq cS^q\int_0^{n^{-1}}\int_{n^{-1}}^T |t-\tau|^{-\beta q-1}dtd\tau +c\int_0^{T}\int_{0}^T \frac{|u(t)-u(\tau)|^q}{|t-\tau|^{1+\beta q}}dtd\tau\\
&\leq cS^q n^{\beta q-1} +c\int_0^{T}\int_{0}^T \frac{|u(t)-u(\tau)|^q}{|t-\tau|^{1+\beta q}}dtd\tau. \notag
\end{align}
Treating the regions involving intervals $\left[T-\frac{1}{n},T\right]$ similarly concludes the proof.
\end{proof}

A function $\varphi\in L_{\loc}^1(\mathbb{R}^n,\mathbb{R}^m)$, $\varphi=(\varphi_1,...,\varphi_m)$, is locally of bounded variation, denoted $\varphi\in BV_{\loc}(\mathbb{R}^n)^m$, if locally its distributional partial derivatives $D_i\varphi$ are $\mathbb{R}^m$-valued vector measures in the sense of \cite[Definition 1.4 (a)]{AFP}. We write again $\left\|D\varphi\right\|$ for the total variation of the gradient measure $D\varphi$ of $\varphi$. Elementary norm comparison in $\mathbb{R}^m$ implies that
\begin{equation}\label{E:l1dom}
\left\|D\varphi\right\|\leq \sum_{i=1}^m \left\|D\varphi_i\right\|,
\end{equation}
where $\left\|D\varphi_i\right\|$ is the total variation of the gradient measure $D\varphi_i$ of $\varphi_i$. We record a consequence of the chain rule for $BV$-functions, \cite[Theorem 3.96]{AFP}. See \cite{AmbrosioDalMaso, LeoniMorini} for more general chain rules.

\begin{lemma}\label{L:Lipschitzcomp}
If $m\geq 1$, $\varphi\in BV_{\loc}(\mathbb{R}^n)^m$ with $\varphi=(\varphi_1,...,\varphi_m)$ and $\Phi:\mathbb{R}^m\to\mathbb{R}$ is a $\mathcal{C}^1$-function with bounded gradient and $\Phi(0)=0$, then $\Phi\circ \varphi \in BV_{\loc}(\mathbb{R}^n)$ and
\begin{equation}\label{E:contract}
\left\|D(\Phi\circ\varphi)\right\|\leq \left\|\nabla\Phi\right\|_{\sup}\left\|D\varphi\right\|.
\end{equation} 
If $X\in V(\varphi_i,s,p)$ for all $i$, then also $X\in V(\Phi\circ\varphi,s,p)$.
\end{lemma}
\begin{proof}
By \cite[Theorem 3.96 and its proof]{AFP} we have $\Phi\circ \varphi \in BV_{\loc}(\mathbb{R}^n)$ and (\ref{E:contract}).
Together with (\ref{E:l1dom}) this implies that for any compact $K\subset \mathbb{R}^n$ we have
\[U^{1-s}(\left\|D(\Phi\circ\varphi)\right\||K)(x)\leq \left\|\nabla \Phi\right\|_{\sup}U^{1-s}(\left\|D\varphi\right\||_K)(x)\leq \left\|\nabla \Phi\right\|_{\sup}\sum_{i=1}^m U^{1-s}(\left\|D\varphi_i\right\||_K)(x),\quad x\in\mathbb{R}^n.\]
Now the second statement follows from (\ref{E:SVp2}). 
\end{proof}

We prove Proposition \ref{prop:dyda-bound}.

\begin{proof}[Proof of \ref{prop:dyda-bound}]
For any $N\geq 2$ let $\Phi_N\in \mathcal{C}^1(\mathbb{R})$ be an increasing function with $\left\|\Phi'\right\|_{\sup}\leq 1$ and such that $\Phi_N(y)=-N$ for $y\leq -N$, $\Phi_N(y)=N$ for $y\geq N$ and $\Phi_N(y)=y$ for $-(N-1)<y<N-1$. Then, by Lemma \ref{L:Lipschitzcomp} (with $m=1$) and the hypotheses of Proposition \ref{prop:dyda-bound}, we have $\Phi_N(\varphi)\in BV_{\loc}(\mathbb{R}^n)$ and $X$ is $(s,p)$-variable w.r.t. each $\Phi_N(\varphi)$. Suppose that (\ref{E:dyda-bound}) holds for all $\Phi_N(\varphi)$ in place of $\varphi$. Then, by (\ref{E:contract}), 
\begin{align}
&\int_0^T \frac{|(\Phi_N(\varphi) \circ X)(t)|^p}{t^{\beta p}} dt\notag\\ 
&\leq c[X]_{\alpha,\infty}^{s p}\bigg(\int_0^T\left[\int_{\mathcal{U}}\frac{\left\|D(\Phi_N(\varphi))\right\|(dy)}{|X_t-y|^{1-s}}\right]^{p}dt + \int_0^T |(\Phi_N(\varphi) \circ X)(t)|^p dt\bigg).\notag\\
&\leq c[X]_{\alpha,\infty}^{s p}\bigg(\int_0^T\left[\int_{\mathcal{U}}\frac{\left\|D\varphi\right\|(dy)}{|X_t-y|^{1-s}}\right]^{p}dt + \int_0^T |(\varphi \circ X)(t)|^p dt\bigg),\notag
\end{align}
and (\ref{E:dyda-bound}) for $\varphi$ follows using Fatou's lemma. Consequently it suffices to prove (\ref{E:dyda-bound}) under the assumption that $\sup_{x\in\mathbb{R}^n}|\varphi(x)|\leq N$, and we do so in the sequel. 

Let $(\eta_\varepsilon)_{\varepsilon>0}$ be a (radially symmetric) flat mollifier as in Appendix \ref{S:smooth_approximation}.  For each $\varepsilon>0$ the composition $\varphi_\varepsilon \circ X$ of the mollified function $\varphi_\varepsilon := \varphi\ast \eta_\varepsilon$ with the path $X$ is continuous on $[0,T]$, so that by \eqref{eq:dyda-general}, 
\begin{equation}
\label{eq:dyda-bound-mollified}
\int_0^T \frac{|\varphi_\varepsilon (X_t)|^p}{t^{\beta p}}dt \leq c\left(\int_0^T\int_0^T \frac{|\varphi_\varepsilon(X_t)-\varphi_\varepsilon(X_\tau)|^p}{|t-\tau|^{1+\beta p}}dtd\tau + \int_0^T |\varphi_\varepsilon(X_t)|^p dt\right)
\end{equation}
with a constant $c>0$ depending only on $\beta$ and $p$. By (\ref{E:convundermolly}) we have
\begin{equation}
\label{eq:pointwise-derivative}
\lim_{\varepsilon\to 0}\varphi_\varepsilon(y) = \varphi(y),\quad y\in  \mathbb{R}^n\setminus S_{\varphi},
\end{equation}
where $S_{\varphi}$ is the approximate discontinuity set of $\varphi$. Since
\[\sup_{x\in \mathbb{R}^n}|\varphi_\varepsilon(x)|\leq \sup_{x\in \mathbb{R}^n}\int_{\mathbb{R}^n}|\varphi(y)|\eta_\varepsilon(x-y)dy\leq N\]
and $\mu_X^{[0,T]}$ is finite, we can use (\ref{eq:pointwise-derivative}) together with Proposition \ref{P:key} and bounded convergence to conclude that
\begin{equation}
\label{eq:mollified-lp-bound}
\lim_{\varepsilon\to 0}\int_0^T |\varphi_\varepsilon(X_t)|^p dt= \lim_{\varepsilon\to 0} \int_{\mathbb{R}^n}|\varphi_\varepsilon(x)|^p\mu_X^{[0,T]}(dx)= \int_{\mathbb{R}^n}|\varphi(x)|^p\mu_X^{[0,T]}(dx) = \int_0^T |\varphi(X_t)|^p dt. 
\end{equation}

As before, let $\mu_{\varphi,\mathcal{U}}$ denote the restriction of $\Vert D\varphi\Vert$ to a relatively compact open set $\mathcal{U}$ containing $X([0,T])$. Let $\mu_\varepsilon:=\left\|D\varphi_\varepsilon\right\||_{X([0,T])}$. For sufficiently small $\varepsilon$ the open $\varepsilon$-parallel set $(X([0,T]))_\varepsilon$ of $X([0,T])$ is contained in $\mathcal{U}$, so that by Corollary \ref{C:mollypot} we have
\begin{equation}
\label{eq:potential-bound}
U^{1-s}(\mu_\varepsilon)(x)\leq c(n,s) U^{1-s}\mu_{\varphi,\mathcal{U}}(x),\quad x\in\mathbb{R}^n,
\end{equation}
where $c(n,s)>0$ is a constant depending only on $n$ and $s$.
Let next $0<\tau<t\leq T$ be distinct and such that $X_t,X_\tau \in \mathbb{R}^n \setminus S_\varphi$. Then Proposition \ref{P:meanvalue} and bound \eqref{eq:maximal-function-bound} imply
$$
|\varphi_\varepsilon(X_t) - \varphi_\varepsilon(X_\tau)|  \leq c(n,s)|X_t-X_\tau|^s [U^{1-s}\mu_\varepsilon(X_t) + U^{1-s}\mu_\varepsilon(X_\tau)].
$$ 
Combining with \eqref{eq:potential-bound} gives
\begin{equation}
\label{eq:difference-bound}
|\varphi_\varepsilon(X_t) - \varphi_\varepsilon(X_\tau)|  \leq c(n,s)|X_t-X_\tau|^s [U^{1-s}\mu_{\varphi,\mathcal{U}}(X_t) + U^{1-s}\mu_{\varphi,\mathcal{U}}(X_\tau)].
\end{equation}
Therefore we have 
\[
\frac{|\varphi_\varepsilon(X_t) - \varphi_\varepsilon(X_\tau)|^p}{(t-\tau)^{1+\beta p}}
\leq c(n,s)^p [X]_{\alpha,\infty}^{sp}(t-\tau)^{\alpha sp-\beta p-1}[U^{1-s}\mu_{\varphi,\mathcal{U}}(X_t) + U^{1-s}\mu_{\varphi,\mathcal{U}}(X_\tau)]^p
\]
for any such $t$ and $\tau$, and as at the end of the proof of Proposition \ref{P:basicest} we see that, thanks to $(s,p)$-variability, the right hand side is integrable over $[0,T]^2$. Since by (\ref{eq:pointwise-derivative}) we have 
\[\lim_{\varepsilon\to 0} \frac{|\varphi_\varepsilon(X_t) - \varphi_\varepsilon(X_\tau)|^p}{(t-\tau)^{1+\beta p}}=\frac{|\varphi(X_t) - \varphi(X_\tau)|^p}{(t-\tau)^{1+\beta p}}\]
for any such $t$ and $\tau$, symmetry and dominated convergence imply 
\begin{align}\label{E:limitsdoubleint}
\lim_{\varepsilon\to 0} \int_0^T\int_0^T &\frac{|\varphi_\varepsilon(X_t) - \varphi_\varepsilon(X_\tau)|^p}{|t-\tau|^{1+\beta p}}\:d\tau\:dt\\
&=\lim_{\varepsilon\to 0}\int_0^T\int_0^T \frac{|\varphi_\varepsilon(X_t) - \varphi_\varepsilon(X_\tau)|^p}{|t-\tau|^{1+\beta p}}\:\mathbf{1}_{\mathbb{R}^n\setminus S_\varphi}(X_\tau)\mathbf{1}_{\mathbb{R}^n\setminus S_\varphi}(X_t)\:d\tau\:dt\notag\\
&=\int_0^T\int_0^T \frac{|\varphi(X_t) - \varphi(X_\tau)|^p}{|t-\tau|^{1+\beta p}}\:d\tau\:dt.\notag
\end{align}

Applying Fatou's lemma to the left hand side of (\ref{eq:dyda-bound-mollified}) and using (\ref{eq:mollified-lp-bound}) and (\ref{E:limitsdoubleint}), we obtain 
\begin{align}\label{E:dydaboundfinal}
\int_0^T \frac{|\varphi (X_t)|^p}{t^{\beta p}}dt&\leq \liminf_{\varepsilon\to 0} \int_0^T \frac{|\varphi_\varepsilon (X_t)|^p}{t^{\beta p}}dt\notag\\
&\leq c\lim_{\varepsilon\to 0}\left(\int_0^T\int_0^T \frac{|\varphi_\varepsilon(X_t)-\varphi_\varepsilon(X_\tau)|^p}{|t-\tau|^{1+\beta p}}dtd\tau + \int_0^T |\varphi_\varepsilon(X_t)|^p dt\right)\notag\\
&=c\left(\int_0^T\int_0^T \frac{|\varphi(X_t)-\varphi(X_\tau)|^p}{|t-\tau|^{1+\beta p}}dtd\tau + \int_0^T |\varphi(X_t)|^p dt\right).
\end{align}
Using Proposition \ref{P:basicest} on the first integral on the right hand side and readjusting constants, we arrive at (\ref{E:dyda-bound}). 
\end{proof}

\subsection{Proof of existence and regularity of the integral}\label{SS:intexandreg}

Now the proof of Theorem \ref{thm:existence} follows easily.
\begin{proof}[Proof of Theorem \ref{thm:existence}] To show that $\varphi\circ X\in W^{\beta,1 }_0(0,T)$ as claimed in (i) we have to show that 
$$
\Vert \varphi\circ X\Vert_{\beta,1} = \int_0^T \frac{|\varphi(X_t)|}{t^\beta}dt +[\varphi\circ X]_{\beta,1} < +\infty.
$$
Here, the finiteness of the weighted $L^1$-term follows from Proposition \ref{prop:dyda-bound} and Lemma \ref{lemma:Lp-integrability}, and the finiteness of the Gagliardo seminorm follows from Proposition \ref{P:basicest}. Since $Y\in\mathcal{C}^{\gamma}([0,T],\mathbb{R})$ we have $Y \in W^{1-\beta,\infty}_T$ by (\ref{E:HolderinSobo}), provided that $1-\beta<\gamma$. Thus the existence of the integral \eqref{E:pathwiseint} as claimed in (ii) follows from Proposition \ref{the:ZSIntegralBound} by choosing $\beta\in (1-\gamma, \alpha s)$. To conclude the H\"older regularity claimed in (iii) we can follow \cite[Proposition 4.1 (II)]{NualartRascanu} and \cite[Proposition 6.2 (i)]{SchneiderZahle} and note that, with $\beta$ and $\gamma$ as stated and $0\leq \tau<t\leq T$, we have
\[
|\int_0^t\varphi(X_u)dY_u-\int_0^\tau \varphi(X_u)dY_u|
\leq \left\|Y\right\|_{W_T^{1-\beta,\infty}(0,T)}\left(\int_\tau^t\frac{|\varphi(X_u)|}{(u-\tau)^\beta}\:du+\int_\tau^t\int_\tau^u\frac{|\varphi(X_u)-\varphi(X_r)|}{(u-r)^{1+\beta}}\:drdu\right).
\]
If $X\in V(\varphi,s,p)$, then $\varphi\circ X \in W^{\beta,p}(0,T)$ by Propositions \ref{P:basicest} and  \ref{prop:dyda-bound} and Lemma \ref{lemma:Lp-integrability}. For $p=+\infty$ we have
\[\int_\tau^t\frac{|\varphi(X_u)|}{(u-\tau)^\beta}\:du\leq c\left\|\varphi\circ X\right\|_{W^{\beta,\infty}(0,T)} (t-\tau)^{1-\beta}\]
and
\[\int_\tau^t\int_\tau^u\frac{|\varphi(X_u)-\varphi(X_r)|}{(u-r)^{1+\beta}}\:drdu\leq c\left\|\varphi\circ X\right\|_{W^{\beta,\infty}(0,T)} (t-\tau)\]
as in \cite{NualartRascanu}. For $1\leq p<+\infty$ we can proceed similarly as in \cite{SchneiderZahle} and use H\"older's inequality to see that
\[\int_\tau^t\frac{|\varphi(X_u)|}{(u-\tau)^\beta}\:du\leq c \left\|\varphi\circ X\right\|_{W^{\beta,p}(0,T)} (t-\tau)^{(1-\beta)-1/p}. \]
Using $\beta<\beta'<\alpha s$ and $\frac{1}{p}+\frac{1}{q}=1$ we also have
\begin{align}
\int_\tau^t\int_\tau^u & \frac{|\varphi(X_u)-\varphi(X_r)|}{(u-r)^{1+\beta}}\:drdu \notag\\
&\leq \int_\tau^t\int_\tau^u\frac{|\varphi(X_u)-\varphi(X_r)|}{(u-r)^{\beta'+1/p}}\:\frac{1}{(u-r)^{\beta-\beta'+1/q}}\:drdu\notag\\
&\leq \left(\int_\tau^t\int_\tau^u\frac{|\varphi(X_u)-\varphi(X_r)|}{(u-r)^{\beta'p+1}}\:drdu\right)^{1/p}\left(\int_\tau^t\int_\tau^r(r-u)^{(\beta'-\beta)q-1}dudr\right)^{1/q}\notag\\
&\leq c \left\|\varphi\circ X\right\|_{W^{\beta',p}(0,T)}(t-\tau)^{(1-\beta)-1/p+\beta'}.\notag
\end{align}
\end{proof}
 
The following quantitative estimates are a byproduct of the above proof, Propositions \ref{P:basicest} and \ref{prop:dyda-bound} and (\ref{eq:ZSIntegralBound}).

\begin{corollary}\label{C:quantest}
If $\varphi$, $X$, and $Y$ satisfy the hypotheses of Theorem \ref{thm:existence} (ii), then for any $1-\gamma<\beta<\alpha s$ and any $t\in [0,T]$ we have 
\begin{equation}\label{E:fullestimate}
\left|\int_0^t\varphi(X_u)dY_u\right|
\leq c\left\|Y\right\|_{W_T^{1-\beta,\infty}(0,T)}\left([X]^s_{\alpha,\infty}\int_0^T\int_{\mathcal{U}}\frac{\left\|D\varphi\right\|(dy)}{|X_t-y|^{n-1+s}}dt + \int_0^T |\varphi(X_t)|dt \right).
\end{equation}
If $\varphi$, $X$, and $Y$ satisfy the hypotheses of Theorem \ref{thm:existence} (ii), then for any $1-\gamma<\beta<\alpha s$ and any $t\in [0,T]$ we have
\[
\left\|\int_0^{\cdot}\varphi(X_u)dY_u\right\|_{\mathcal{C}^{1-\beta}([0,T])}
\leq c\left\|Y\right\|_{W_T^{1-\beta,\infty}(0,T)}\left([X]^{sp}_{\alpha,\infty}\int_0^T\left[\int_{\mathcal{U}}\frac{\left\|D\varphi\right\|(dy)}{|X_t-y|^{n-1+s}}\right]^p dt + \int_0^T |\varphi(X_t)|^pdt \right)
\]
with straightforward modification for $p=+\infty$.
\end{corollary}

\subsection{Interpretation as currents}\label{SS:currents}

Although it will not be used in the sequel, we briefly comment on an alternative interpretation of (\ref{E:fullestimate}) which is close to the concept of stochastic currents investigated in \cite{FlandoliGubinelliRusso} and \cite{FlandoliGubinelliGiaquintaTortorelli}. Given a path $X:[0,T]\to\mathbb{R}^n$ and a number $s\in (0,1)$, set
\[[\varphi]_{X,s}:=\left\|U^{1-s}\left\|D\varphi\right\|\right\|_{L^1(X)}, \quad \varphi \in BV(\mathbb{R}^n).\]
Obviously this defines a seminorm on $BV(\mathbb{R}^n)$. Recall that a sequence $(\varphi_n)_n\subset BV(\mathbb{R}^n)$ is said to strictly converge to $\varphi \in BV(\mathbb{R}^n)$ if $\lim_n \varphi_n=\varphi$ in $L^1(\mathbb{R}^n)$ and $\lim_n \left\|D\varphi_n\right\|(\mathbb{R}^n)=\left\| D\varphi\right\|(\mathbb{R}^n)$, \cite[Definition 3.14]{AFP}. 

\begin{proposition}
For any path $X$ and any $s\in (0,1)$ the seminorm $[\cdot]_{X,s}$ is lower semicontinuous on $BV(\mathbb{R}^n)$ w.r.t. strict convergence. Moreover, 
\begin{equation}\label{E:obscure}
\left\lbrace \varphi \in BV(\mathbb{R}^n): [\varphi]_{X,s}<+\infty\right\rbrace
\end{equation}
is a subspace of $BV(\mathbb{R}^n)$, closed w.r.t. strict convergence.
\end{proposition}
\begin{proof}
If $(\varphi_n)_n\subset BV(\mathbb{R}^n)$ strictly converges to $\varphi\in BV(\mathbb{R}^n)$, then, by \cite[Proposition 3.15]{AFP}, also $\lim_n \left\|D\varphi_n\right\|=\left\|D\varphi\right\|$ vaguely. Hence, by \cite[Chapter I, Theorem 1.3]{Landkof}, we have 
\[U^{1-s}\left\|D\varphi\right\|(x)\leq \liminf_n U^{1-s}\left\|D\varphi_n\right\|(x),\quad x\in\mathbb{R}^n,\]
and by Fatou's lemma, used w.r.t. $\mu_X^{[0,T]}$, the lower semicontinuity follows. The closedness of the above subspace follows similarly.
\end{proof}

Although it is too strict for most applications, let us mention that the norm $\left\|\varphi\right\|_{BV}:=\left\|\varphi\right\|_{L^1(\mathbb{R}^n)}+\left\|D\varphi\right\|(\mathbb{R}^n)$ makes $BV(\mathbb{R}^n)$ a Banach space, and convergence in this norm implies strict convergence. Consequently (\ref{E:obscure}) is also closed w.r.t. $\left\|\varphi\right\|_{BV}$, hence itself Banach with this norm.

We write $V_{X,s}(\mathbb{R}^n,\mathbb{R}^n)$ for the Banach space of all $\varphi=(\varphi_1,...,\varphi_n) \in (BV(\mathbb{R}^n))^n$ with $[\varphi_i]_{X,s}<+\infty$ for all $i$ with norm
$\left\|\varphi\right\|_{X,s}:=\sum_{i=1}^n \left(\left\|\varphi_i\right\|_{BV(\mathbb{R}^n)} + [\varphi_i]_{X,s}\right)$.

The following interpretation of the integral as a bounded linear functional on $V_{X,s}(\mathbb{R}^n,\mathbb{R}^n)$ is a special case of (\ref{E:fullestimate}) and seems close to \cite[Remark 12] {FlandoliGubinelliGiaquintaTortorelli}.
\begin{corollary}
Let $X\in \mathcal{C}^\gamma([0,T],\mathbb{R}^n)$ with $\gamma>\frac12$ and let $s\in (\frac{1}{\gamma}-1,1)$. Then the integral $\int_0^T\varphi(X_u)dX_u$ exists for any $\varphi\in V_{X,s}(\mathbb{R}^n,\mathbb{R}^n)$ and satisfies
\[\left|\int_0^T \varphi(X_u)dX_u\right|\leq c\left\|X\right\|_{W_T^{1-\beta,\infty}(0,T,\mathbb{R}^n)}\left(1+[X]_{\gamma,\infty}^s\right)\left\|\varphi\right\|_{X,s}\]
for any $\beta\in (1-\gamma,\gamma s)$. 
\end{corollary}

\subsection{Proof of the change of variable formula}\label{SS:changeofvarproof}

We provide a proof of Theorem \ref{thm:ito} which follows by mollification of the coefficient and taking the limit. The following result yields convergence in $W^{\beta, 1}_0(0,T)$.

\begin{lemma}
\label{lem:mollified-gagliardo-convergence}
Let $\varphi\in BV_{\loc}(\mathbb{R}^n)$ and $X\in\mathcal{C}^\alpha([0,T],\mathbb{R}^n)$. Suppose that $s\in (0,1)$, $p\in [1,+\infty)$, and $X\in V(\varphi,s,p)$. Set 
$\varphi_\varepsilon = \varphi \ast \eta_\varepsilon$, where $(\eta_\varepsilon)_{\varepsilon>0}$ is a mollifier. Then for any $\beta\in (0,\alpha s)$ we have
\begin{equation}\label{eq:convergence-for-change-of-variable}
\lim_{\varepsilon\to 0}\left\|\varphi\circ X-\varphi_\varepsilon \circ X\right\|_{W^{\beta,p}_0} = 0.
\end{equation}
\end{lemma}
\begin{proof}
Compare also with \cite{Fiscellaetal}. Consider first the seminorm part
\begin{equation}\label{E:seminormpart}
[\varphi\circ X-\varphi_\varepsilon \circ X]_{\beta,p} = \int_0^T \int_0^T \frac{|\varphi(X_t)-\varphi(X_\tau) - \varphi_\varepsilon(X_t)+\varphi_\varepsilon(X_\tau)|^p}{|t-\tau|^{1+\beta p}}d\tau dt.   
\end{equation}
Since $X\in V(\varphi,s,p)$, (\ref{E:convundermolly}) and Proposition \ref{P:key} imply that 
\begin{equation}\label{E:nice}
\lim_{\varepsilon\to 0}|\varphi(X_t)-\varphi_\varepsilon(X_t)| =0
\end{equation}
for a.e. $t\in[0,T]$. Thus it suffices to find an integrable upper bound in order to apply dominated convergence theorem. For this we use
\[|\varphi(X_t)-\varphi(X_\tau) - \varphi_\varepsilon(X_t)+\varphi_\varepsilon(X_\tau)|^p\leq 2^{p-1}|\varphi(X_t)-\varphi(X_\tau)|^p +2^{p-1}|\varphi_\varepsilon(X_t)-\varphi_\varepsilon(X_\tau)|^p.\]
For the first summand on the right hand side we can use Proposition \ref{P:basicest}. For the second, Proposition \ref{P:meanvalue}, the bound (\ref{eq:maximal-function-bound}), and Corollary \ref{C:mollypot} imply that for $\varepsilon$ sufficiently small and $\mathcal{L}^1$-a.e. $t,\tau\in [0,T]$,
\begin{equation}\label{E:stereotype}
\begin{split}
\frac{|\varphi_\varepsilon(X_t)-\varphi_\varepsilon(X_\tau)|^p}{|t-\tau|^{1+\beta p}}
&\leq c\frac{|X_t-X_\tau|^{sp}}{|t-\tau|^{1+\beta p}}[U^{1-s}(\left\|D\varphi_\varepsilon\right\||_{X([0,T])})(X_t)+U^{1-s}(\left\|D\varphi_\varepsilon\right\||_{X([0,T)]})(X_\tau)]^p\\
&\leq c\frac{|X_t-X_\tau|^{sp}}{|t-\tau|^{1+\beta p}}[U^{1-s}(\left\|D\varphi\right\||_{\mathcal{U}})(X_t)+U^{1-s}(\left\|D\varphi\right\||_{\mathcal{U}})(X_\tau)]^p
\end{split}
\end{equation}
with $\mathcal{U}\supset X([0,T])$ relatively compact open as in Definition \ref{D:SVp-condition}. The quantity (\ref{E:seminormpart}) goes to zero as $\varepsilon\to 0$ by the dominated convergence theorem. As in (\ref{eq:dyda-bound-mollified}) and (\ref{E:dydaboundfinal}), we obtain 
\begin{equation}\label{E:weightedpart}
\begin{split}
\int_0^T \frac{|\varphi (X_t)-\varphi_\varepsilon(X_t)|^p}{t^{\beta p}}dt\leq & c\left(\int_0^T\int_0^T \frac{|\varphi(X_t)-\varphi_\varepsilon(X_t)+\varphi_\varepsilon(X_\tau)-\varphi(X_\tau)|^p}{|t-\tau|^{1+\beta p}}dtd\tau\right.\\
&\qquad\left. + \int_0^T |\varphi(X_t)-\varphi_\varepsilon(X_t)|^p dt\right).
\end{split}
\end{equation}
Recycling the pinning argument, we choose $t_0\in[0,T]$ such that $X_{t_0} \in \mathbb{R}^n\setminus S_\varphi$ to see that
$$
|\varphi_\varepsilon(X_t)|^p \leq 2^{p-1}\left[|\varphi_\varepsilon(X_t)-\varphi_\varepsilon(X_{t_0})|^p+|\varphi_\varepsilon(X_{t_0})|^p\right],
$$
and treat the difference $|\varphi_\varepsilon(X_t)-\varphi_\varepsilon(X_{t_0})|$ using the maximal inequalities and potential bounds as in the proof of Proposition \ref{P:basicest}, which gives an integrable upper bound. By (\ref{E:nice}) and dominated convergence we have
\[\lim_{\varepsilon\to 0} \int_0^T |\varphi(X_t)-\varphi_\varepsilon(X_t)|^p dt = 0,\]
and combining with the above, we see that also (\ref{E:weightedpart}) goes to zero as $\varepsilon\to 0$.
\end{proof}

We prove Theorem \ref{thm:ito}.
\begin{proof}[Proof of Theorem \ref{thm:ito}]
Let $F_\varepsilon = F\ast \eta_\varepsilon$, where $(\eta_\varepsilon)_{\varepsilon>0}$ is a mollifier. By Theorem \ref{thm:existence} all integrals $\int_0^t \partial_k F_\varepsilon(X_u) dX^k_u$
are well defined. We have $\partial_k F_\varepsilon = \partial_k F \ast \eta_\varepsilon$ and a standard Taylor approximation argument yields
\begin{equation}
\label{eq:Taylor}
F_\varepsilon(X_t) = F_\varepsilon(\mathring{x}) + \sum_{k=1}^n \int_0^t \partial_k F_\varepsilon (X_u)dX^k_u.
\end{equation}
By Proposition \ref{P:key}, $X_t \in \mathbb{R}^n\setminus \bigcup_{k=1}^n S_{\partial_k F}$ for $\mathcal{L}^1$-a.e. $t\in[0,T]$. Thus, by Lemma \ref{lem:mollifier-result-new}, $X_t \in \mathbb{R}^n\setminus S_{F}$ for $\mathcal{L}^1$-a.e. $t\in[0,T]$ as well and we have
\[\lim_{\varepsilon\to 0} F_\varepsilon(X_t) = F(X_t)\quad \text{and}\quad 
\lim_{\varepsilon \to 0} F_\varepsilon(\mathring{x}) = F(\mathring{x}).\]
By Lemma \ref{lem:mollified-gagliardo-convergence} we have
$$
\lim_{\varepsilon \to 0}\Vert \partial_k F_{\varepsilon}(X_\cdot)-\partial_k F(X_\cdot)\Vert_{W^{\beta,1}_0(0,T)} = 0
$$
for each $k=1,2,\ldots,n$ and for some $\beta\in(1-\alpha,\alpha s)$. Using (\ref{eq:ZSIntegralBound}) we therefore arrive at (\ref{E:changeofvar}). Finally, for continuous $F$ we have $S_F=\emptyset$ concluding the proof.
\end{proof}

\section{Existence and uniqueness proofs}\label{S:exandunique}

In this section we prove Theorems \ref{T:exonedim}, \ref{T:exjoint}, \ref{T:exshift} and \ref{thm:DE-uniqueness} and Corollary \ref{C:exshiftrandom}. 

We make repeated use of the elementary facts that a function $f=(f_1,...,f_m)$ is in the space $W^{\beta,p}(0,T,\mathbb{R}^m)$ if and only if all its components $f_i$ are in  $W^{\beta,p}(0,T)$, and that, by the norm equivalence in $\mathbb{R}^m$, the norm $\left\|f\right\|_{W^{\beta,p}(0,T,\mathbb{R}^m)}$ is comparable to $\sum_{i=1}^m \left\|f_i\right\|_{W^{\beta,p}(0,T)}$ (similarly for other function spaces). We apply the fact that estimates like (\ref{eq:ZSIntegralBound}) or (\ref{E:contimultbound}) remain valid for vector valued functions of compatible dimensions at the expense of having a different multiplicative constant in front. We also use the symbol $L^p(X,\mathbb{R}^m)$ for a vector valued version of $L^p(X)$.

\subsection{Invertibility and $BV$-regularity}\label{SS:Prelim}

In this subsection we verify Lemma \ref{L:niceinverse}. Recall that by the Cayley-Hamilton theorem the inverse $A^{-1}$ of an invertible $(n\times n)$-matrix $A$ satisfies
\begin{equation}\label{E:CayleyHamilton}
A^{-1}=\frac{(-1)^{n-1}}{\det(A)}\:(A^{n-1}+c_{n-1} A^{n-2}+...+c_1 I_n),
\end{equation}
where $I_n$ is the $(n\times n)$-identity matrix and, for $k=0,..., n-1$, one has
\[c_{n-k}=B_k(s_1, -1!s_2, 2!s_3,...,(-1)^{k-1}(k-1)!s_k),\]
with $B_k$ denoting the $k$-th complete exponential Bell polynomial and $s_l=\tr(A^l)$ being the trace of $A^l$. We prove Lemma \ref{L:niceinverse}.

\begin{proof}[Proof of Lemma \ref{L:niceinverse}]
Suppose that $\sigma$ satisfies Assumption \ref{A:main}. Recall that we always consider fixed Lebesgue representatives of the components $\sigma_{ij}$. Let $N\subset \mathbb{R}^n$ be a $\mathcal{L}^n$-null set such that
\begin{equation}\label{E:sigmadetbound}
\det (\sigma(x))>\varepsilon\quad \text{and}\quad |\sigma_{ij}(x)|\leq \left\|\sigma\right\|_{L^\infty(\mathbb{R}^n,\mathbb{R}^{n\times n})}\quad \text{for all $x\in\mathbb{R}^n\setminus N$ and any $i$ and $j$.}
\end{equation}
For such $x$ the matrix $\sigma(x)$ is invertible, and we set $\hat{\sigma}(x):=(\sigma(x))^{-1}$. For $x\in N$ we can set $\hat{\sigma}(x):=0$. For $x\in \mathbb{R}^n\setminus N$ the matrices $\hat{\sigma}(x)$ of $\sigma(x)$ satisfy (\ref{E:CayleyHamilton}) in place of $A^{-1}$ and $A$. In particular, since for all quadratic matrices $A$ the elements of the matrix products $A^k$, traces $\text{tr}(A^k)$, and the determinant $\text{det}(A)$ are polynomials of the elements of $A$, we observe that the coefficients $\hat{\sigma}_{ij}(x)$ of $\hat{\sigma}(x)$ are rational functions of the coefficients $\sigma_{ij}(x)$ of the form
\begin{equation}\label{E:polynomials}
\hat{\sigma}_{ij}(x) = \frac{P_{ij}(\sigma(x))}{\det(\sigma(x))},
\end{equation}
where for each $i$ and $j$ the function $P_{ij}(\sigma)$ is a polynomial of degree $n-1$ in the coefficients $\sigma_{kl}$, $k,l=1,...,n$, of $\sigma$. By the boundedness of the $\sigma_{kl}$ we have $P_{ij}(\sigma)\in L^\infty(\mathbb{R}^n)$ and by (\ref{E:sigmadetbound}) and (\ref{E:polynomials}) also $\hat{\sigma}_{ij}\in L^\infty(\mathbb{R}^n)$. For fixed $i$ and $j$ let $p:=P_{ij}-P_{ij}(0)$ and let $u\in \mathcal{C}^1_c(\mathbb{R}^{n\times n},\mathbb{R}^2)$ be such that $u:=(p,\det)$ on the image $\{\sigma(x):\ x\in \mathbb{R}^n\setminus N\}\subset  \mathbb{R}^{n\times n}$ of 
$\mathbb{R}^n\setminus N$ under $\sigma$. Let $F\in \mathcal{C}_c^1(\mathbb{R}^2)$ be such that $F(a,b)=\frac{a}{b}$ for $(a,b)\in [-\left\|p(\sigma)\right\|_{L^\infty(\mathbb{R}^n)}, \left\|p(\sigma)\right\|_{L^\infty(\mathbb{R}^n)}]\times [\varepsilon, \left\|\det(\sigma)\right\|_{L^\infty(\mathbb{R}^n)}]$. Then $\Phi:=F\circ u$ is an element of $\mathcal{C}^1(\mathbb{R}^{n\times n})$, its gradient $\nabla\Phi$ is bounded and 
\begin{equation}\label{E:viewascompo}
\hat{\sigma}_{ij}(x)=\Phi\circ \sigma (x),\quad x\in\mathbb{R}^n\setminus N.
\end{equation}
As in Lemma \ref{L:Lipschitzcomp} the chain rule, \cite[Theorem 3.96]{AFP}, now shows that $\hat{\sigma}_{ij}\in BV_{\loc}(\mathbb{R}^n)$. To see the last statement of the lemma, note that for $x\in \mathbb{R}^n\setminus N$ we have $\det \sigma(x)\leq n! \left\|\sigma\right\|_{L^\infty(\mathbb{R}^n,\mathbb{R}^{n\times n})}$, so that  
\[\det(\hat{\sigma}(x))=\frac{1}{\det(\sigma(x))}\geq \frac{1}{n! \left\|\sigma\right\|_{L^\infty(\mathbb{R}^n,\mathbb{R}^{n\times n})}}.\]
\end{proof}

\subsection{Solutions to the deterministic equation}\label{SS:invertible}

In this subsection we provide a proof for Proposition \ref{P:solvedeteq}. Assumptions \ref{A:main} and \ref{A:main2} allow to conclude that $\hat{\sigma}$ has a Lipschitz potential.

\begin{proposition}\label{P:Lipschitz}
Suppose $\sigma$ satisfies Assumptions \ref{A:main} and \ref{A:main2}. Then there exists a Lipschitz function $g:\mathbb{R}^n\to \mathbb{R}^n$ such that its Jacobian matrix $\nabla g$, defined a priori in distributional sense, satisfies 
\begin{equation}\label{E:potentialg}
\nabla g =\hat{\sigma}\quad \text{$\mathcal{L}^n$-a.e.}
\end{equation}
In particular, $g\in W^{1,\infty}(\mathcal{O})$ for any bounded domain $\mathcal{O}\subset\mathbb{R}^n$.
\end{proposition}

We record a short helpful argument to conclude local continuity and boundedness. By $\mathcal{S}(\mathbb{R}^n,\mathbb{R}^n)$ and $\mathcal{S}'(\mathbb{R}^n,\mathbb{R}^n)$ we denote the spaces of $\mathbb{R}^n$-valued Schwartz functions and tempered distributions, respectively.

\begin{lemma}\label{L:gradienttofunction}
If $n\geq 2$, $1<p<+\infty$ and $G\in\mathcal{S}'(\mathbb{R}^n)$ is such that $\nabla G\in L^p(\mathbb{R}^n,\mathbb{R}^n)$, then $G\in L^q(\mathbb{R}^n)$, where $\frac{1}{q}=\frac{1}{p}-\frac{1}{n}$. 
\end{lemma}
\begin{proof}
Given $\varphi\in\mathcal{S}(\mathbb{R}^n)$ let $\psi:=-\nabla (-\Delta)^{-1}\varphi$, where $\Delta$ denotes the Laplacian on $\mathbb{R}^n$ and $(-\Delta)^{-1}$ the Newton potential. Then $\psi\in \mathcal{S}(\mathbb{R}^n,\mathbb{R}^n)$ and 
\[\left\langle G,\varphi\right\rangle=-\left\langle G, \diverg \psi\right\rangle=\left\langle \nabla G, \psi\right\rangle=\int_{\mathbb{R}^n}(\nabla G)\cdot\psi \:dx,\]
so that 
\begin{align}
|\left\langle G,\varphi\right\rangle|&\leq \left\||\nabla G|\right\|_{L^p(\mathbb{R}^n)}\left\||\nabla (-\Delta)^{-1}\varphi|\right\|_{L^{p'}(\mathbb{R}^n)}\notag\\
&\leq c\left\||\nabla G|\right\|_{L^p(\mathbb{R}^n)}\big\| (-\Delta)^{-1/2}\varphi\big\|_{L^{p'}(\mathbb{R}^n)}\notag\\
&\leq c\left\||\nabla G|\right\|_{L^p(\mathbb{R}^n)}\left\|\varphi\right\|_{L^{q'}(\mathbb{R}^n)}\notag
\end{align}
by the $L^{p'}$-boundedness of the Riesz transform, \cite[Chapter II, Section 4, Theorem 3 and  Chapter III, Section 1]{Stein70}, and the fractional Sobolev inequality, \cite[Chapter V, Section 1, Theorem 1]{Stein70}. Consequently $G\in L^q(\mathbb{R}^n)$ by duality.
\end{proof}

\begin{remark}\label{R:Hoelder}
Recall that for $q>n$ and for any smooth bounded domain $\mathcal{O}\subset\mathbb{R}^n$ we have the Sobolev embedding $W^{1,q}(\mathcal{O})\subset \mathcal{C}^\delta(\overline{\mathcal{O}})$, where $0<\delta\leq 1-\frac{n}{q}$. For instance, see \cite[4.27]{AdamsSobo}.
\end{remark}

Using Lemma \ref{L:gradienttofunction} and Remark \ref{R:Hoelder} we can prove Proposition \ref{P:Lipschitz}.

\begin{proof} In the case that $n=1$ we can simply integrate $\hat{\sigma}$ to find a locally bounded Lipschitz function $g$ satisfying (\ref{E:potentialg}). Suppose $n\geq 2$. By Assumption \ref{A:main2} and a distributional version of Poincar\'e's Lemma, there exists $g=(g_1,...,g_n)\in\mathcal{S}'(\mathbb{R}^n,\mathbb{R}^n)$ such that $\nabla g =\hat{\sigma}$ holds in distributional sense. That is, the $j$th row $\nabla g_j =(D_1 g_j,...,D_n g_j)$ of $\nabla g$ equals the $j$th row
$\hat{\sigma}=(\hat{\sigma}_{j1},...,\hat{\sigma}_{jn})$ of $\hat{\sigma}$ in $\mathcal{S}'(\mathbb{R}^n, \mathbb{R}^n)$. See \cite[Chapter 4, Section 3, Proposition 9]{Horvath} of \cite[Chapter II, Section 6, Th\'eor\`eme VI]{Schwartz} (the result is stated for the dual of smooth compactly supported functions, but the proof does not change for Schwartz functions and tempered distributions). For any $j$ and $k$ we have $D_k g_j=\hat{\sigma}_{jk}\in L^1(\mathbb{R}^n)\cap L^\infty(\mathbb{R}^n)$ by  Assumption \ref{A:main} and Lemma \ref{L:niceinverse}. Now suppose that $\mathcal{O}\subset\mathbb{R}^n$ is a bounded domain. 
By Lemma \ref{L:gradienttofunction} and the boundedness of $\mathcal{O}$ we can conclude that $g\in W^{1,q}(\mathcal{O},\mathbb{R}^n)$ for arbitrarily large $q<+\infty$, and Remark \ref{R:Hoelder} implies that $g$ is continuous on $\overline{\mathcal{O}}$, and hence also bounded on $\overline{\mathcal{O}}$. As a consequence, we have $g\in W^{1,\infty}(\mathcal{O},\mathbb{R}^n)$, and a standard mollifier argument yields $|g(x)-g(y)|\leq \left\|\nabla g\right\|_{L^\infty(\mathbb{R}^n,\mathbb{R}^{n\times n})}|x-y|$ for all $x,y\in\mathcal{O}$, see \cite[Theorem 4.1]{Heinonen}. 
\end{proof}

Our proof of Proposition \ref{P:solvedeteq} is based on an inverse function theorem for Sobolev functions proved in \cite[Theorem 1]{Kovalev}. We quote a special case of this result sufficient for our purposes. See also \cite{KovalevOnninenRajala}.
\begin{proposition}\label{P:Kovalev}
Let $n\geq 2$. If $g\in W_{\loc}^{1,n}(\mathbb{R}^n,\mathbb{R}^n)$ is non-constant and there exist $\delta>-1$ and $\kappa\in [1,+\infty)$ such that for $\mathcal{L}^n$-a.e. $x\in\mathbb{R}^n$ we have 
\begin{equation}\label{E:angle2}
\left\langle \nabla g(x)\:\xi,\xi\right\rangle \geq \delta\:|\nabla g(x)\:\xi|\:|\xi|,\quad \xi \in\mathbb{R}^n,
\end{equation}
and
\begin{equation}\label{E:quasireg}
|\nabla g(x)|^n\leq \kappa\:\det (\nabla g(x)), 
\end{equation}
then $g$ is a homeomorphism of $\mathbb{R}^n$ onto itself.
\end{proposition}

\begin{remark}\label{R:quasireg} Condition (\ref{E:quasireg}) is usually rephrased by saying that $g$ is \emph{$\kappa$-quasiregular} or \emph{of bounded distortion}, see for instance \cite[Section 1.2]{HenclKoskela} or \cite{Reshetnyak}. 
A result similar to Proposition \ref{P:Kovalev} is \cite[Theorem 1]{Cristea}. There condition (\ref{E:angle2}) is replaced by a $\mathcal{L}^n$-a.e. nonnegativity condition on the sums of principal minors of $\sigma$.
\end{remark}

We prove Proposition \ref{P:solvedeteq}.

\begin{proof}[Proof of Proposition \ref{P:solvedeteq}]
By Assumption \ref{A:main}, Lemma \ref{L:niceinverse}, Assumption \ref{A:main2}, and Proposition \ref{P:Lipschitz} there exists a Lipschitz function $g\in W_{\loc}^{1,n}(\mathbb{R}^n,\mathbb{R}^n)$ such that (\ref{E:potentialg}) holds. In the case $n=1$, Assumption \ref{A:main} and (\ref{E:potentialg}) imply that $\left\|\sigma\right\|_{L^\infty(\mathbb{R})}^{-1}\leq \hat{\sigma}(x)=g'(x)$
for $\mathcal{L}^1$-a.e. $x\in\mathbb{R}$, and hence $g$ is strictly monotone and bi-Lipschitz with inverse $f=g^{-1}$
having Lipschitz constant bounded by $\left\|\sigma\right\|_{L^\infty(\mathbb{R})}$. Assume $n\geq 2$. Again by Assumption \ref{A:main} and Lemma \ref{L:niceinverse} we have 
\[\frac{1}{n!\left\|\sigma\right\|_{L^\infty(\mathbb{R}^n,\mathbb{R}^{n\times n})}^n}\leq \frac{1}{\det(\sigma (x))}=\det(\hat{\sigma}(x))=\det(\nabla g(x))\]
for $\mathcal{L}^n$-a.e. $x\in\mathbb{R}^n$, and by (\ref{E:potentialg}) also $|\nabla g(x)|^n\leq \left\|\hat{\sigma}\right\|_{L^\infty(\mathbb{R}^n,\mathbb{R}^{n\times n})}^n$. Hence (\ref{E:quasireg}) holds with \[\kappa=n!\:\left\|\hat{\sigma}\right\|_{L^\infty(\mathbb{R}^n,\mathbb{R}^{n\times n})}^n\left\|\sigma\right\|_{L^\infty(\mathbb{R}^n,\mathbb{R}^{n\times n})}^n.\]
Condition (\ref{E:angle2}) is immediate from (\ref{E:angular}). Consequently $g$ is a homeomorphism by Proposition \ref{P:Kovalev}, and we denote its inverse by $f=g^{-1}$. For any bounded domain $\mathcal{O}$ the restriction $g|_\mathcal{O}$ is a homeomorphism with inverse $f|_{g(\mathcal{O})}$, and by Proposition \ref{P:Lipschitz} also $g|_{\mathcal{O}}\in W^{1,n-1}(\mathcal{O},\mathbb{R}^n)$. This implies that $f\in W^{1,1}_{\loc}(g(\mathcal{O}),\mathbb{R}^n)$, see \cite[Theorem 5.2]{HenclKoskela}. Consequently the chain rule can be applied to $g\circ f=\id$, and taking into account (\ref{E:potential}) we obtain 
\[I=\nabla (g\circ f)(y)=(\nabla g)(f(y))\nabla f(y)=\sigma^{-1}(f(y))\nabla f(y)\]
for $\mathcal{L}^n$-a.a. $y\in \mathbb{R}^n$. The invertibility of $\sigma$ yields $\nabla f=\sigma(f)$ $\mathcal{L}^n$-a.e. and since $\sigma \in L^\infty(\mathbb{R}^n,\mathbb{R}^{n\times n})$, it follows that $f$ is Lipschitz with Lipschitz constant bounded by $\left\|\sigma\right\|_{L^\infty(\mathbb{R}^n,\mathbb{R}^{n\times n})}$.
\end{proof} 

\subsection{Existence of solutions in the invertible case}

We observe the following stability property of conditions of upper regularity type.

\begin{lemma}\label{L:upperregstable}
Let $Y:[0,T]\to\mathbb{R}^n$ be a path, $B\subset \mathbb{R}^n$ Borel and $d>0$. Suppose that $f:\mathbb{R}^n\to\mathbb{R}^n$ is a Borel function with Lipschitz inverse $g=f^{-1}$. Then there is a constant $M>0$ such that for all $u\in [0,T]$ we have  
\[ \int_0^T |f(Y_t)-f(Y_u)|^{-d}dt\leq M\int_0^T |Y_t-Y_u|^{-d}dt.\]
\end{lemma}
\begin{proof}
For any $t,u\in [0,T]$ we have $|Y_t-Y_u|=|g(f(Y_t))-g(f(Y_u))|\leq \lip(g)|f(Y_t)-f(Y_u)|$,
and consequently the claimed estimate holds with $M= \lip(g)^d$. 
\end{proof}

In the next lemma we justify the application of the change of variable (\ref{E:changeofvar}).

\begin{lemma}\label{L:solve}
Let $g$ and $f$ be as in Lemma \ref{L:upperregstable}. Let $s\in (0,1)$, $\gamma\in (\frac{1}{1+s},1)$, $Y\in\mathcal{C}^\gamma([0,T],\mathbb{R}^n)$ with $Y_0=0$, $B\subset\mathbb{R}^n$ a Borel set, and $\mathring{x}\in\mathbb{R}^n$.  Suppose that $\sigma$ is upper $d$-regular on $B$ with $d>n-1+s$ or that 
\begin{equation}\label{E:onceagain}
\sup_{x\in Y([0,T])\cap (f^{-1}(B)-f^{-1}(\mathring{x}))} \int_0^T |Y_t-x|^{-n+1-s}dt<+\infty,
\end{equation}
and similarly for $B^c$. Then the path $X:[0,T]\to\mathbb{R}^n$, defined by $X_t:=f(Y_t+g(\mathring{x}))$, $t\in [0,T]$, is $(s,1)$-variable w.r.t. $\sigma$, and
\begin{equation}\label{E:Xsolves}
X_t=\mathring{x}+\int_0^t \sigma(X_u)dY_u,\quad t\in [0,T].
\end{equation}
\end{lemma}
\begin{proof}
If $Y$ satisfies (\ref{E:onceagain}), then by Lemma \ref{L:upperregstable} we have 
\[\sup_{u\in [0,T], X_u\in B} \int_0^T |X_t-X_u|^{-n+1-s}dt<+\infty,\]
and a similar conclusion is true for $B^c$. Therefore the hypotheses on $\sigma$, together with Corollary \ref{C:localizeB} and Remark \ref{R:SVviaoccu}, imply that $X$ is $(s,1)$-variable w.r.t. $\sigma$. If $f$ is bi-Lipschitz on $\mathbb{R}^n$ it is proper, and applying \cite[Theorem 3.16]{AFP} component-wise, we may conclude that for each $j$ and $k$ the function $\partial_k f_j=\sigma_{jk}(f)$ is in $BV_{\loc}(\mathbb{R}^n)$, see also \cite{Hencl}. By the same theorem and the bijectivity of $f$ we also have 
\[\left\|D (\partial_k f_j)\right\|=\left\|D(\sigma_{jk}\circ g^{-1})\right\|=\left\|D(g_\#\sigma_{jk})\right\|\leq(\lip(g))^{n-1}\:g_\#\left\|D\sigma_{jk}\right\|,\]
with the notation $g_\#$ from \cite[Theorem 3.16]{AFP}, which is the pushforward operation on measures, however, defined differently when operating on functions.
Using $|f(Y_t^{g(\mathring{x})})-z|\leq \lip(g)|Y_t^{g(\mathring{x})}-g(z)|$, the fact that $g$ is proper, and the fact that $X=f(Y^{g(\mathring{x})})$ is $(s,1)$-variable w.r.t. $\sigma$, 
\begin{align}
\int_0^{T} U^{1-s} &\left\|D (\partial_k f_j)  \right\||_K(Y_t^{g(\mathring{x})})dt \notag\\
&\leq (\lip(g))^{n-1} \int_0^{T}\int_K |Y_t^{g(\mathring{x})}-a|^{-n+1-s} g_\#\left\|D\sigma_{jk}\right\|(da)dt\notag\\
&=(\lip(g))^{n-1} \int_0^{T}\int_{g^{-1}(K)} |Y_t^{g(\mathring{x})}-g(z)|^{-n+1-s}\left\|D\sigma_{jk}\right\|(dz)dt\notag\\
&\leq (\lip(g))^{2n-2+s}\int_0^{T}\int_{f(K)}|f(Y_t^{g(\mathring{x})})-z|^{-n+1-s}\left\|D\sigma_{jk}\right\|(dz)dt<+\infty.\notag
\end{align}
That is, $Y^{g(\mathring{x})}$  is $(s,1)$-variable w.r.t. $\partial_k f_j$. Consequently, by Theorem \ref{thm:ito},
\[X_t=f(Y_t^{g(\mathring{x})})=f(Y_0^{g(\mathring{x})}) +\int_0^t\nabla f(Y_u^{g(\mathring{x})})dY_u=\mathring{x}+\int_0^t \sigma(X_u)dY_u\]
for all $t\in [0,T]$. The stated H\"older regularity is clear since $f$ is Lipschitz. Finally, if $n=1$ and $g$ is as in Theorem \ref{T:exonedim}, we can use the fact that $f'(Y_t)=\sigma(X_t)$ for $\mathcal{L}^1$-a.e. $t\in [0,T]$ and arrive at the same conclusion.
\end{proof}

\begin{lemma}\label{L:solveweigthed}
Let $g$ and $f$ be as in Lemma \ref{L:upperregstable}.  Let $s\in (0,1)$, $\gamma\in (\frac{1}{1+s},1)$, $Y\in\mathcal{C}^\gamma([0,T],\mathbb{R}^n)$ with $Y_0=0$ and let $\mathring{x}\in\mathbb{R}^n$. Suppose that there are $\varepsilon\in (0,1-s)$, $c>0$, and $\delta\in (0,n-1+s-\varepsilon)$ such that (\ref{E:LauriSolen}) and (\ref{E:momentconditionx0}) hold. Then  $X_t:=f(Y_t+g(\mathring{x}))$, $t\in [0,T]$, defines a path $X$ that is $(s,1)$-variable w.r.t. $\sigma$ and satisfies (\ref{E:Xsolves}).
\end{lemma}
\begin{proof}
By Lemma \ref{L:upperregstable} and the hypotheses we have 
\begin{multline}
\int_0^T|X_t-x|^{-n+\varepsilon}dt=\int_0^T|f(Y_t^{g(\mathring{x})})-x|^{-n+\varepsilon}dt\leq M\:\int_0^T|Y_t^{g(\mathring{x})}-g(x)|^{-n+\varepsilon}dt\notag\\
=\int_0^T|Y_t-(g(x)-g(\mathring{x}))|^{-n+\varepsilon}dt\leq cM\:|g(x)-g(\mathring{x})|^{-n+\delta}
\leq cM\:\lip(f)^{n-\delta}|x-\mathring{x}|^{-n+\delta}.\notag
\end{multline}
Consequently, Corollary \ref{C:altcondrandom} implies that for $\mathbb{P}$-a.e. $\omega\in\Omega$ the path $X(\omega)$ is $(s,1)$-variable w.r.t. $\sigma$. One can now follow the arguments in the proof of Lemma \ref{L:solve}.
\end{proof}

A slight modification gives a probabilistic version of the statement.
\begin{corollary}\label{C:solverandom}
Let $g$ and $f$ be as in Lemma \ref{L:upperregstable}.  Let $s\in (0,1)$, $\gamma\in (\frac{1}{1+s},1)$, let $Y=(Y_t)_{t\in [0,T]}$ be a process with $\mathbb{P}$-a.s. H\"older continuous of order $\gamma$ as in Corollary \ref{C:exshiftrandom}, and let $\mathring{x}\in\mathbb{R}^n$. Suppose that there are $\varepsilon\in (0,1-s)$, $c>0$, and $\delta\in (0,n-1+s-\varepsilon)$ such that (\ref{E:momentconditionx0}) and (\ref{E:LauriSolenrandom}) hold. Then for $\mathbb{P}$-a.e. $\omega\in \Omega$ the path $X(\omega):[0,T]\to\mathbb{R}^n$, defined by $X_t(\omega):=f(Y_t(\omega)+g(\mathring{x}))$, $t\in [0,T]$, is $(s,1)$-variable w.r.t. $\sigma$, and satisfies \[X_t(\omega)=\mathring{x}+\int_0^t \sigma(X_u(\omega))dY_u(\omega),\quad t\in [0,T].\]
\end{corollary}
\begin{proof}
In view of the preceding proof it suffices to note that 
\[\mathbb{E}\int_0^T|X_t-x|^{-n+\varepsilon}dt=\mathbb{E}\int_0^T|Y_t-(g(x)-g(\mathring{x}))|^{-n+\varepsilon}dt\leq cM\:|g(x)-g(\mathring{x})|^{-n+\delta}.\]
\end{proof}

We can now verify our existence results.

\begin{proof}[Proofs of Theorems \ref{T:exonedim} and \ref{T:exjoint}.] 
Statement (i) in Theorem \ref{T:exonedim} is clear, the first part of statement (ii) follows as in \cite[Theorem 2.1]{LauriSole}, the second part of (ii) is easily seen using Lemma \ref{L:upperregstable}. By Proposition \ref{P:solvedeteq} also the hypotheses on $\sigma$ in Theorem \ref{T:exjoint} guarantee the existence of $f$ and $g$ are as required by Lemma \ref{L:upperregstable}. Consequently Lemma \ref{L:solve} applies and yields the desired statement.
\end{proof}

\begin{proof}[Proofs of Theorem \ref{T:exshift} and Corollary \ref{C:exshiftrandom}.]
Proposition \ref{P:solvedeteq} yields functions $f$ and $g$ as in Lemma \ref{L:upperregstable}. Lemma \ref{L:solveweigthed} then
implies Theorem \ref{T:exshift} and Corollary \ref{C:solverandom} gives Corollary \ref{C:exshiftrandom}.
\end{proof}

\subsection{Uniqueness of solutions in the invertible case}

In this subsection we prove Theorem \ref{thm:DE-uniqueness}. We follow the basic idea of \cite[Lemma 3.5]{Johanna} and \cite[Theorem 2.2]{LauriSole}, and employ a smoothing argument which permits to use Riemann sum approximations. 

The following Proposition is a multidimensional version of the results proved in \cite[Theorem 4.3.1]{Zahle98} and straightforward to see from there.

\begin{proposition}
\label{pro:matrix-ito}
Let $h=(h_1,...,h_n)$ be an element of $\mathcal{C}^2(\mathbb{R}^n, \mathbb{R}^n)$. Then for any $X\in \mathcal{C}^\alpha([0,T],\mathbb{R}^n)$ with $\alpha>\frac12$ and any $t\in[0,T]$ we have
$$
h(X_t) = h(X_0) + \int_0^t \nabla h(X_u)dX_u.
$$
Moreover, the integrals can be understood also as Riemann-Stieltjes integrals.
\end{proposition}

Suppose now that $\sigma$ satisfies Assumptions \ref{A:main} and \ref{A:main2}, let $\hat{\sigma}$ be as in Lemma \ref{L:niceinverse}, and let $g$ be the Lipschitz solution to \eqref{E:potentialg}. That is, $\nabla g=\hat{\sigma}$ $\mathcal{L}^n$-a.e on $\mathbb{R}^n$. We denote by $g_\delta=(g_{1,\delta},...,g_{n,\delta})$ the element-wise smooth approximation of $g=(g_1,...,g_n)$ defined by 
\begin{equation}
\label{eq:elementwise-mollifier}
g_{i,\delta} = g_i \ast \eta_{\delta},
\end{equation}
where $(\eta_\delta)_{\delta>0}$ is a mollifier as in Lemma \ref{lem:mollifier-result-new}. From the identity 
$\nabla\left(g \ast \eta_{\delta}\right) = \nabla g \ast \eta_{\delta}$
it is immediate that 
\begin{equation}
\label{eq:g_n}
\nabla g_\delta =\hat{\sigma}_\delta\quad \text{$\mathcal{L}^n$-a.e.},
\end{equation}
where $\hat{\sigma}_\delta=((\hat{\sigma}_{jk})_\delta)_{1\leq j,k\leq n}$ is the element-wise mollified version of $\hat{\sigma}$
defined by 
\[(\hat{\sigma}_{jk})_\delta=\hat{\sigma}_{jk}\ast \eta_{\delta}.\]
By the continuity of $g$, we have 
\begin{equation}\label{E:convg}
\lim_{\delta\to 0}g_\delta (y) = g (y), \quad y \in \mathbb{R}^n,
\end{equation}
and since $g_i\in W^{1,1}_{\loc}(\mathbb{R}^n)$ for all $i$ and $\hat{\sigma}_{jk} \in BV(\mathbb{R}^n)$ for all $j$ and $k$ by Lemma \ref{L:niceinverse}, we have 
\begin{equation}\label{E:convsigmahat}
\lim_{\delta\to 0} \hat{\sigma}_\delta(y)= \hat{\sigma}(y), \quad y\in \mathbb{R}^n \setminus S_{\hat{\sigma}} 
\end{equation}
by Lemma  \ref{lem:mollifier-result-new}.

\begin{proposition}
\label{prop:g_n-rep}
Suppose that $\sigma$ satisfies Assumptions \ref{A:main} and \ref{A:main2}, $Y\in\mathcal{C}^\gamma([0,T],\mathbb{R}^n)$, and $\mathring{x}\in\mathbb{R}^n$. If $\alpha\in (\frac{1}{2},1)$ and $X\in \mathcal{C}^\alpha([0,T],\mathbb{R}^n)$ is a variability solution for $\sigma$ and $Y$ started at $\mathring{x}$, then for any $\delta>0$ and $t\in[0,T]$ we have 
\begin{equation}
\label{eq:g_n-solution}
g_\delta(X_t) = g_\delta(\mathring{x}) + \int_0^t \hat{\sigma}_\delta(X_u)\sigma(X_u)dY_u,
\end{equation}
where $g_\delta$ and $\hat{\sigma}_\delta$ are as in (\ref{eq:g_n}).
\end{proposition}

To prove Proposition \ref{prop:g_n-rep} we can follow the strategy of \cite[Theorem 4.4.2]{Zahle98} and proceed by Riemann sum approximation. We make use of the following elementary observation.

\begin{lemma}
\label{lem:RS-convergence}
Suppose that $\beta \in (0,1)$, $t\in [0,T]$ and that $B\in W^{\beta,1}(0,t,\mathbb{R}^{n\times n})\cap L^\infty(0,t,\mathbb{R}^{n\times n})$ is an $(n\times n)$-matrix valued function. If $(A_m)_{\geq 1}$ is a sequence of $(n\times n)$-matrix valued functions $A_m\in W^{\beta,1}_0(0,t,\mathbb{R}^{n\times n})\cap L^\infty(0,t,\mathbb{R}^{n\times n})$ satisfying
\begin{equation}\label{E:Amassumptions}
\lim_{m\to \infty} \left\|A_m\right\|_{W^{\beta,1}_0(0,t,\mathbb{R}^{n\times n})}=0,\ \sup_m \left\|A_m\right\|_{L^\infty(0,t,\mathbb{R}^{n\times n})}<+\infty\ \text{ and } \ \lim_{m\to \infty}  |A_m(u)|=0\quad \text{for $\mathcal{L}^1$-a.e. $u\in [0,t],$}
\end{equation}
then 
\[\lim_{m\to \infty} \left\|A_m B\right\|_{W^{\beta,1}_0(0,t,\mathbb{R}^{n\times n})}=0.\]
\end{lemma}
\begin{proof}
The boundedness of $B$ and the first assumption in (\ref{E:Amassumptions}) give 
\[\limsup_{m\to \infty} \int_0^t \frac{| A_m(u)B(u)|}{u^\beta}du \leq  \left\|B\right\|_{L^\infty(0,t,\mathbb{R}^{n\times n})} \lim_{m\to\infty} \int_0^t \frac{| A_m(u)|}{u^\beta}du = 0.\]
On the other hand, we have
\[| A_m(u)B(u)-A_m(r)B(r)| \leq | A_m(u)||B(u)-B(r)| + |B(r)| |A_m(u)-A_m(r)|.\]
The second and the third assumption in (\ref{E:Amassumptions}) imply
\[\lim_{m\to \infty}\int_0^t \int_0^t \frac{|A_m(u)| |B(u)-B(r)|}{|u-r|^{\beta+1}}du dr =0\]
by bounded convergence, and the first assumption in (\ref{E:Amassumptions}) ensures that
\[\limsup_{m\to \infty}\int_0^t \int_0^t \frac{|B(r)| |A_m(u)-A_m(r)|}{|u-r|^{\beta+1}}du dr \leq \left\|B\right\|_{L^\infty (0,t,\mathbb{R}^{n\times n})} \lim_{m\to \infty} \left\|A_m\right\|_{W^{\beta,1}_0(0,t,\mathbb{R}^{n\times n})}= 0.\]
Combining, we arrive at the result.
\end{proof}

We prove Proposition \ref{prop:g_n-rep}. 

\begin{proof}
To show that the integral is well-defined we first claim that for any $\beta<\alpha s$ we have 
\begin{equation}\label{E:finitenessclaim}
\Vert \hat{\sigma}_\delta(X_\cdot)\sigma(X_\cdot)\Vert_{W_0^{\beta,1}(0,T,\mathbb{R}^{n\times n})} < \infty.
\end{equation}
To see this, note that $\hat{\sigma}_k$ is bounded on a locally compact neighborhood $\mathcal{O}$ containing $X([0,T])$, and since $X$ is a variability solution, Proposition \ref{prop:dyda-bound} implies that
$$
\int_0^T \frac{|\hat{\sigma}_\delta(X_u)\sigma(X_u)|}{u^{\beta}}du \leq \sup_{x\in\mathcal{O}}|\hat{\sigma}_\delta(x)| \int_0^T \frac{| \sigma(X_u)|}{u^{\beta}}du < \infty.
$$
Using (\ref{E:Linftybound}) we obtain 
\begin{align*}
&\int_0^T \frac{| \hat{\sigma}_\delta(X_u)\sigma(X_u)-\hat{\sigma}_\delta(X_r)\sigma(X_r)|}{|u-r|^{\beta+1}}dudr\\
\leq & \sup_{x\in\mathcal{O}}|\hat{\sigma}_\delta(x)| \int_0^T \frac{| \sigma(X_u)-\sigma(X_r)|}{|u-r|^{\beta+1}}dudr
+ \left\|\sigma\right\|_{L^\infty(\mathbb{R}^n,\mathbb{R}^{n\times n})}  \int_0^T \frac{| \hat{\sigma}_\delta(X_u)-\hat{\sigma}_\delta(X_r)|}{|u-r|^{\beta+1}}dudr.
\end{align*}
Proposition \ref{P:basicest} implies the boundedness of the first integral on the right hand side,
and the second integral is bounded by H\"older continuity of $X$ and differentiability of $\hat{\sigma}_\delta$. Thus we have (\ref{E:finitenessclaim}), and as a consequence, the integral in \eqref{eq:g_n-solution} is well-defined. 

To verify equation \eqref{eq:g_n-solution} we first note that, since $g_\delta$ and $\hat{\sigma}_\delta$ are smooth and $X$ is H\"older continuous of order $\alpha>\frac12$, Proposition \ref{pro:matrix-ito} yields
$$
g_\delta(X_t) = g_\delta(\mathring{x}) + \int_0^t \hat{\sigma}_\delta(X_u)dX_u. 
$$
The integrals can be approximated by Riemann-Stieltjes sums 
$$
\int_0^t \hat{\sigma}_\delta(X_u)dX_u =\lim_{m\to \infty} \sum_{j=1}^{N_m} \hat{\sigma}_\delta(X_{t_{j-1}^{(m)}})\cdot \left(X_{t_j^{(m)}}-X_{t_{j-1}^{(m)}}\right),
$$
where $(\pi_m)_m$ is a refining sequence of finite partitions of $[0,t]$ with subinterval endpoints $t_j^{(m)}$ and mesh $|\pi_m| = \max_j|t_j^{(m)}-t_{j-1}^{(m)}|$. In particular, for 
\[\Pi_X^m(u) = \sum_{j=1}^{N_m} \hat{\sigma}_\delta(X_{t_{j-1}^{(m)}})\mathbf{1}_{(t_{j-1}^{(m)},t_j^{(m)}]}(u),\quad u\in [0,t],\] 
we have
\begin{equation}
\label{eq:RS-approx-2beta}
\lim_{m\to \infty}\left\| \hat{\sigma}_\delta(X_\cdot) - \Pi_X^m(\cdot)\right\|_{W^{\beta,1}_0(0,t,\mathbb{R}^{n\times n})} = 0,
\end{equation}
\begin{equation}
\label{eq:RS-approx-pointwise}
\sup_m \left\|\hat{\sigma}_\delta(X_u) - \Pi_X^m(u)\right\|_{L^\infty(0,t,\mathbb{R}^{n\times n})}<+\infty\quad\text{ and }\quad   \lim_{m\to \infty}|\hat{\sigma}_\delta(X_u) - \Pi_X^m(u)| = 0, \quad u\in[0,t],
\end{equation}
which can be seen similarly as in \cite[Theorem 4.1.1]{Zahle98}. Since 
$X$ solves \eqref{eq:SDE},
 we obtain
 \begin{equation}\label{eq21}
 \int_0^t \hat{\sigma}_\delta(X_u)dX_u = \lim\limits_{m \to \infty} \sum_{j=1}^{N_m} \int_{t_{j-1}^{(m)}}^{t_j^{(m)}}\hat{\sigma}_\delta(X_{t_{j-1}^{(m)}})\sigma(X_u)dY_u = \lim\limits_{m \to \infty} \int^t_0  \Pi^m_X(u)\sigma(X_u)dY_u.
\end{equation}
Since $\sigma(X)\in W^{\beta,1}_0(0,t,\mathbb{R}^{n\times n})\cap L^\infty(0,t,\mathbb{R}^{n\times n})$ we may, by 
\eqref{eq:RS-approx-2beta} and \eqref{eq:RS-approx-pointwise}, apply Lemma \ref{lem:RS-convergence} with 
\[A_m(u)=\hat{\sigma}_\delta(X_u) - \Pi_X^m(u)\quad \text{and} \quad  B(u)=\sigma(X_u),\quad u\in [0,t],\] 
to obtain
\[\lim_{m\to \infty} \left\|\hat{\sigma}_\delta(X_\cdot)\sigma(X_\cdot)- \Pi^m_X(\cdot)\sigma(X_\cdot)\right\|_{W^{\beta,1}_0(0,t,\mathbb{R}^{n\times n})} = 0.\]
Then, by (\ref{eq:ZSIntegralBound}) and (\ref{eq21}),
\[\int_0^t \hat{\sigma}_\delta(X_u)dX_u = \int_0^t \hat{\sigma}_\delta(X_u)\sigma(X_u)dY_u\]
concluding the proof.
\end{proof}

Theorem \ref{thm:DE-uniqueness} will follow from (\ref{eq:g_n-solution}) as $\delta \to 0$ if taking limits can be justified. To provide this justification we note that variability w.r.t. $\sigma$ implies variability w.r.t. $\hat{\sigma}$.

\begin{proposition}
\label{prop:SV-sigma-hat} Suppose that $\sigma=(\sigma_{jk})_{1\leq j,k\leq n}$ satisfies Assumption \ref{A:main} and let 
$\hat{\sigma}$ be as in Lemma \ref{L:niceinverse}. If for $s\in(0,1)$ and $p\in[1,+\infty]$ we have $X\in V(\sigma,s,p)$, then also $X\in V(\hat{\sigma},s,p)$.
\end{proposition}
\begin{proof}
Recall (\ref{E:viewascompo}), i.e., that each component $\hat{\sigma}_{ij}$ of $\hat{\sigma}$ can be written in the form 
$\hat{\sigma}_{ij}=\Phi\circ \sigma$ 
with a suitable function $\Phi\in \mathcal{C}^1(\mathbb{R}^{n\times n})$ (depending on $i$ and $j$) with $\Phi(0)=0$ and bounded gradient. Consequently $X\in V(\hat{\sigma}_{ij},s,p)$ for each $i$ and $j$ by Lemma \ref{L:Lipschitzcomp}, and this means that  $X\in V(\hat{\sigma},s,p)$.
\end{proof}

We prove Theorem \ref{thm:DE-uniqueness}.
\begin{proof}[Proof of Theorem \ref{thm:DE-uniqueness}]
Since by Proposition \ref{prop:SV-sigma-hat} the path $X$ is $(s,1)$-variable w.r.t. $\hat{\sigma}$, we have
$X_u \in \mathbb{R}^n \setminus S_{\hat{\sigma}}$ for $\mathcal{L}^1$-a.e. $u\in [0,t]$, and, together with (\ref{E:convsigmahat}),
\[\lim_{\delta \to 0}|\hat{\sigma}_\delta(X_u)\sigma(X_u) - I| =0\]
for such $u$. Since also $\hat{\sigma}(X_{\cdot})\in L^\infty(0,t,\mathbb{R}^{n\times n})$ by (\ref{E:Linftybound}), 
dominated convergence, together with the boundedness of $\sigma$ and $\hat\sigma$, implies
\[\lim_{\delta\to 0}\int_0^t \frac{|\hat{\sigma}_\delta(X_u)\sigma(X_u) - I|}{u^{\beta}}du = 0.\]
For the Gagliardo seminorm we first note that 
\begin{align*}
| \hat{\sigma}_\delta(X_u)\sigma(X_u) - \hat{\sigma}_\delta(X_r)\sigma(X_r)| 
\leq & \left\|\hat{\sigma}(X_{\cdot})\right\|_{L^\infty(0,t,\mathbb{R}^{n\times n})} |\sigma(X_u) - \sigma(X_r)|\notag\\
 &+   \left\|\sigma(X_{\cdot})\right\|_{L^\infty(0,t,\mathbb{R}^{n\times n})} |\hat{\sigma}_\delta(X_u) - \hat{\sigma}_\delta(X_r)|.
\end{align*}
Since we have
\[[\sigma(X_{\cdot})]_{\beta,1}\leq c[X]_{\alpha,\infty}^s \sum_{i,j=1}^n \left\|U^{1-s}\left\|D\hat{\sigma}_{ij}\right\|\right\|_{L^1(X)}\]
by (\ref{E:basicest}) in Proposition \ref{P:basicest} and, reasoning as in (\ref{E:stereotype}), also 
\[[\sigma_\delta (X_{\cdot})]_{\beta,1}\leq c[X]_{\alpha,\infty}^s \sum_{i,j=1}^n \left\|U^{1-s}\left\|D\hat{\sigma}_{ij}\right\|\right\|_{L^1(X)}\]
with a constant $c>0$ independent of $\delta$, we can again use dominated convergence to conclude that 
\[\lim_{\delta\to 0}
\int_0^t \int_0^t \frac{|\hat{\sigma}_\delta(X_u)\sigma(X_u) - \hat{\sigma}_\delta(X_r)\sigma(X_r)|}{|u-r|^{\beta+1}}dudr= 0.\]
This shows that 
\[\lim_{\delta \to 0}\left\|\hat{\sigma}_\delta(X_{\cdot})\sigma(X_{\cdot}) - I\right\|_{W^{\beta,1}_0(0,t,\mathbb{R}^{n\times n})} =0,\]
and using (\ref{eq:ZSIntegralBound}) we obtain
\[\lim_{\delta\to 0}\int_0^t \hat{\sigma}_\delta(X_u)\sigma(X_u)dY_u = \int_0^t dY_u = Y_t - Y_0.\]
Taking into account (\ref{E:convg}) and (\ref{eq:g_n-solution}), we arrive at  
\begin{equation}
\label{eq:g-X-is-Y}
g(X_t) - g(\mathring{x}) = Y_t- Y_0, \quad t\in [0,T].
\end{equation}
Now suppose that $X$ and $\widetilde{X}$ are two solutions. By \eqref{eq:g-X-is-Y} we must have 
\[g(X_t) = g(\widetilde{X}_t), \quad t\in [0,T].\]
Since $g$ is invertible with inverse $f$ satisfying \eqref{eq:deterministic-DE} (cf. Proof of Theorem \ref{T:exjoint}), this implies that 
\[X_t = \widetilde{X}_t=f(Y_t), \quad t\in [0,T].\]
\end{proof}

\appendix

\section{Convolution of Riesz kernels}\label{S:standard}

Statement (i) in the following proposition is well known, see \cite[Lemma 25.2 and (25.38)]{SKM}, \cite[Theorem 1.15 in Chapter I]{Landkof} or \cite[Chapter V, Section 1.1]{Stein70}, and it can also be obtained by a quick subordination argument. Statement (ii) is a direct consequence of (i) and Fubini's theorem.

\begin{proposition}\label{P:convolution}
Let $\gamma_1,\gamma_2>0$ and $\gamma_1+\gamma_2<n$. 
\begin{enumerate}
\item[(i)] We have
\[c_{\gamma_1}c_{\gamma_2}\int_{\mathbb{R}^n}|x|^{-n+\gamma_1}|x-y|^{-n+\gamma_2}dx=c_{\gamma_1+\gamma_2}|y|^{-n+\gamma_1+\gamma_2},\quad y\in\mathbb{R}^n\setminus \{0\}.\]
\item[(ii)] If  $\nu_1$ and $\nu_2$ are finite nonnegative Borel measures with compact support, then 
\[\int_{\mathbb{R}^n}U^{\gamma_1}\nu_1(x)U^{\gamma_2}\nu_2(x)dx=\int_{\mathbb{R}^n} U^{\gamma_2}U^{\gamma_1}\nu_1(x)\nu_2(dx)=\int_{\mathbb{R}^n}U^{\gamma_1+\gamma_2}\nu_1(x)\nu_2(dx).\]
\end{enumerate}
\end{proposition}

\section{Mollification results}\label{S:smooth_approximation}

We collect some useful known approximation results used in the main text. 

We begin with an approximation lemma for Riesz potentials that is a slight variant of \cite[Section I.3, Theorem 1.11 and its proof]{Landkof}. As usual we say that $(\eta_\varepsilon)_{\varepsilon>0}$ is a (radially symmetric) \emph{mollifier} if $\eta\in \mathcal{C}_c^\infty(\mathbb{R}^n)$ is a nonnegative radial function, compactly supported inside the unit ball and such that $\int_{\mathbb{R}^n} \eta(x)dx=1$, and $\eta_\varepsilon(x):=\varepsilon^{-n}\eta(\varepsilon^{-1}x)$, $x\in\mathbb{R}^n$, for any $\varepsilon>0$, cf. \cite[p. 41]{AFP}. We say that a mollifier $(\eta_\varepsilon)_{\varepsilon>0}$ is a \emph{flat mollifier} if for all $x\in\mathbb{R}^n$ with $|x|\leq \frac12$ we have $\eta(x)= c_\eta$ with a suitable constant $c_\eta>0$. To have $\eta$ constant in a small ball around the origin is useful to quickly see the following. 

\begin{lemma}\label{L:Landkof}
Let $(\eta_\varepsilon)_{\varepsilon>0}$ be a flat mollifier, let $\nu$ be a nonnegative Borel measure on $\mathbb{R}^n$, and $0<\gamma<n$.  Then:
\begin{enumerate}
\item[(i)] There is a constant $c(n,\gamma)>0$ depending only on $n$ and $\gamma$ such that for any $\varepsilon>0$ and all $x\in\mathbb{R}^n$ we have
\[U^\gamma(\nu\ast \eta_\varepsilon)(x)\leq c(n,\gamma) U^\gamma\nu(x).\]
\item[(ii)] For any $x\in\mathbb{R}^n$ we have 
\[\lim_{\varepsilon\to 0} U^\gamma (\nu\ast \eta_\varepsilon)(x)=U^\gamma\nu(x).\]
\end{enumerate}
\end{lemma}

The proof is as in \cite[Section I.3, Theorem 1.11]{Landkof}, but since we use a slightly different mollifier, we repeat the short arguments for (i) for convenience.

\begin{proof}
We have 
\begin{align}
U^\gamma(\nu\ast\eta_\varepsilon)(x)&=c_\gamma\int_{\mathbb{R}^n} \eta_\varepsilon(z-x) \int_{\mathbb{R}^n}\frac{\nu(dy)}{|z-y|^{n-\gamma}}\:dz\notag\\
&=c_\gamma\int_{\mathbb{R}^n}\frac{1}{|x-y|^{n-\gamma}}\left(\int_{\mathbb{R}^n}\eta_\varepsilon(z-x)\;\frac{|x-y|^{n-\gamma}}{|z-y|^{n-\gamma}}\:dz\right)\nu(dy).\notag
\end{align}
The inner integral is bounded by $\Phi(\frac{x}{\varepsilon},\frac{y}{\varepsilon})$, where 
\[\Phi(x,y):=|x-y|^{n-\gamma}\:\int_{|\zeta-x|\leq \frac{1}{2}}\frac{d\zeta}{|\zeta-y|^{n-\gamma}}.\]
The function $\Phi$ defines a continuous function of $|x-y|$ which is zero if $|x-y|=0$ and tends to $1$ for $|x-y|\to +\infty$, and denoting the maximum of this function by $c(n,\gamma)$, we obtain (i). For (ii) one can follow the proof in \cite[p. 73]{Landkof}.
\end{proof}

The next lemma contain versions of a statement from \cite[Theorem 2.2]{AFP}. Given a Borel set $E\subset \mathbb{R}^n$ and $\varepsilon>0$, write $E_\varepsilon:=\{x\in\mathbb{R}^n: \dist(E,x)<\varepsilon\}$.

\begin{lemma}\label{L:neededbound}
Let $\varphi\in BV_{\loc}(\mathbb{R}^n)$ and let $(\eta_\varepsilon)_{\varepsilon}$ be a mollifier. Then for any Borel set $E\subset \mathbb{R}^n$ and any Borel function $\psi:\mathbb{R}^n \to [0,+\infty]$ we have 
\begin{equation}\label{E:neededbound}
\int_{\mathbb{R}^n} \psi(y)\left\|D(\varphi\ast \eta_\varepsilon)\right\||_E(dy)\leq \int_{\mathbb{R}^n}\psi(y)\left(\left\|D\varphi\right\||_{E_\varepsilon}\ast \eta_{\varepsilon}\right)(y)dy.
\end{equation}
In particular, 
\[\int_E|\nabla(\varphi\ast \eta_\varepsilon)|(y)dy\leq \left\|D\varphi\right\|(E_\varepsilon).\]  
\end{lemma}
The proof basically follows \cite[Theorem 2.2]{AFP}.
\begin{proof}
Since $\left\|D(\varphi\ast \eta_\varepsilon)\right\|=|\nabla(\varphi\ast \eta_\varepsilon)|\cdot \mathcal{L}^n$ and 
$\nabla(\varphi\ast \eta_\varepsilon)=(D\varphi)\ast \eta_\varepsilon$, \cite[Proposition 3.2]{AFP}, the left hand side of inequality (\ref{E:neededbound}) is seen to equal
\begin{align}
\int_{\mathbb{R}^n}\psi(y)\mathbf{1}_E(y)\left|\int_{\mathbb{R}^n} \eta_\varepsilon (y-z)D\varphi(dz)\right|dy &\leq \int_{\mathbb{R}^n}\int_{\mathbb{R}^n}\psi(y)\mathbf{1}_E(y)\eta_\varepsilon(y-z)dy\left\|D\varphi\right\|(dz)\notag\\
&\leq \int_{\mathbb{R}^n}\int_{\mathbb{R}^n}\psi(y)\eta_\varepsilon(y-z)dy\mathbf{1}_{E_\varepsilon}(z)\left\|D\varphi\right\|(dz),\notag
\end{align} 
where we have used Fubini's theorem. Another application of the latter shows that the last integral equals the right hand side of (\ref{E:neededbound}). The case $\psi\equiv 1$ yields the second statement because the mollifier has integral equal to one.
\end{proof}

We obtain the following consequence for potentials.
\begin{corollary}\label{C:mollypot}
Let $\varphi\in BV_{\loc}(\mathbb{R}^n)$, $0<\gamma<n$, let $K\subset \mathbb{R}^n$ be a compact set and $(\eta_\varepsilon)_{\varepsilon>0}$ a symmetric flat mollifier. Then
\[U^\gamma(\left\|D(\varphi\ast\eta_\varepsilon)\right\||_K)(x)\leq c(n,\gamma)\:U^\gamma(\left\|D\varphi\right\||_{K_\varepsilon})(x)\]
for any $x\in\mathbb{R}^n$ and $\varepsilon >0$.
\end{corollary}
\begin{proof}
Given $x\in\mathbb{R}^n$ write $\psi_x(y):=|x-y|^{-n+\gamma}$. Then 
\begin{align*}
&U^\gamma(\left\|D(\varphi\ast \eta_\varepsilon)\right\||_K)(x)=\int_{\mathbb{R}^n} \psi_x(y)\left\|D(\varphi\ast \eta_\varepsilon)\right\||_K(dy)\\
\leq &\int_{\mathbb{R}^n}\psi_x(y)\left(\left\|D\varphi\right\||_{K_\varepsilon}\ast \eta_{\varepsilon}\right)(y)dy=U^\gamma(\left\|D\varphi\right\||_{K_\varepsilon}\ast \eta_\varepsilon)(x)
\end{align*}
by Lemma \ref{L:neededbound}, and an application of Lemma \ref{L:Landkof} yields the desired bound.
\end{proof}

Let $(\eta_\varepsilon)_{\varepsilon>0}$ be a mollifier. Recall that if $\varphi\in L^1_{\loc}(\mathbb{R}^n)$ and $\widetilde{\varphi}$ is a Lebesgue representative of $\varphi$, then 
\begin{equation}\label{E:convundermolly}
\lim_{\varepsilon\to 0}\varphi\ast \eta_{\varepsilon}(x)=\widetilde{\varphi}(x),\quad x\in \mathbb{R}^n\setminus S_{\varphi},
\end{equation}
as shown in \cite[Proposition 3.64 (b)]{AFP}.

\begin{lemma}\label{lem:mollifier-result-new}
Let $F\in W^{1,1}_{\loc}(\mathbb{R}^n)$, such that $\sigma_k:=\partial_k F\in BV_{\loc}(\mathbb{R}^n)$ for $k=1,\ldots,n$. Then $S_F\subset\bigcup_{k=1}^n S_{\sigma_k}$ for $k=1,\ldots,n$. If $(\eta_\varepsilon)_{\varepsilon>0}$ is a mollifier, then
\begin{enumerate}[(i)]
\item $\lim_{\varepsilon \to 0}F\ast\eta_{\varepsilon}= F$ pointwise on $\mathbb{R}^n\setminus S_F$,
\item $\lim_{\varepsilon\to 0}\sigma_k\ast\eta_\varepsilon=\lim_{\varepsilon\to 0}F\ast \partial_k\eta_{\varepsilon}= \sigma_k$ pointwise on $\mathbb{R}^n\setminus S_{\sigma_k}$.
\end{enumerate}
\end{lemma}
\begin{proof}
We show that $\bigcap_{k=1}^n(\mathbb{R}^n\setminus S_{\sigma_k})\subset (\mathbb{R}^n\setminus S_F)$. Once this is shown, (i) and (ii) follow using (\ref{E:convundermolly}) and $F\ast\partial_k\eta_\varepsilon=\partial_k (F\ast\eta_\varepsilon)=\sigma_k\ast\eta_\varepsilon$.

Let $x\in \bigcap_{k=1}^n (\mathbb{R}^n\setminus S_{\sigma_k})$. Then, writing $\sigma:=(\sigma_1,\ldots,\sigma_n)$, we see that $x\in \mathbb{R}^n\setminus S_{|\sigma|}$. Let $r>0$. For any $\varepsilon >0$ there exists $r>\delta>0$ such that
\[\begin{split}&\int_0^r t^{-n}\int_{B(x,t)}|\sigma(y)| dy dt\\
=&\int_\delta^r t^{-n}\int_{B(x,t)}|\sigma(y)| dy dt+c(n)\int_0^\delta\frac{1}{\mathcal{L}^n(B(x,t))}\int_{B(x,t)}|\sigma(y)| dy dt\\
\le&\frac{(r-\delta)}{\delta^n}\|\sigma\|_{L^1(B(x,r))}+\delta c(n)\left(\lambda_{|\sigma|}(x)+\varepsilon\right)< + \infty.\end{split}\]
Moreover, for sufficiently small $\varrho>0$,
\[\frac{1}{\varrho^{n-1}}\int_{B(x,\varrho)}|\sigma(y)| dy=c(n)\frac{\varrho}{\mathcal{L}^n(B(x,\varrho))}\int_{B(x,\varrho)}|\sigma(y)| dy\le c(n)\varrho\left(\lambda_{|\sigma|}(x)+\varepsilon\right)\]
that converges to zero as as $\varrho\searrow 0$. Therefore, by  \cite[Remark 3.82 and Exercise 3.14]{AFP}, we have $x\in \mathbb{R}^n\setminus S_F$.  
\end{proof}

\section{Some properties of maximal functions}\label{S:maximalfcts}

We record a version of \cite[Corollary 4.3]{AK}. See also \cite{HLNT, KinnunenSaksman}.  For any given measure $\nu$ on $(\mathbb{R}^n,\mathcal{B}(\mathbb{R}^d))$, any $\gamma\in [0,1)$, and any $R\in (0,+\infty]$, let 
\begin{equation}\label{E:maximalfct}
\mathcal{M}_{\gamma,R} \nu(x):=\sup_{0<r<R}r^{\gamma-n}\:\nu(B(x,r)), \quad x\in \mathbb{R}^n,
\end{equation}
denote the \emph{(truncated) fractional Hardy-Littlewood maximal function} of $\nu$ of order $\gamma$. In the case $R=+\infty$ we use the standard notation $\mathcal{M}_\gamma\nu=\mathcal{M}_{\gamma,+\infty}\nu$.

\begin{proposition}\label{P:meanvalue}
Let $\varphi \in BV_{\loc}(\mathbb{R}^n)$ and $s\in (0,1]$. Then there is a constant  $c=c(n,s)>0$ such that for all Lebesgue points $x,y\in \mathbb{R}^n$ of $\varphi$  we have
\[|\varphi(x)-\varphi(y)|\leq c|x-y|^s\left(\mathcal{M}_{1-s,4|x-y|}\left\|D\varphi\right\|(x)+\mathcal{M}_{1-s,4|x-y|}\left\|D\varphi\right\|(y)\right).\]
\end{proposition}

In order to prove the claim, we quote \cite[Lemma 4.1 and its proof]{AK}. For any $0<s<+\infty$, any $R>0$, and any locally integrable function $f:\mathbb{R}^n\to\mathbb{R}$, the fractional sharp maximal function $f_{s,R}^\#$ of $f$ is defined by 
\[f_{s,R}^\#(x):=\sup_{0<r<R} r^{-s-n}\int_{B(x,r)}|f-f_{B(x,r)}| dy.\]

\begin{lemma}
Let $f:\mathbb{R}^n\to \mathbb{R}$ be locally integrable and $0<s<+\infty$. Then there is a constant $c(n,s)>0$ such that for all Lebesgue points $x,y\in \mathbb{R}^n$ of $f$ we have
\[|f(x)-f(y)|\leq c(n,s)|x-y|^s \left(f_{s,4|x-y|}^\#(x)+ f_{s,4|x-y|}^\#(y)\right).\]
\end{lemma} 

We prove Proposition \ref{P:meanvalue}.

\begin{proof}
For any $x\in\mathbb{R}^n$ and $r>0$ we have the $1$-$1$-Poincar\'e inequality 
\[\int_{B(x,r)}|f-f_{B(x,r)}| dy\leq c\:r\:\left\|Df\right\|(B(x,r)),\]
see \cite[Remark 3.45]{AFP}. Multiplying both sides by $r^{-s-n}$ and taking suprema we obtain
\[\sup_{0<r<R}r^{-s-n}\:\int_{B(x,r)}|f-f_{B(x,r)}| dy\leq c\:\mathcal{M}_{1-s,R}\left\|Df\right\|(x).\]
By \cite[Lemma 4.1]{AK} there is a constant $c(n,s)>0$ such that for all Lebesgue points $x,y\in \mathbb{R}^n$ of $f$ we have
\[|f(x)-f(y)|\leq c(n,s)|x-y|^s \left(f_{s,4|x-y|}^\#(x)+ f_{s,4|x-y|}^\#(y)\right).\]
\end{proof}

It is trivial to see that for any Borel measure $\nu$ on $\mathbb{R}^n$, any $\gamma\in (0,1)$, and any $R\in (0,+\infty]$ we have
\begin{equation}\label{E:trivialbound}
\mathcal{M}_{\gamma,R}\nu(x)\leq c\:U^\gamma\nu(x),\quad x\in\mathbb{R}^n
\end{equation}
with a constant $c>0$ depending only on $n$ and $\gamma$. Together with Proposition \ref{P:meanvalue} we obtain the following immediate consequence.
\begin{corollary}\label{C:justHoelder}
Let $\varphi \in BV_{\loc}(\mathbb{R}^n)$ and $s\in (0,1)$. If 
\[\sup_{x\in\mathbb{R}^n}U^{1-s}\left\|D\varphi\right\|(x)<+\infty,\]
then $\varphi$ has a Borel version that is H\"older continuous of order $s$, and any Lebesgue representative of $\varphi$ coincides with this version on the Lebesgue set of $\varphi$.
\end{corollary}

\section{Elements of fractional calculus}\label{S:fraccalc}

We briefly recall the definitions of fractional integrals and derivatives, \cite{SKM}, and the generalized Lebesgue-Stieltjes integral introduced by Z\"ahle in \cite{Zahle98} and used in \cite{NualartRascanu, Zahle01}. 

For fixed $T<\infty$, the fractional left and right Riemann--Liouville integrals of order $\theta > 0$ of a function $f \in L_1(0,T)$ are denoted by
\[
 I^\theta_{0+}f(t)
 \weq \frac{1}{\Gamma(\theta)} \int_0^t \frac{f(s)}{(t-s)^{1-\theta}}  ds
\]
and
\[
 I^\theta_{T-}f(t)
 \weq \frac{(-1)^{-\theta}}{\Gamma(\theta)} \int_t^T \frac{f(s)}{(t-s)^{1-\theta}}  ds.
\]
The integral operators $I^\theta_{0+}, I^\theta_{T-}: L_1(0,T) \to L_1(0,T)$ are linear and one-to-one, and the inverse operators, denoted by $I^{-\theta}_{0+} = (I^\theta_{0+})^{-1}$ and $I^{-\theta}_{T-} = (I^\theta_{T-})^{-1}$, are known as (left and right sided) Riemann--Liouville fractional derivatives. Furthermore, for any $\theta \in (0,1)$ and for any $f \in I^\theta_{0+}(L_1(0,T))$ and $g \in I^\theta_{T-}(L_1(0,T))$, the (left and right sided) Weyl--Marchaud derivatives 
\begin{equation}\label{E:WMforward}
 D_{0+}^\theta f(t)
 \weq \frac{1}{\Gamma(1-\theta)}\left( \frac{f(t)}{t^\theta} + \theta \int_0^t \frac{f(t)-f(s)}{(t-s)^{\theta+1}}  ds \right)
\end{equation}
and
\begin{equation}\label{E:WMbackward}
 D_{T-}^\theta g(t)
 \weq \frac{(-1)^\theta}{\Gamma(1-\theta)}\left( \frac{g(t)}{(T-t)^\theta} + \theta \int_t^T \frac{g(t)-g(s)}{(s-t)^{\theta+1}}  ds \right)
\end{equation}
are well defined, and coincide with the Riemann--Liouville derivatives by relations $D_{0+}^\theta f(t) = I^{-\theta}_{0+} f(t)$ and $D_{T-}^\theta g(t) = I^{-\theta}_{T-} g(t)$ for almost every $t \in (0,T)$.

For functions $f$ and $g$ such that the limits $f(0+), g(0+), g(T-)$ exist in $\mathbb{R}$, set $f_{0+}(t) = f(t) - f(0+)$ and $g_{T-}(t) = g(t) - g(T-)$. If $f_{0+} \in I^\theta_{0+}(L_p(0,T))$ and $g_{T-} \in I^{1-\theta}_{T-}(L_q(0,T))$ for some $\theta \in [0,1]$ and $p,q \in [1,\infty]$ such that $1/p+1/q =1$,
the fractional version of the Stieltjes integral introduced by Z\"ahle \cite{Zahle98} is defined by
\begin{equation}
 \label{eq:ZSIntegral}
 \begin{aligned}
 \int_0^T f_t  dg_t
 &:\weq (-1)^\theta \int_0^T D^{\theta}_{0+} (f-f(0+))(t) \, D^{1-\theta}_{T-} (g-g(T-))(t)  dt \\
 &\qquad \qquad + f(0+)(g(T-)-g(0+)),
 \end{aligned}
\end{equation}
It can be shown that here the right side does not depend on $\theta$. Moreover, if $\theta p<1$, then  
\begin{equation}
 \label{eq:ZSIntegralSimple}
\int_0^T f_t  dg_t = (-1)^\theta \int_0^T D^{\theta}_{0+} f(t) \, D^{1-\theta}_{T-} (g-g(T-))(t) dt
\end{equation}
which coincides with Definition \ref{def:integral}.


\begin{thebibliography}{100}

\bibitem{AK}
D.~Aalto and J.~Kinnunen.
\newblock Maximal functions in {S}obolev spaces.
\newblock In {\em Sobolev spaces in mathematics. {I}}, volume~8 of {\em Int.
  Math. Ser. (N. Y.)}, pages 25--67. Springer, New York, 2009.

\bibitem{AH96}
D.~R. Adams and L.~I. Hedberg.
\newblock {\em Function spaces and potential theory}, volume 314 of {\em
  Grundlehren der Mathematischen Wissenschaften [Fundamental Principles of
  Mathematical Sciences]}.
\newblock Springer-Verlag, Berlin, 1996.

\bibitem{AdamsSobo}
R.~A. Adams and J.~J.~F. Fournier.
\newblock {\em Sobolev spaces}, volume 140 of {\em Pure and Applied Mathematics
  (Amsterdam)}.
\newblock Elsevier/Academic Press, Amsterdam, second edition, 2003.

\bibitem{AmbrosioDalMaso}
L.~Ambrosio and G.~Dal~Maso.
\newblock A general chain rule for distributional derivatives.
\newblock {\em Proc. Amer. Math. Soc.}, 108(3):691--702, 1990.

\bibitem{AFP}
L.~Ambrosio, N.~Fusco, and D.~Pallara.
\newblock {\em Functions of bounded variation and free discontinuity problems}.
\newblock Oxford Mathematical Monographs. The Clarendon Press, Oxford
  University Press, New York, 2000.

\bibitem{Auzinger}
W.~Auzinger.
\newblock Sectorial operators and normalized numerical range.
\newblock {\em Appl. Numer. Math.}, 45(4):367--388, 2003.

\bibitem{Ayache}
A.~Ayache, D.~Wu, and Y.~Xiao.
\newblock Joint continuity of the local times of fractional {B}rownian sheets.
\newblock {\em Ann. Inst. Henri Poincar\'{e} Probab. Stat.}, 44(4):727--748,
  2008.

\bibitem{Barchiesietal}
M.~Barchiesi, D.~Henao, and C.~Mora-Corral.
\newblock Local invertibility in {S}obolev spaces with applications to nematic
  elastomers and magnetoelasticity.
\newblock {\em Arch. Ration. Mech. Anal.}, 224(2):743--816, 2017.

\bibitem{BarreiraPesinSchmeling}
L.~Barreira, Y.~Pesin, and J.~Schmeling.
\newblock Dimension and product structure of hyperbolic measures.
\newblock {\em Ann. of Math. (2)}, 149(3):755--783, 1999.

\bibitem{BSV2008}
C.~Bender, T.~Sottinen, and E.~Valkeila.
\newblock {P}ricing by hedging and no-arbitrage beyond semimartingales.
\newblock {\em Finance Stoch.}, 12:441--468, 2008.

\bibitem{BV2016}
C.~Bender and L.~Viitasaari.
\newblock {\em Fractional Brownian Motion in Financial Modeling}, pages 1--5.
\newblock John Wiley and Sons, Ltd, 2016.

\bibitem{Berman69}
S.~M. Berman.
\newblock Local times and sample function properties of stationary {G}aussian
  processes.
\newblock {\em Trans. Amer. Math. Soc.}, 137:277--299, 1969.

\bibitem{Berman70}
S.~M. Berman.
\newblock Gaussian processes with stationary increments: {L}ocal times and
  sample function properties.
\newblock {\em Ann. Math. Statist.}, 41:1260--1272, 1970.

\bibitem{Berman73}
S.~M. Berman.
\newblock Local nondeterminism and local times of {G}aussian processes.
\newblock {\em Bull. Amer. Math. Soc.}, 79:475--477, 1973.

\bibitem{Bertoin}
J.~Bertoin.
\newblock Sur la mesure d'occupation d'une classe de fonctions self-affines.
\newblock {\em Japan J. Appl. Math.}, 5(3):431--439, 1988.

\bibitem{BevFlan14}
A.~Bevilacqua and F.~Flandoli.
\newblock An occupation time formula for semimartingales in {$\mathbb{R}^N$}.
\newblock {\em Stochastic Process. Appl.}, 124(10):3342--3361, 2014.

\bibitem{BliedtnerHansen}
J.~Bliedtner and W.~Hansen.
\newblock {\em Potential theory}.
\newblock Universitext. Springer-Verlag, Berlin, 1986.
\newblock An analytic and probabilistic approach to balayage.

\bibitem{CatellierGubinelli}
R.~Catellier and M.~Gubinelli.
\newblock Averaging along irregular curves and regularisation of {ODE}s.
\newblock {\em Stochastic Process. Appl.}, 126(8):2323--2366, 2016.

\bibitem{Chen}
Z.~Chen.
\newblock {\em On Pathwise Stochastic Integration of Processes with Unbounded
  Power Variation}.
\newblock PhD thesis, Aalto University, 2016.

\bibitem{CLV16}
Z.~Chen, L.~Leskel\"{a}, and L.~Viitasaari.
\newblock Pathwise {S}tieltjes integrals of discontinuously evaluated
  stochastic processes.
\newblock {\em Stochastic Process. Appl.}, 129(8):2723--2757, 2019.

\bibitem{ChenSong}
Z.-Q. Chen and R.~Song.
\newblock Hardy inequality for censored stable processes.
\newblock {\em Tohoku Math. J. (2)}, 55(3):439--450, 2003.

\bibitem{Ch2001}
P.~Cheridito.
\newblock {M}ixed fractional {B}rownian motion.
\newblock {\em Bernoulli}, 7(6):913--934, 2001.

\bibitem{CrippadeLellis}
G.~Crippa and C.~De~Lellis.
\newblock Estimates and regularity results for the {D}i{P}erna-{L}ions flow.
\newblock {\em J. Reine Angew. Math.}, 616:15--46, 2008.

\bibitem{Cristea}
M.~Cristea.
\newblock A note on the global injectivity of some light {S}obolev mappings.
\newblock {\em J. Math. Pures Appl. (9)}, 128:213--224, 2019.

\bibitem{Davie}
A.~M. Davie.
\newblock Uniqueness of solutions of stochastic differential equations.
\newblock {\em Int. Math. Res. Not. IMRN}, (24):Art. ID rnm124, 26, 2007.

\bibitem{DiNezza}
E.~Di~Nezza, G.~Palatucci, and E.~Valdinoci.
\newblock Hitchhiker's guide to the fractional {S}obolev spaces.
\newblock {\em Bull. Sci. Math.}, 136(5):521--573, 2012.

\bibitem{Doss}
H.~Doss.
\newblock Liens entre \'{e}quations diff\'{e}rentielles stochastiques et
  ordinaires.
\newblock {\em Ann. Inst. H. Poincar\'{e} Sect. B (N.S.)}, 13(2):99--125, 1977.

\bibitem{dyda}
B.~Dyda.
\newblock A fractional order {H}ardy inequality.
\newblock {\em Illinois J. Math.}, 48(2):575--588, 2004.

\bibitem{EngelbertSchmidt1}
H.~J. Engelbert and W.~Schmidt.
\newblock On one-dimensional stochastic differential equations with generalized
  drift.
\newblock In {\em Stochastic differential systems ({M}arseille-{L}uminy,
  1984)}, volume~69 of {\em Lect. Notes Control Inf. Sci.}, pages 143--155.
  Springer, Berlin, 1985.

\bibitem{EngelbertSchmidt2}
H.~J. Engelbert and W.~Schmidt.
\newblock On solutions of one-dimensional stochastic differential equations
  without drift.
\newblock {\em Z. Wahrsch. Verw. Gebiete}, 68(3):287--314, 1985.

\bibitem{Falconer}
K.~Falconer.
\newblock {\em Fractal geometry}.
\newblock John Wiley \& Sons, Ltd., Chichester, 1990.
\newblock Mathematical foundations and applications.

\bibitem{Federer}
H.~Federer.
\newblock {\em Geometric measure theory}.
\newblock Die Grundlehren der mathematischen Wissenschaften, Band 153.
  Springer-Verlag New York Inc., New York, 1969.

\bibitem{Figalli}
A.~Figalli.
\newblock Existence and uniqueness of martingale solutions for {SDE}s with
  rough or degenerate coefficients.
\newblock {\em J. Funct. Anal.}, 254(1):109--153, 2008.

\bibitem{Fiscellaetal}
A.~Fiscella, R.~Servadei, and E.~Valdinoci.
\newblock Density properties for fractional {S}obolev spaces.
\newblock {\em Ann. Acad. Sci. Fenn. Math.}, 40(1):235--253, 2015.

\bibitem{Flandoli}
F.~Flandoli.
\newblock {\em Random perturbation of {PDE}s and fluid dynamic models}, volume
  2015 of {\em Lecture Notes in Mathematics}.
\newblock Springer, Heidelberg, 2011.
\newblock Lectures from the 40th Probability Summer School held in Saint-Flour,
  2010, \'{E}cole d'\'{E}t\'{e} de Probabilit\'{e}s de Saint-Flour.
  [Saint-Flour Probability Summer School].

\bibitem{FlandoliGubinelliGiaquintaTortorelli}
F.~Flandoli, M.~Gubinelli, M.~Giaquinta, and V.~M. Tortorelli.
\newblock Stochastic currents.
\newblock {\em Stochastic Process. Appl.}, 115(9):1583--1601, 2005.

\bibitem{FlandoliGubinelliPriola}
F.~Flandoli, M.~Gubinelli, and E.~Priola.
\newblock Well-posedness of the transport equation by stochastic perturbation.
\newblock {\em Invent. Math.}, 180(1):1--53, 2010.

\bibitem{FlandoliGubinelliRusso}
F.~Flandoli, M.~Gubinelli, and F.~Russo.
\newblock On the regularity of stochastic currents, fractional {B}rownian
  motion and applications to a turbulence model.
\newblock {\em Ann. Inst. Henri Poincar\'{e} Probab. Stat.}, 45(2):545--576,
  2009.

\bibitem{FonsecaGangbobook}
I.~Fonseca and W.~Gangbo.
\newblock {\em Degree theory in analysis and applications}, volume~2 of {\em
  Oxford Lecture Series in Mathematics and its Applications}.
\newblock The Clarendon Press, Oxford University Press, New York, 1995.
\newblock Oxford Science Publications.

\bibitem{FonsecaGangbo}
I.~Fonseca and W.~Gangbo.
\newblock Local invertibility of {S}obolev functions.
\newblock {\em SIAM J. Math. Anal.}, 26(2):280--304, 1995.

\bibitem{GG20b}
L.~Galeati and M.~Gubinelli.
\newblock Prevalence of $\rho$-irregularity and related properties.
\newblock {\em Preprint}, pages 1--54, 2020.
\newblock \url{https://arxiv.org/abs/2004.00872}.

\bibitem{GG20a}
L.~Galeati and M.~Gubinelli.
\newblock Noiseless regularisation by noise.
\newblock {\em Rev. Mat. Iberoam.}, 38(2):433--502, 2022.

\bibitem{Johanna}
J.~Garz\'{o}n, J.~A. Le\'{o}n, and S.~Torres.
\newblock Fractional stochastic differential equation with discontinuous
  diffusion.
\newblock {\em Stoch. Anal. Appl.}, 35(6):1113--1123, 2017.

\bibitem{GemanHorowitz}
D.~Geman and J.~Horowitz.
\newblock Occupation densities.
\newblock {\em Ann. Probab.}, 8(1):1--67, 1980.

\bibitem{BG2018}
B.~Gess.
\newblock Regularization and well-posedness by noise for ordinary and partial
  differential equations.
\newblock In {\em Stochastic partial differential equations and related
  fields}, volume 229 of {\em Springer Proc. Math. Stat.}, pages 43--67.
  Springer, Cham, 2018.

\bibitem{GrahamHareRitter}
C.~C. Graham, K.~E. Hare, and D.~L. Ritter.
\newblock The size of {$L^p$}-improving measures.
\newblock {\em J. Funct. Anal.}, 84(2):472--495, 1989.

\bibitem{Gubinelli}
M.~Gubinelli.
\newblock Controlling rough paths.
\newblock {\em J. Funct. Anal.}, 216(1):86--140, 2004.

\bibitem{HarangPerkowski}
F.~A. Harang and N.~Perkowski.
\newblock {$C^\infty$}-regularization of {ODE}s perturbed by noise.
\newblock {\em Stoch. Dyn.}, 21(8):Paper No. 2140010, 1--29, 2021.

\bibitem{Doug}
D.~P. Hardin and E.~B. Saff.
\newblock Minimal {R}iesz energy point configurations for rectifiable
  {$d$}-dimensional manifolds.
\newblock {\em Adv. Math.}, 193(1):174--204, 2005.

\bibitem{HareRoginskaya}
K.~E. Hare and M.~Roginskaya.
\newblock The energy of signed measures.
\newblock {\em Proc. Amer. Math. Soc.}, 132(2):397--406, 2004.

\bibitem{HLNT}
T.~Heikkinen, J.~Lehrb\"{a}ck, J.~Nuutinen, and H.~Tuominen.
\newblock Fractional maximal functions in metric measure spaces.
\newblock {\em Anal. Geom. Metr. Spaces}, 1:147--162, 2013.

\bibitem{Heinonen}
J.~Heinonen.
\newblock {\em Lectures on {L}ipschitz analysis}, volume 100 of {\em Report.
  University of Jyv\"{a}skyl\"{a} Department of Mathematics and Statistics}.
\newblock University of Jyv\"{a}skyl\"{a}, Jyv\"{a}skyl\"{a}, 2005.

\bibitem{Hencl}
S.~Hencl.
\newblock Bilipschitz mappings with derivatives of bounded variation.
\newblock {\em Publ. Mat.}, 52(1):91--99, 2008.

\bibitem{HenclKoskela}
S.~Hencl and P.~Koskela.
\newblock {\em Lectures on mappings of finite distortion}, volume 2096 of {\em
  Lecture Notes in Mathematics}.
\newblock Springer, Cham, 2014.

\bibitem{HTV2022}
M.~Hinz, J.~M. T\"{o}lle, and L.~Viitasaari.
\newblock {S}obolev regularity of occupation measures and paths, variability
  and compositions.
\newblock {\em Electron. J. Probab.}, 27(73):1--29, 2022.

\bibitem{Horvath}
J.~Horv\'{a}th.
\newblock {\em Topological vector spaces and distributions. {V}ol. {I}}.
\newblock Addison-Wesley Publishing Co., Reading, Mass.-London-Don Mills, Ont.,
  1966.

\bibitem{JW84}
A.~Jonsson and H.~Wallin.
\newblock Function spaces on subsets of {${\bf R}^n$}.
\newblock {\em Math. Rep.}, 2(1), 1984.

\bibitem{Kahane}
J.-P. Kahane.
\newblock {\em Some random series of functions}, volume~5 of {\em Cambridge
  Studies in Advanced Mathematics}.
\newblock Cambridge University Press, Cambridge, second edition, 1985.

\bibitem{Kamae}
T.~Kamae.
\newblock A characterization of self-affine functions.
\newblock {\em Japan J. Appl. Math.}, 3(2):271--280, 1986.

\bibitem{KinnunenSaksman}
J.~Kinnunen and E.~Saksman.
\newblock Regularity of the fractional maximal function.
\newblock {\em Bull. London Math. Soc.}, 35(4):529--535, 2003.

\bibitem{Kono}
N.~K\^{o}no.
\newblock On self-affine functions.
\newblock {\em Japan J. Appl. Math.}, 3(2):259--269, 1986.

\bibitem{Kovalev}
L.~V. Kovalev and J.~Onninen.
\newblock On invertibility of {S}obolev mappings.
\newblock {\em J. Reine Angew. Math.}, 656:1--16, 2011.

\bibitem{KovalevOnninenRajala}
L.~V. Kovalev, J.~Onninen, and K.~Rajala.
\newblock Invertibility of {S}obolev mappings under minimal hypotheses.
\newblock {\em Ann. Inst. H. Poincar\'{e} Anal. Non Lin\'{e}aire},
  27(2):517--528, 2010.

\bibitem{KrylovRoeckner}
N.~V. Krylov and M.~R\"{o}ckner.
\newblock Strong solutions of stochastic equations with singular time dependent
  drift.
\newblock {\em Probab. Theory Related Fields}, 131(2):154--196, 2005.

\bibitem{LambertiVespri}
P.~D. Lamberti and V.~Vespri.
\newblock Remarks on {S}obolev-{M}orrey-{C}ampanato spaces defined on
  {$C^{0,\gamma}$} domains.
\newblock {\em Eurasian Math. J.}, 10(4):47--62, 2019.

\bibitem{Lamperti}
J.~Lamperti.
\newblock A simple construction of certain diffusion porcesses.
\newblock {\em J. Math. Kyoto Univ.}, 4:161--170, 1964.

\bibitem{Landkof}
N.~S. Landkof.
\newblock {\em Foundations of modern potential theory}.
\newblock Springer-Verlag, New York-Heidelberg, 1972.
\newblock Translated from the Russian by A. P. Doohovskoy, Die Grundlehren der
  mathematischen Wissenschaften, Band 180.

\bibitem{LeGall}
J.-F. Le~Gall.
\newblock Applications du temps local aux \'{e}quations diff\'{e}rentielles
  stochastiques unidimensionnelles.
\newblock In {\em Seminar on probability, {XVII}}, volume 986 of {\em Lecture
  Notes in Math.}, pages 15--31. Springer, Berlin, 1983.

\bibitem{LeeTrutnau}
H.~Lee and G.~Trutnau.
\newblock Well-posedness for a class of degenerate {I}t\^o stochastic
  differential equations with fully discontinuous coefficients.
\newblock {\em Symmetry}, 12(4, 570):1--33, 2020.

\bibitem{LeonNualartTindel}
J.~A. Le\'{o}n, D.~Nualart, and S.~Tindel.
\newblock Young differential equations with power type nonlinearities.
\newblock {\em Stochastic Process. Appl.}, 127(9):3042--3067, 2017.

\bibitem{LeoniMorini}
G.~Leoni and M.~Morini.
\newblock Necessary and sufficient conditions for the chain rule in
  {$W^{1,1}_{\rm loc}(\mathbb{R}^N;\mathbb{R}^d)$} and {${\rm BV}_{\rm
  loc}(\mathbb{R}^N;\mathbb{R}^d)$}.
\newblock {\em J. Eur. Math. Soc. (JEMS)}, 9(2):219--252, 2007.

\bibitem{LinsSpitkovskyZhong}
B.~Lins, I.~M. Spitkovsky, and S.~Zhong.
\newblock The normalized numerical range and the {D}avis-{W}ielandt shell.
\newblock {\em Linear Algebra Appl.}, 546:187--209, 2018.

\bibitem{LyonsQian}
T.~Lyons and Z.~Qian.
\newblock {\em System control and rough paths}.
\newblock Oxford Mathematical Monographs. Oxford University Press, Oxford,
  2002.

\bibitem{Lyons}
T.~J. Lyons.
\newblock Differential equations driven by rough signals.
\newblock {\em Rev. Mat. Iberoamericana}, 14(2):215--310, 1998.

\bibitem{LLC}
T.~J. Lyons, M.~Caruana, and T.~L\'{e}vy.
\newblock {\em Differential equations driven by rough paths}, volume 1908 of
  {\em Lecture Notes in Mathematics}.
\newblock Springer, Berlin, 2007.
\newblock Lectures from the 34th Summer School on Probability Theory held in
  Saint-Flour, July 6--24, 2004, With an introduction concerning the Summer
  School by Jean Picard.

\bibitem{MaslowskiNualart}
B.~Maslowski and D.~Nualart.
\newblock Evolution equations driven by a fractional {B}rownian motion.
\newblock {\em J. Funct. Anal.}, 202(1):277--305, 2003.

\bibitem{Mattila}
P.~Mattila.
\newblock {\em Geometry of sets and measures in {E}uclidean spaces}, volume~44
  of {\em Cambridge Studies in Advanced Mathematics}.
\newblock Cambridge University Press, Cambridge, 1995.
\newblock Fractals and rectifiability.

\bibitem{MattilaMoranRey}
P.~Mattila, M.~Mor\'{a}n, and J.-M. Rey.
\newblock Dimension of a measure.
\newblock {\em Studia Math.}, 142(3):219--233, 2000.

\bibitem{Nakao}
S.~Nakao.
\newblock On pathwise uniqueness and comparison of solutions of one-dimensional
  stochastic differential equations.
\newblock {\em Osaka J. Math.}, 20(1):197--204, 1983.

\bibitem{NualartOuknine}
D.~Nualart and Y.~Ouknine.
\newblock Regularization of differential equations by fractional noise.
\newblock {\em Stochastic Process. Appl.}, 102(1):103--116, 2002.

\bibitem{NualartRascanu}
D.~Nualart and A.~R\u{a}\c{s}canu.
\newblock Differential equations driven by fractional {B}rownian motion.
\newblock {\em Collect. Math.}, 53(1):55--81, 2002.

\bibitem{Reshetnyak}
Y.~G. Reshetnyak.
\newblock {\em Space mappings with bounded distortion}, volume~73 of {\em
  Translations of Mathematical Monographs}.
\newblock American Mathematical Society, Providence, RI, 1989.
\newblock Translated from the Russian by H. H. McFaden.

\bibitem{RevuzYor}
D.~Revuz and M.~Yor.
\newblock {\em Continuous martingales and {B}rownian motion}, volume 293 of
  {\em Grundlehren der Mathematischen Wissenschaften [Fundamental Principles of
  Mathematical Sciences]}.
\newblock Springer-Verlag, Berlin, third edition, 1999.

\bibitem{SKM}
S.~G. Samko, A.~A. Kilbas, and O.~I. Marichev.
\newblock {\em Fractional integrals and derivatives}.
\newblock Gordon and Breach Science Publishers, Yverdon, 1993.

\bibitem{Schwartz}
L.~Schwartz.
\newblock {\em Th\'{e}orie des distributions}.
\newblock Publications de l'Institut de Math\'{e}matique de l'Universit\'{e} de
  Strasbourg, No. IX-X. Nouvelle \'{e}dition, enti\'{e}rement corrig\'{e}e,
  refondue et augment\'{e}e. Hermann, Paris, 1966.

\bibitem{Sottinen-Yazigi}
T.~Sottinen and A.~Yazigi.
\newblock Generalized {G}aussian bridges.
\newblock {\em Stochastic Process. Appl.}, 124(9):3084--3105, 2014.

\bibitem{Stein70}
E.~M. Stein.
\newblock {\em Singular integrals and differentiability properties of
  functions}.
\newblock Princeton Mathematical Series, No. 30. Princeton University Press,
  Princeton, N.J., 1970.

\bibitem{TaoWright}
T.~Tao and J.~Wright.
\newblock {$L^p$} improving bounds for averages along curves.
\newblock {\em J. Amer. Math. Soc.}, 16(3):605--638, 2003.

\bibitem{LauriSole}
S.~Torres and L.~Viitasaari.
\newblock Stochastic differential equations with discontinuous coefficients.
\newblock {\em Preprint}, pages 1--18, 2019.
\newblock \url{https://arXiv.org/abs/1908.03183}.

\bibitem{Triebel}
H.~Triebel.
\newblock {\em Theory of function spaces}, volume~78 of {\em Monographs in
  Mathematics}.
\newblock Birkh\"{a}user Verlag, Basel, 1983.

\bibitem{Veretennikov}
A.~J. Veretennikov.
\newblock Strong solutions and explicit formulas for solutions of stochastic
  integral equations.
\newblock {\em Mat. Sb. (N.S.)}, 111(153)(3):434--452, 480, 1980.
\newblock transl. in Math. USSR Sb. 39(3):387--403, 1981.

\bibitem{Xiao97}
Y.~Xiao.
\newblock H\"{o}lder conditions for the local times and the {H}ausdorff measure
  of the level sets of {G}aussian random fields.
\newblock {\em Probab. Theory Related Fields}, 109(1):129--157, 1997.

\bibitem{Xiao}
Y.~Xiao.
\newblock Random fractals and {M}arkov processes.
\newblock In {\em Fractal geometry and applications: a jubilee of {B}enoit
  {M}andelbrot, {P}art 2}, volume~72 of {\em Proc. Sympos. Pure Math.}, pages
  261--338. Amer. Math. Soc., Providence, RI, 2004.

\bibitem{Xiao06}
Y.~Xiao.
\newblock Properties of local-nondeterminism of {G}aussian and stable random
  fields and their applications.
\newblock {\em Ann. Fac. Sci. Toulouse Math. (6)}, 15(1):157--193, 2006.

\bibitem{Yamato}
Y.~Yamato.
\newblock Stochastic differential equations and nilpotent {L}ie algebras.
\newblock {\em Z. Wahrsch. Verw. Gebiete}, 47(2):213--229, 1979.

\bibitem{Yaskov}
P.~Yaskov.
\newblock On pathwise {R}iemann-{S}tieltjes integrals.
\newblock {\em Statist. Probab. Lett.}, 150:101--107, 2019.

\bibitem{Young}
L.~C. Young.
\newblock An inequality of the {H}\"{o}lder type, connected with {S}tieltjes
  integration.
\newblock {\em Acta Math.}, 67(1):251--282, 1936.

\bibitem{Zahle98}
M.~Z\"{a}hle.
\newblock Integration with respect to fractal functions and stochastic
  calculus. {I}.
\newblock {\em Probab. Theory Related Fields}, 111(3):333--374, 1998.

\bibitem{Zahle01}
M.~Z\"{a}hle.
\newblock Integration with respect to fractal functions and stochastic
  calculus. {II}.
\newblock {\em Math. Nachr.}, 225:145--183, 2001.

\bibitem{SchneiderZahle}
M.~Z\"{a}hle and E.~Schneider.
\newblock Forward integrals and {SDE} with fractal noise.
\newblock In {\em Horizons of fractal geometry and complex dimensions}, volume
  731 of {\em Contemp. Math.}, pages 279--302. Amer. Math. Soc., Providence,
  RI, 2019.

\bibitem{Zhang10}
X.~Zhang.
\newblock Stochastic flows of {SDE}s with irregular coefficients and stochastic
  transport equations.
\newblock {\em Bull. Sci. Math.}, 134(4):340--378, 2010.

\bibitem{Ziemer}
W.~P. Ziemer.
\newblock {\em Weakly differentiable functions}, volume 120 of {\em Graduate
  Texts in Mathematics}.
\newblock Springer-Verlag, New York, 1989.
\newblock Sobolev spaces and functions of bounded variation.

\end{thebibliography}

\end{document}